\documentclass[a4paper]{article}
\usepackage[utf8]{inputenc}
\usepackage{times}
\usepackage{enumitem}
\usepackage{amsmath}
\usepackage{stmaryrd}
\usepackage{amsthm}
\usepackage{amssymb}
\usepackage{pifont}
\usepackage{mathtools}
\usepackage{cellspace}
\usepackage{mathrsfs}
\usepackage{seqsplit} 
\usepackage{listings} 
\usepackage{textcomp} 
\usepackage{young}
\usepackage{stmaryrd} 
\usepackage{bm}
\usepackage{empheq}
\usepackage{xcolor}
\usepackage{intcalc}
\usepackage{geometry}
\usepackage{multirow}
\usepackage{hyperref}
\usepackage{cleveref}
\usepackage{thmtools}
\usepackage{pgfplots} 
\usepackage{float}
\usepackage{subcaption}
\usepackage[framemethod=tikz]{mdframed}
\usepackage{faktor}
\usepackage[overload]{textcase} 
\usepackage{tikz}
\usepackage{graphics}
\usepackage{tikz-cd}
\usepackage{adjustbox}
\usetikzlibrary{calc, intersections}
\pgfplotsset{compat=1.12}
\usepackage[style=alphabetic, hyperref, sorting=nty, bibencoding=latin1, giveninits=true,isbn=false, backend=bibtex, maxnames=10]{biblatex}
\usepackage{etoc}
\usepackage{scalerel,stackengine} 
\stackMath 
\allowdisplaybreaks
\usepackage{algorithm}
\usepackage{algpseudocode}
\usepackage{comment}
\usepackage{accents} 

\setlist[enumerate]{label=(\roman*)}

\mdfsetup{frametitlealignment=\center} 

\theoremstyle{plain}
\newtheorem{theorem}{Theorem}[section]
\newtheorem{corollary}[theorem]{Corollary}
\newtheorem{proposition}[theorem]{Proposition}
\newtheorem{lemma}[theorem]{Lemma}

\theoremstyle{definition}
\newtheorem{definition}[theorem]{Definition} 

\theoremstyle{remark}
\newtheorem{remark}[theorem]{Remark}

\hypersetup{
	colorlinks,
	citecolor=black,
	filecolor=black,
	linkcolor=black, 
	urlcolor=black
} 

\theoremstyle{definition}

\newcommand{\fungraph}[2]{\langle #1, #2 \rangle}

\newcommand{\Aut}{\textnormal{Aut}}

\newcommand{\C}{\mathbb{C}}
\newcommand{\Gcal}{\mathcal{G}}

\newcommand{\Char}{\textnormal{char}}
\newcommand{\kbar}{\overline{\K}}
\newcommand{\0}{{\mathcal{O}}}
\renewcommand{\P}{{\mathbb{P}}}
\newcommand{\F}{{\mathbb{F}}}
\newcommand{\HH}{\mathbb{H}}
\newcommand{\Gal}{\textnormal{Gal}}
\newcommand{\Hess}{\mathrm{H}}
\newcommand{\K}{{\Bbbk}}

\renewcommand{\S}{{\mathcal{S}}}
\newcommand{\NN}{{\mathbb{N}}}
\newcommand{\ZZ}{{\mathbb{Z}}}
\newcommand{\QQ}{{\mathbb{Q}}}
\newcommand{\End}{\mathrm{End}}
\newcommand{\Per}{\mathrm{Per}}
\newcommand{\ord}{\mathrm{ord}}

\newcommand{\cmark}{\ding{51}}%
\newcommand{\xmark}{\textcolor{lightgray}{\ding{55}}}%

\newcommand{\tr}{\mathrm{tr}}

\newcommand{\T}{\mathrm{T}}
\newcommand{\TT}{\mathcal{T}}
\newcommand{\z}{\zeta_3}
\newcommand{\Frob}{\textnormal{Frob}}
\newcommand\reallywidehat[1]{%
	\savestack{\tmpbox}{\stretchto{%
			\scaleto{%
				\scalerel*[\widthof{\ensuremath{#1}}]{\kern.1pt\mathchar"0362\kern.1pt}%
				{\rule{0ex}{\textheight}}
			}{\textheight}%
		}{2.4ex}}%
	\stackon[-6.9pt]{#1}{\tmpbox}%
}

\newcommand{\Twist}{\mathrm{Twist}}

\title{The Hessian of elliptic curves \textcolor{black}{as a Lattès map}}
 \author{
  Marzio Mula\thanks{University of the Bundeswehr Munich, \texttt{marziomula@gmail.com}} \and 
  Federico Pintore\thanks{University of Trento, \texttt{federico.pintore@unitn.it}} \and 
  Daniele Taufer\thanks{Max Planck Institute for Mathematics in the Sciences, \texttt{daniele.taufer@gmail.com}}
}
\date{July 2026}

\begin{document}

\maketitle

\begin{abstract}
We prove that the Hessian transformation of elliptic curves, both as an action on $j$-invariants and on the Hesse pencil, is a rigid Lattès map \textcolor{black}{fitting into a reduced diagram, hence it lifts} to a degree-$3$ endomorphism $\psi$ of a prescribed elliptic curve $E$.
This result provides an effective tool to investigate the dynamics of the Hessian transformation, \textcolor{black}{whose symmetries are inherited from those of $\psi$, which we characterize}.
In particular, over arbitrary fields of characteristic different from $2$ and $3$, \textcolor{black}{the functional graphs of the Hessian and, more generally, of Lattès maps fitting into analogous reduced diagrams, are}
completely determined by the action of $\psi$ on the twists of $E$.
When the underlying field is finite, we specialize these results to obtain a complete classification of Hessian functional graphs and derive an efficient method for computing iterated Hessians.
\end{abstract}


\section{Introduction}
Let $\K$ be a field of characteristic different from $2$ and $3$. 
Given a projective hypersurface $V(F) \subseteq \P^n(\K)$, its \emph{Hessian} 
is defined as the zero locus of the determinant of the Hessian matrix of $F$.
This classical construction enjoys several geometric properties, whose study began with the seminal works by Hesse~\cites{Hesse1,Hesse2} and has continued to the present day~\cite{cilibertoOttaviani:hessianMap,CilibertoOttaviani2026}.
When the considered form $F$ is a projective plane cubic (i.e., $n=2$ and $\deg(F) = 3$), the Hessian of $V(F)$ 
is, in turn, defined by a (possibly zero) cubic form.
Moreover, if $V(F)$ is also smooth (i.e., elliptic), the Hessian has proven to be a crucial tool for investigating its arithmetic properties \cite{Hesse3torsion,stanojkovskiVoll:hessianEC,TauSala:loops}.

It is natural to ask how a given elliptic curve relates to its Hessian and, more generally, to its iterated Hessian.
This question motivates the study of the (discrete) dynamical systems defined by the Hessian transformation, whose corresponding functional graphs will be referred to as \emph{Hessian graphs}.
The study of these graphs has already been undertaken by several authors for $\K=\C$ or $\K=\mathbb{R}$~\cite{salmon:higherPlaneCurves, hollcroft1926, popescuPampu2008:iterHess,artebaniDolgachev:hassePencil, cataneseSernesi2024:hesseModSpace,HessianJake}, and it falls within a broader field of research regarding the dynamics of rational functions \cite{NiederreiterShparlinski2003, silverman2007arithmetic,}.
However, despite their fundamental nature and theoretical relevance, the structure and symmetries of Hessian graphs have not been fully investigated yet.

The moduli space classically employed to study the Hessian transformation is $X(1) \simeq \textcolor{black}{\P^1(\kbar)}$, which parametrizes $\kbar$-isomorphism classes of elliptic curves 
by means of their $j$-invariants. 
However, as hinted in~\cite[Rmk.\,5.2]{popescuPampu2008:iterHess}, one may also consider a more structured model, which covers \textcolor{black}{$\P^1(\kbar)$}, by viewing the Hessian transformation as a \emph{Lattès map}~\cite{milnor2006lattes}\textcolor{black}{, namely a rational function $\phi: \P^1(\K) \to \P^1(\K)$, whose extension to $\P^1(\kbar)$ fits into a commutative diagram}
\begin{equation} \label{eq:LattDiag}
    \begin{tikzcd}
    E \arrow{d}{\pi}\arrow{r}{\psi} & E \arrow{d}{\pi} \\
    \P^1(\kbar)  \arrow{r}{\phi} & \P^1(\kbar)
\end{tikzcd}
\end{equation}
\textcolor{black}{where $E$ is an elliptic curve, $\psi$ a curve morphism and $\pi$ a finite projection.
The diagram of \cref{eq:LattDiag} is called \emph{reduced} if $\pi$ is given by the natural projection by a subgroup $\Gamma \subseteq \Aut(E)$. The map $\phi$ is called \emph{rigid} if $\deg(\pi) \neq 2$ or $\psi$ cannot be chosen as $[P] \mapsto [m]P+T$, for any $m \in \ZZ$ and $T \in E$.}

\textcolor{black}{In this work, we consider two projective Hessian transformations, namely
\begin{align*}
    \Hess : \P^1(\K) \to \P^1(\K), \quad &[u:v] \mapsto [(u-6912v)^3:-27u^2v], \\
    \Lambda : \P^1(\K) \to \P^1(\K), \quad &[u:v] \mapsto [ u^3 + 108 v^3 : -3u^2v ].
\end{align*} 
The map $\Hess$ encodes the action of the Hessian map on $j$-invariants, while $\Lambda$ describes the Hessian action on the Hesse pencil~\cite{artebaniDolgachev:hassePencil,Abdallah2023}, namely the curves defined by $\{ x^3 + y^3 + z^3 + \lambda xyz \}_{\lambda \in \K} \cup \{xyz\}$.
We characterize the fibers of the $j$-function on the Hesse pencil (\Cref{Sect2:Prop-IsoClasses}), and we exhibit an explicit bijection between the order-$2$ $\K$-rational points of an elliptic curve $E$ and $\Hess^{-1}(j(E))$ (\Cref{prop:j2tors}).
This shows that computing the order-$2$ points of $E(\K)$ is equivalent to computing its inverse Hessian. 
We then turn our attention to the Lattès nature of the Hessian, proving the following:}

\smallskip

\textcolor{black}{\textbf{Theorem} [\Cref{thm:projectphi3}, \Cref{cor:rigidlattes}]:
    For any $\K$, both $\Hess$ and $\Lambda$ are rigid Lattès maps fitting into reduced Lattès diagrams.
}
\smallskip

\textcolor{black}{Such Lattès diagrams are explicitly given by projecting a degree-$3$ endomorphism $\psi$ of the curve}
\[ E \ : \ y^2 = x^3 - 1728. \]

\textcolor{black}{This curve has $j$-invariant $0$, therefore we have an embedding $\ZZ[(-1-\sqrt{-3})/2] \hookrightarrow \End(E)$.
Under this identification, the lifted Hessian $\psi$ corresponds to the complex multiplication by $\sqrt{-3}$ (\Cref{subsec:Ek}), 
which implies that $E$ is a CM-curve in characteristic $0$.
This is a non-generic phenomenon for elliptic curves, which leads to additional arithmetic structure and number-theoretic applications \cite{kohel:endomRings,cox:primesOfFormX2+nY2,ECHasse}.
Such CM-curves are natural candidates to construct rigid Lattès maps (\Cref{rmk:CMforRigid}).}

\textcolor{black}{To understand the global symmetries displayed by the dynamical systems of $\Hess$ and $\Lambda$, we study the regular dynamics induced by group homomorphisms with prime kernel:}

\smallskip

\textcolor{black}{\textbf{Theorem} [\Cref{thm:structure}]:
    Let $\phi$ be a group endomorphism with prime kernel, and let $\tau_{\0}$ be the functional graph defined by all the elements that are eventually mapped to the group identity $\0$.
    The connected components in the functional graph of $\phi$ consist of cycles, lines or semilines, whose vertices are roots of arborescences that are isomorphic to sub-arborescences of $\tau_{\0}$.
}
\smallskip

\textcolor{black}{The above result provides a rigid classification of the connected components that may arise in the functional graph of $\psi$.
However, when the field $\K$ is not algebraically closed, the fiber $\S_{\pi} = \pi^{-1}\big( \P^1(\K) \big)$ need not be a group, hence the regular dynamics of $\psi$ on $E$ is not necessarily inherited on $\S_{\pi}$.
We resolve this issue for reduced Lattès diagrams by proving the following structural result:}

\smallskip

\textcolor{black}{\textbf{Theorem} [\Cref{thm:StructureS}]:
    Let $E$ be an elliptic curve in (short) Weierstrass form and $\pi:E \to \P^1$ be a suitable projection.
    Then $\S_{\pi}$ is covered by isomorphic copies of twists of $E$ over $\K$.
}
\smallskip

\textcolor{black}{The above covering is minimal, as different components intersect only in points with some zero coordinate (\Cref{prop:newdisjpartsSk}).
We prove that the regular action of $\psi$ on these components is preserved after $\pi$, with only a few possible identifications (\Cref{prop:identifyancestors,prop:GenProjectingSk}).
These general tools may be applied to study the dynamics of Lattès maps fitting into a reduced diagram.}

\textcolor{black}{We then specialize these results to the Hessian transformations, showing that the arithmetic properties of $E$ can be leveraged to understand the loops (\Cref{lem:selfloops}) and the structure of connected components (\Cref{prop:Ord3,prop:torsion}) of $\psi$.
This allows us to fully describe the identifications arising from $\pi$, and therefore provide a detailed characterization of the connected components in the dynamical systems generated by $\Hess$ and $\Lambda$ (\Cref{prop:projectingSk}).}

In the last part of the paper, we restrict our attention to finite fields $\K = \F_q$.
\textcolor{black}{In this case, we give a full characterization of elements $j \in \P^1(\F_q)$ with a prescribed size of $\Hess^{-1}(j)$:}

\smallskip

\textcolor{black}{\textbf{Theorem} [\Cref{lemma:trace2}, \Cref{thm:indeg1}]:
    Let $E_{(j)}$ be an elliptic curve with $j$-invariant $j \in \F_q^*$. Then $|\Hess^{-1}(j)| = 0$ if and only if the trace of $E_{(j)}$ is even.
    If also $j \neq 1728$, then $|\Hess^{-1}(j)| = 1$ if and only if $j - 1728$ is not a square in $\F_q^*$.
}
\smallskip

We exploit this characterization to provide a complete description of the structure of Hessian graphs, which 
decisively depends on $q \bmod 3$.
When $q \equiv 1 \bmod 3$, we prove that Hessian graphs 
consist of unions of six known functional graphs of group endomorphisms, up to pruning certain branches (\Cref{thm:structureq=1}).
When $q \equiv 2 \bmod 3$, the curve $E$ is supersingular, and we exploit its known group structure to show that Hessian graphs arise from an ordered merge of two isomorphic graphs, whose depth and cycles are fully determined (\Cref{thm:structureq=2}).

Finally, we employ our results to guarantee that any iteration of the Hessian transformation over $\F_q$ can be computed in polynomial time in the size of $q$ (\Cref{prop:iterHess}).
This paves the way for potential cryptographic applications of Hessian graphs.
In this direction, we prove that supersingular $j$-invariants all lie in prescribed connected components (\Cref{prop:locatess}), whose vertices have all the same known trace modulo $3$ (\Cref{prop:supersingTrace}).

\medskip
\noindent 
\textbf{Paper organization.}
In \Cref{Sec:Prel}, we recall the main properties of the objects considered in our work and provide novel preliminary results on the Hessian transformation.
\textcolor{black}{In \Cref{sec:HessianIsLattes}, we give an explicit reduced Lattès structure to the Hessian maps. 
In the following sections, we prove general results about the structure of functional graphs of group endomorphisms with prime kernel (\Cref{sec:ConnComp}), and the projection of such regular graphs in reduced Lattès diagrams (\Cref{sec:Sk}).
In \Cref{sec:Functionalpsik}, we apply these results to understand the Hessian dynamics.}
In \Cref{sec:HessianFF}, we further specialize the study of Hessian graphs to completely describe their structure over finite fields.
In \Cref{sec:crypto}, we discuss the potential use of our results for cryptographic purposes, while, in \Cref{sec:conclusion}, we draw our conclusions.
Finally, we include complementary materials, algorithms, and visual examples in \Cref{App:IsoClasses,App:twists,App:tightnessOfStructureThm,App:iterHess,App:examples}.


\section{Preliminaries}
\label{Sec:Prel}

\noindent \textbf{Notation.} Throughout the whole paper, $\K$ always denotes a field of characteristic $\Char(\K) \neq 2,3$, and $\F_q$ denotes the finite field of size $q=p^r$ for some prime integer $p$ and $r \in \NN_{>0}$.
The set of invertible elements of $\K$ will be denoted by $\K^* = \K \setminus \{0\}$, and its algebraic closure by $\kbar$.
For $n \in \NN$, we denote \textcolor{black}{the set of $n$-th powers in $\K$ (resp., in $\K^*$)} by $\K^n = \{k^n\}_{k \in \K}$ (resp., $(\K^*)^n = \{k^n\}_{k \in \K^*}$).
We also \textcolor{black}{fix $\sqrt{-3} \in \kbar$ and a primitive third root of unity $\z = \frac{-1-\sqrt{-3}}{2} \in \kbar$}.


\subsection{Elliptic curves}

In this work, we will mostly consider irreducible cubic forms $F\in \K[x,y,z]$, which defines a projective curve in $\P^2(\K)$ \textcolor{black}{with a smooth $\K$-rational point}.
Since $\mathsf{char}(\K) \neq 2,3$, with a suitable change of coordinates over $\K$ \cite[§III.1]{silv:arithEll}, we can always assume such an $F$ to be in short Weierstrass form:
\begin{equation}
\label{eq:WeirForm}
    F = x^3 + Axz^2 + Bz^3 - y^2z, \qquad A,B \in \K.
\end{equation}
In particular, when its \emph{discriminant} $\Delta(F) = -16(4 A^3 + 27 B^2)$ is not $0$, it defines a smooth plane cubic, known as \emph{elliptic curve} (over $\K$).
Under this assumption, we conventionally denote it by $E$. 
We will often restrict the projective plane to the affine chart $z \neq 0$, where the curve $E$ is defined by
\begin{equation}
\label{eqn:weierstrassEqn}
y^2=x^3+Ax + B, \qquad A,B \in \K.
\end{equation}
We will denote its unique non-affine point by $\0 = [0:1:0] \in \P^2(\K)$.

We refer to~\cites{silv:arithEll,washington:ellipticCurves} for classical definitions and results on elliptic curves, and recall only what is needed for the present paper.
The \emph{$\K$-rational points} of an elliptic curve $E$, denoted by $E(\K)$, are the solutions of \cref{eqn:weierstrassEqn} over $\K$, together with the point \emph{at infinity} $\0$.
%
Every elliptic curve $E$ can be endowed with the structure of an abelian group $(E(\kbar),+)$, whose zero element is $\0$ and with $E(\K)$ as a subgroup. 
An \emph{isogeny} between two elliptic curves $E_1,E_2$ over $\K$ is a curve morphism
\[\varphi \colon E_1 \rightarrow E_2\]
that is also a group homomorphism. We say that $\varphi$ is \emph{defined over $\K$}  
if the rational functions defining $\varphi$ can be chosen with coefficients in $\K$.
We say that a separable isogeny has \emph{degree} $n \in \NN_{>0}$, or it is an \emph{$n$-isogeny}, if its kernel has (finite) size $n$.
A separable isogeny of degree $1$ is an \emph{isomorphism}.
Every $\kbar$-isomorphism class of an elliptic curve \textcolor{black}{$E$ defined over $\K$} is uniquely determined by its \emph{$j$-invariant} $j(E) \in \K$.
It can be easily retrieved from the coefficients of any elliptic curve $E$ in the isomorphism class.
In particular, when $E$ is defined by \cref{eqn:weierstrassEqn}, its $j$-invariant is 
\[ j(E) = 1728\frac{4A^3}{4A^3 + 27B^2}. \]
An elliptic curve $E'$ that is isomorphic to $E$ over $\kbar$ is called a \emph{twist} of $E$.
\textcolor{black}{In this work, we only consider isomorphisms of \emph{elliptic curves}, namely, those respecting their group structure.}
If such an isomorphism is not defined over the base field $\K$ of $E$, then the twist $E'$ is called \emph{proper}.
The set of twists of $E$, up to $\K$-isomorphism, is denoted by \textcolor{black}{$\mathrm{Twist}(E,\0)$}, and its structure is well-known.

\begin{lemma}[{\cite[Prop.\,X.5.4]{silv:arithEll}}] \label{lemma:twists}
    Let $E$ be the elliptic curve defined over $\K$ by $y^2=x^3+Ax+B$, and
       \[
            n = \begin{cases}
            2 \quad \text{if $j(E) \not \in  \{0,1728\}$},\\
            4 \quad \text{if $j(E) =1728$},\\
            6 \quad \text{if $j(E) =0$}.\\
            \end{cases}
       \]
    Then $\textcolor{black}{\mathrm{Twist}(E,\0)} \cong \K^*/(\K^*)^n$. The elements $E^{(D)} \in \textcolor{black}{\mathrm{Twist}(E,\0)}$ can be listed for $D \in \K^*/(\K^*)^n$ as 
        \begin{enumerate}
            \item \makebox[3.7cm][l]{$y^2 = x^3 + D^2 Ax + D^3 B$,} \quad if $j(E) \not \in \{0,1728\}$,
            \item \makebox[3.7cm][l]{$y^2 = x^3 + D Ax$,} \quad if $j(E)=1728$,
            \item \makebox[3.7cm][l]{$y^2 = x^3 + D B$,} \quad if $j(E)=0$.
        \end{enumerate}
\textcolor{black}{ Moreover, writing $D^{1/n}\in \kbar$ for a chosen $n$-th root of $D$, we denote the corresponding $\kbar$-isomorphism by
    \[
    \iota_{D}:E^{(D)} \to E, \quad (x,y)\mapsto\big(D^{-2/n}x,\;D^{-3/n}y\big).
    \]}
\end{lemma}


An isogeny from an elliptic curve $E$ to itself is called an \textit{endomorphism}.
The ring of endomorphisms of $E$ is denoted by $\End(E)$, \textcolor{black}{while the group of its automorphisms (fixing $\0$) is denoted by $\Aut(E)$. 
This is a well-known cyclic group, which only depends on the $j$-invariant of $E$ \cite[§III.10]{silv:arithEll}.
The automorphisms of $E$ that are defined over $\K$ will be denoted by $\Aut_{\K}(E)$}.
For any $m \in \NN$, the \textit{multiplication-by-$m$ map} $[m] \in \End(E)$ is 
defined for every $P \in E$ by
\[ [m]P = \underbrace{P + P +\dots+ P}_{m~\text{times}}. \]
The above definition extends to $m \in \ZZ$ by setting $[-m]P = -([m]P)$.
For each $m \in \ZZ_{\geq 1}$, the \emph{$m$-torsion} of an elliptic curve $E$ is its subgroup $E[m] = \ker([m])\subseteq E(\kbar)$, whose points are called $m$-\emph{torsion points}.
	

\subsection{Hessian of a cubic}

Given a homogeneous cubic $F \in \K[x,y,z]$, we define its \emph{Hessian} $\Hess(F)$ as the determinant of its Hessian matrix $\HH_F$ \cite[§5]{cilibertoOttaviani:hessianMap}.
It is easy to see that $\Hess(F)$ is either $0$ or again a homogeneous cubic.
This induces an equivariant transformation of the associated projective curves: if $\phi \in \mathrm{GL}_3(\K)$ is a linear change of variables, represented by the matrix $\Phi$ in the standard basis $\{x,y,z\}$, a straightforward application of the chain rule shows that
\[ \Hess\big(\phi(F)\big) = \det\big( \Phi^{T} \phi(\HH_{F}) \Phi \big) = (\det\Phi)^2\phi\big(\Hess(F)\big). \]
Thus, the Hessian respects $\K$-isomorphism classes.
As anticipated, we are mostly interested in irreducible cubics $F$, which 
we can assume to be in short Weierstrass form (\cref{eq:WeirForm}).
A straightforward computation then shows that in this case
\begin{align*}
\Hess(F)&= -8( 3 x y^2 + 3A x^2 z + 9B x z^2 - A^2 z^3).
\end{align*}
The following result analyzes the Hessian of smooth cubics, i.e., elliptic curves.
In particular, it shows that the elliptic curves whose Hessian also defines an elliptic curve are exactly those with non-zero $j$-invariant. 

\begin{proposition} \label{prop:HessofE}
Let $E$ be an elliptic curve defined by $y^2z=x^3+Axz^2+Bz^3$.
\begin{enumerate}
    \item\label{eni} If $A=0$, then $\Hess(E) = -24x(y^2+3Bz^2)$ defines three independent lines. Moreover, we have $\Hess\big(\Hess(E)\big)=(-24^3B)\Hess(E)$.
    \item\label{enii} If $A\neq0$, the short Weierstrass form of $\Hess(E)$ is
    \begin{equation}\label{eqn:hessianSWForm}
    x^3 - \frac{A^3 + 9B^2}{3A^4} x z^2 - \frac{A^3B + 6B^3}{3A^6} z^3 - zy^2,
    \end{equation} 
and we have
\begin{align}
\nonumber
    \Delta\big(\Hess(E)\big) &= \frac{(64/27)A^3 + 16B^2}{A^6} = -\frac{\Delta(E)}{27 A^6},\\
\label{eqn:jInvariant}
    j\big(\Hess(E)\big)&=\frac{1728(A^3+9B^2)^3}{ A^6  \bigl(A^3 + (27/4)B^2\bigr) }=j(E)  \frac{(A^3+9B^2)^3 }{ A^9}=\frac{ \big(6912-j(E) \big)^3}{27 j(E)^2}.
\end{align}
\end{enumerate}
\end{proposition}
\begin{proof}
\ref{eni}: When $A=0$, we straightforwardly check that
\[ \Hess(E) = -24x(y^2+3Bz^2) = -24x(y+\sqrt{-3B}z)(y-\sqrt{-3B}z). \]
Since $E$ is smooth we have $B \neq 0$, hence the linear factors of $\Hess(E)$ define independent lines.
Another direct computation shows that $\Hess\big(\Hess(E)\big)=(-24^3B)\Hess(E)$.

\ref{enii}: When $A \neq 0$, the Hessian of $E$ is 
\begin{equation*}
    \Hess(E) = -216A^2\left(x\left(\frac{y}{3A}\right)^2 + \frac{1}{3A} x^2 \left(\frac{z}{3}\right) + \frac{3B}{A^2}x\left(\frac{z}{3}\right)^2- \left(\frac{z}{3}\right)^3\right)
\end{equation*}
and the change of coordinates $[x:y:z] \mapsto [z : 3A y : 3x+3B/A^2z]$ yields the short Weierstrass form of \cref{eqn:hessianSWForm}. 
The discriminant and $j$-invariant can be directly computed from the coefficients of \cref{eqn:hessianSWForm}. 
\end{proof}

\begin{remark} \label{rem:degenerate}
In general, 
the Hessian fixes three independent lines: if
\[ F = (a_{11}x+a_{12}y+a_{13}z)(a_{21}x+a_{22}y+a_{23}z)(a_{31}x+a_{32}y+a_{33}z), \]
then $\Hess(F) = (2\det M^2)F$, where $M=(a_{ij})_{1\leq i,j \leq 3}$.
\end{remark}

\begin{remark} 
    \label{rem:flexPoints}
    The points in $E \cap H(E)$ are called \emph{flex points} or \emph{inflection points} of~$E$.
    It is known that the inflection points of an elliptic curve coincide with its $3$-torsion points, provided that the base point is chosen to be an inflection point \cite{dickson:pointsInflexion}.
    Moreover, as observed in~\cite[§3]{Hesse3torsion}, if the base point of $\Hess(E)$ is also chosen among the inflection points of $E$ (the natural choice being the same base point of $E$) and $\Hess(E)$ is again elliptic, then $E \cap \Hess(E)$ also coincides with the $3$-torsion points of $\Hess(E)$.
    Indeed, the addition laws on $E$ and $\Hess(E)$ 
    coincide on $E \cap \Hess(E) = E[3]$ because the chord-tangent construction on $E$ agrees with the same procedure on $\Hess(E)$: a line through two points (with multiplicity) of $E[3]$ meets $E$ in another point of $3$-torsion, hence it also belongs to $\Hess(E)$.
\end{remark}

\Cref{prop:HessofE} shows that the Hessian defines a transformation between $\kbar$-isomorphism classes, which we denote again by $\Hess$:
\[ \Hess(j) = \frac{ (6912-j)^3}{27 j^2}. \]
This map can be extended to a projective map of $\P^1(\K)$ by assigning to the union of three independent lines the value $\infty = [1:0] \in \P^1(\K)$, i.e.,
\begin{equation}
\label{eq:Hessj}
\Hess : \P^1(\K) \to \P^1(\K), \quad [j:1] \mapsto [(j-6912)^3 : -27j^2], \quad \quad [1:0] \mapsto [1:0].
\end{equation}
Throughout the paper, we refer to this map as the \emph{Hessian transformation} or \emph{Hessian map}. 


\subsection{The Hesse pencil} \label{subsec:hessPenc}

Rather than considering cubics in short Weierstrass form, some authors use the \emph{Hessian form}
\begin{equation} \label{eq:lambda}
    F_\lambda = x^3 + y^3 + z^3 + \lambda xyz, \qquad \lambda \in \K.
\end{equation}
The curve defined by $F_\lambda$ 
is nonsingular if and only if $\lambda^3 \neq -27$ \cite[§2.3.3]{hisil2010:ECs}.
The set of forms $\{ F_{\lambda} \}_{\lambda \in \K} \cup \{ xyz \}$
is called the \emph{Hesse pencil (over $\K$)}.
When $\lambda = 0$ we have $\Hess(F_0) = xyz$, i.e., the union of three independent lines.
For every other $\lambda \neq 0$, a simple computation~\cite[§V.218]{salmon:higherPlaneCurves} shows that $\Hess(F_\lambda) = F_{\lambda'}$ is again in Hessian form, where
\[ \lambda' = -\frac{108+\lambda^3}{3\lambda^2}. \]
By setting $F_{\infty} = xyz$, the Hessian on the Hesse pencil determines again a projective map 
\begin{equation}
\label{eqn:lambda}
\Lambda : \P^1(\K) \to \P^1(\K), \quad [\lambda:1] \mapsto [\lambda^3+108 : -3\lambda^2], \quad \quad [1:0] \mapsto [1:0].
\end{equation}

The elliptic curves of the Hesse pencil have all the same $9$ inflection points, which can be computed explicitly \cites[§7]{dickson:pointsInflexion}[§2]{artebaniDolgachev:hassePencil}:
\begin{align*}
    p_0 &= [0: 1: -1], & p_1 &= [0: 1: -\z],  & p_2 &= [0: 1: -\z^2],\\
    p_3 &= [1: 0: -1], & p_4 &= [1: 0: -\z^2],& p_5 &= [1: 0: -\z],\\
    p_6 &= [1:-1: 0],  & p_7 &= [1:-\z: 0],   & p_8 &= [1:-\z^2: 0].
\end{align*}
In fact, an elliptic curve belongs to the Hesse pencil if and only if its inflection points are $p_0,\dots, p_8$. 

Given an elliptic curve defined by $F_\lambda$, its $\overline{\K}$-isomorphism class is determined by its $j$-invariant, which can be directly computed \cite[§2.3.3]{hisil2010:ECs} as
    \begin{equation}
    \label{eqn:jInvPencil}
        j(\lambda) = \frac{\lambda^3(216-\lambda^3)^3}{(\lambda^3+27)^3}.
    \end{equation}

\begin{remark}
    The Hesse pencil may also contain singular cubics, namely
    \[ F_{-3\z^i} = \z^2 (\z x + \z^i y + \z^{2i} z) (\z^i x + \z^{2i}y + \z z) (\z^{2i} x + \z y + \z^i z), \quad i \in \{0,1,2\}, \]
    which all correspond to $j = \infty$ according to \cref{eqn:jInvPencil}, as prescribed by \Cref{rem:degenerate}.
    Notice that $F_{-3}$ belongs to the Hesse pencil over any field $\K$, while the others precisely when $\z \in \K$.
    Moreover, the relation $\Hess(F_{\lambda}) = F_{-3\z^i}$ is equivalent to
    \[ 0 = \lambda^3 - 9\z^i \lambda^2 + 108 = (\lambda + 3 \z^i)(\lambda - 6 \z^i)^2. \]
    Thus, the singular cubics are fixed points for the Hessian, and they may be obtained as the Hessian of the curves defined by $\{ F_{6\z^i} \}_{i \in \{0,1,2\}}$, which have all $j$-invariant $0$ by \cref{eqn:jInvPencil}.
\end{remark}

The connection between $\lambda$ and $j$ is further detailed by the following proposition.

\begin{proposition}
\label{Sect2:Prop-IsoClasses}
Let $j(\lambda)$ be as in \cref{eqn:jInvPencil} and define the map
\[ \mathcal{J} \colon \P^1(\K) \to \P^1(\K), \quad \lambda \mapsto j(\lambda), \quad [1:0] \mapsto [1:0]. \]
For every $j \in \P^1(\K)$ such that $\mathcal{J}^{-1}(j) \neq \emptyset$, we have:
\begin{enumerate}
     \item If $\z \in \K$, then $|\mathcal{J}^{-1}(j)| = 12$ except for $|\mathcal{J}^{-1}(0)| = 4 = |\mathcal{J}^{-1}(\infty)|$ and $|\mathcal{J}^{-1}(1728)| = 6$. 
        \item\label{propjlambdaii} If $\z \not\in \K$, then $|\mathcal{J}^{-1}(j)| = 2$.
    \end{enumerate} 
\end{proposition}
A proof for the finite-field case can be found in~\cite[Lem.\,10]{farashahi2011:numbOfLegJacHessEd}, while in \Cref{App:IsoClasses} we provide a constructive proof that holds over arbitrary fields $\K$. 


\subsection{Discrete dynamical systems and Hessian graphs}

A \emph{(discrete) dynamical system} consists of a set $S$ and a map $\phi \colon S \rightarrow S$ \cite{silverman2007arithmetic}.
\begin{definition}
    The \emph{functional graph of $\phi:S \to S$} is the directed graph whose vertices are the elements of $S$ and whose edges are $\{x \to \phi(x)\}_{x \in S}$.
    We denote such a graph by $\fungraph{S}{\phi}$.
\end{definition}

\textcolor{black}{Two functional graphs  $\fungraph{S}{\phi}$ and  $\fungraph{S}{\phi'}$ are \emph{isomorphic} if $\phi$ and $\phi'$ are \emph{conjugate}, namely there exists an automorphism $f$ of $S$ such that $\phi=f^{-1} \circ \phi' \circ f$.}
\textcolor{black}{In particular, when $S = \P^1(\K)$, a conjugation is given by a M\"obius transformation $f \in \Aut_{\K}(\P^1)$.}
We call \textit{Hessian graph} the functional graph $\fungraph{\P^1(\K)}{\Hess}$ of the Hessian transformation $\Hess$ defined in \cref{eq:Hessj}.
We denote by $\phi^{(n)}$ the $n$-th iterate of $\phi$, i.e., the composition of $\phi$ with itself $n$ times.
By convention, $\phi^{(0)}$ is the identity map on $S$.
A point $P \in S$ is called \emph{periodic} (with respect to $\phi$) if it lies on a \emph{cycle} of $\fungraph{S}{\phi}$, i.e., if there exists $n \in \NN_{>0}$ such that $\phi^{(n)}(P) = P$.
We denote the set of periodic points of $\phi$ by $\Per(\phi)$.
If there exists $m \in \NN$ such that $\phi^{(m)}(P) \in \Per(\phi)$, we say that $P$ is \emph{preperiodic}.

\begin{definition} \label{defn:depth}
    Let $\phi : S \to S$ and $T \subseteq S$. 
    For every $P \in S$, we define its \emph{depth} (with respect to $T$) as
    \[ d_T(P) = \inf_{n \in \NN} \big\{ \phi^{(n)}(P) \in T \big\}. \]
    When the set $T$ is understood, we will simply denote the depth of $P$ by $d(P)$.
\end{definition}

One can readily see that if a connected component contains periodic elements, then they lie in a single cycle, whose vertices are roots of (possibly empty) directed trees.
In the following sections, we will see that these rooted directed trees, also called \emph{arborescences}, are remarkably regular for Hessian graphs.


\subsection{Lattès maps}

Understanding the structure of functional graphs defined by rational functions over a field $\K$ is, in general, a challenging task.
However, \textcolor{black}{special} families of functions are better understood and enjoy a remarkably regular structure, \textcolor{black}{
namely, the \emph{finite quotients of afﬁne maps}: Lattès, Chebyshev and power maps \cites{milnor2006lattes}.
Those are the maps arising as finite projections of (affine) group morphisms, from which they inherit a symmetric dynamics~\cites[§6.8]{silverman2007arithmetic}{byszewski2019dynamAffMaps}.}
In this paper, we focus on Lattès maps, which we define as in \cites[§6.4]{silverman2007arithmetic}. \textcolor{black}{More results on these maps can be found in~\cites{milnor2006lattes}[§6.4-7]{silverman2007arithmetic}{Pakovich2020}.}

\begin{definition} \label{defn:Lattes}
    Let 
    $\phi\colon \P^1(\textcolor{black}{\K}) \rightarrow \P^1(\textcolor{black}{\K})$ be a rational map of degree $d \geq 2$, \textcolor{black}{and let $\overline{\phi}: \P^1(\kbar) \to \P^1(\kbar)$ denote its base change to $\kbar$}.
    We say that $\phi$ is a \emph{Lattès map} if there is an elliptic curve $E$ defined over \textcolor{black}{$\K$}, a curve morphism $\psi \colon E \rightarrow E$, and a finite separable covering $\pi \colon E \rightarrow \P^1 $ such that the following diagram is commutative.
   \begin{equation}
   \label{eqn:lattesDiag}
        \begin{tikzcd}
E(\kbar) \arrow{r}{\psi} \arrow{d}{\pi} & E(\kbar) \arrow{d}{\pi} \\
\P^1(\kbar) \arrow{r}{\textcolor{black}{\overline{\phi}}} & \P^1(\kbar) 
\end{tikzcd}
   \end{equation}
   \textcolor{black}{Moreover:
   \begin{itemize} \item If there is a non-trivial subgroup $\Gamma \subseteq \Aut(E)$ such that the covering is given by the canonical projection as
   \[ \pi: E \to E/\Gamma \xrightarrow{\simeq} \P^1, \]
   then \cref{eqn:lattesDiag} is called a \emph{reduced Lattès diagram} for $\phi$.
   \item If \cref{eqn:lattesDiag} may be chosen such that $\psi$ is the map $P \mapsto [m]P + T$ for some $m \in \ZZ$ and $T \in E$, and $\deg(\pi) = 2$, then $\phi$ is called \emph{flexible}. Otherwise, it is called \emph{rigid}.
\end{itemize}}
\end{definition}

\textcolor{black}{\begin{remark} \label{rmk:CMforRigid}
    When $E$ is given by a (short) Weierstrass model, any map $\phi$ fitting in a reduced Lattès diagram is conjugated (over $\kbar$) to a Lattès map whose projection is given by
    \[ \pi(x,y) = \begin{cases} 
    x & \textnormal{if } |\Gamma| = 2, \\
    x^2 & \textnormal{if } |\Gamma| = 4, \\
    y & \textnormal{if } |\Gamma| = 3, \\
    x^3 & \textnormal{if } |\Gamma| = 6. 
    \end{cases}\]
    as described in \cite[Prop.\,6.37]{silverman2007arithmetic}.
    In particular, for every $m \in \ZZ$, the projection of $[m] \in \End(E)$ on its first coordinate always yields a reduced diagram for a flexible Lattès map~\cite[§6.5]{silverman2007arithmetic}.
    Similarly, since the degree of a flexible Lattès map is an integer square \cite[Prop. 6.51-(a)]{silverman2007arithmetic}, one can obtain reduced diagrams of rigid Lattès maps by projecting on the first coordinate any $\psi \in \End(E)$ with non-square degree, which exists only if $\ZZ \subsetneq \End(E)$. 
    The latter always occurs over any finite fields, while in $\Char(\K) = 0$ it holds precisely when $E$ is a CM-curve \cite[Rmk. III.4.3]{silv:arithEll}.
\end{remark}}


We now briefly consider the case $\K=\C$.
\textcolor{black}{In this case, every Lattès map fits into a reduced Lattès diagram \cite[Thm. 6.57]{silverman2007arithmetic}. Moreover,} there is a practical criterion to check whether a rational map $\phi : \P^1(\C) \rightarrow \P^1(\C)$ is Lattès, \textcolor{black}{which we briefly recall here}.
The \emph{ramification index} of $\phi$ at $P \in \P^1(\C)$ is defined as
\[ e_P(\phi) = \ord_P\big(\phi(x)-\phi(P)\big). \]
This quantity can be explicitly computed even when $P = \infty$ or $\phi(P) = \infty$, by changing the considered affine chart.
A point $P \in \P^1(\C)$ is called a \emph{critical point} for $\phi$ if $e_P(\phi)\geq 2$. The points of the form $\{ \phi^{(n)}(P) \}_{n \in \NN_{>0}}$, where $P$ is a critical point, are called \emph{post-critical}.

It is known that a finite subset $T \subseteq \P^1(\C)$ such that $\phi^{-1}(T) = T$ may contain at most two elements \cite[Thm.\,1.6]{silverman2007arithmetic}.
Since $\phi$ is surjective, $\phi^{-1}(T) = T$ implies $T = \phi(T)$, hence $\phi$ acts as a permutation on $T$.
The maximal such $T$ is called the set of \emph{exceptional points} of $\phi$.
\begin{theorem}
\label{thm:lattesCrit}
    A rational map $\phi \colon \P^1(\C) \rightarrow  \P^1(\C)$ is a Lattès map if and only if it has no exceptional points and there exists a \emph{ramification function} $v \colon  \P^1(\C) \rightarrow \NN_{>0}$ such that $v\big(\phi(P)\big)=e_P(\phi) v(P)$ for all~$P\in \P^1(\C)$.
\end{theorem}
\begin{proof}
    It follows from \cite[Thm.\,4.1]{milnor2006lattes} and the classification of finite quotients of affine maps given in \cite[§2]{milnor2006lattes}.
\end{proof}


\subsection{Inverse image of \texorpdfstring{$\Hess$}{H}} \label{subsec:InvImHess}

In this subsection, we investigate the fibers of $\Hess$, i.e., for every $j' \in \P^1(\K)$ we consider the set
\[ \Hess^{-1}(j') = \{j \in \P^1(\K) \ | \ \Hess(j) = j' \}. \] 
We already observed in \Cref{prop:HessofE} and \Cref{rem:degenerate} that $\Hess^{-1}(\infty) = \{0,\infty\}$.
The fibers of $0$ and $1728$ are independent of the field $\K$, as it was also noticed in \Cite[Prop.\,4.1]{popescuPampu2008:iterHess}.
Moreover, we prove that the fibers of $j' \not\in \{ 0,1728,\infty \}$ effectively correspond to the $\K$-rational points of order $2$ of any elliptic curve $E_{(j')}$ with that $j$-invariant.

    

\begin{lemma}
    \label{lem:multRoots}
    For every ${j'} \in \K$, we define
    \[H_{j'}(j) = j^2(\Hess(j)-{j'}) \in \K[j].\]
    \begin{enumerate}
        \item \label{itm:HjSimpleRoots} If $j'\notin \{0,1728\}$, then $H_{j'}$ has only simple roots (if any).
        \item \label{itm:Hj=0} If $j'=0$, then
        \[ H_{0}(j)=-\frac{1}{27}(j-6912)^3. \] 
        \item \label{itm:Hj=1728} If $j'=1728$, then
        \[ H_{1728}(j)=-\frac{1}{27}(j-1728)(j+13824)^2. \]
    \end{enumerate}
\end{lemma}    
  
    \begin{proof}
        Cases \ref{itm:Hj=0} and \ref{itm:Hj=1728} follow from explicit calculations.
        To check for multiple roots, we consider the formal derivative
        \[ H_{j'}'(j) = -\frac{1}{9} (j^{2} + 18 j{j'} - 13824 j + 47775744). \]
        The above polynomial vanishes if and only if
        \[ {j'}=\frac{-j^{2} + 13824 j - 47775744}{18 j} \ \implies \ H_{j'}(j) = \frac{1}{54} (j-6912)^2 (j+13824). \]
        Thus, we may have $H_{j'}(j) = 0 = H_{j'}'(j)$ (i.e., multiple roots of $H_{j'}$) only if $j = 6912$ or $j = -13824$, whose corresponding $j'$ are $0$ (case \ref{itm:Hj=0}) and $1728$ (case \ref{itm:Hj=1728}), respectively.
    \end{proof}

\begin{proposition} \label{prop:j2tors}
    Let $j \in \K \setminus \{0,1728\}$ and $E_{(j)} : y^2 = x^3 + Ax + B$ be any elliptic curve with that $j$-invariant, with $A,B \in \K$.
    Then the following is a well-defined \textcolor{black}{bijection}:
    \[ 
    \{(\alpha,0) \in E_{(j)}(\K)\} \to \Hess^{-1}(j), \qquad (\alpha,0) \mapsto 6912 \frac{A+3\alpha^2}{4A+3\alpha^2}. \]
\end{proposition}
\begin{proof}
    We prove that the above map is well-defined: given an order-$2$ point $(\alpha,0) \in E_{(j)}(\K)$, we have 
    \[ B = -\alpha^3-A\alpha \quad \implies \quad j = 6912\frac{A^3}{(A + 3 \alpha^2)^2 (4 A + 3\alpha^2)}, \] 
    and a straightforward computation shows that
    \[ \Hess\left( 6912 \frac{A+3\alpha^2}{4A+3\alpha^2} \right) = j.\]
    
    We now prove that this map is injective: $(\alpha,0), (\beta,0) \in E_{(j)}(\K)$ have the same image if and only if
    \[ 6912 \frac{A+3\alpha^2}{4A+3\alpha^2} = 6912 \frac{A+3\beta^2}{4A+3\beta^2} \quad \iff \quad 3A(\alpha^2 - \beta^2) = 0. \]
    Since $j \neq 0$, then $A \neq 0$ so the above is equivalent to $\alpha = \pm \beta$.
    However, $(\alpha,0)$ and $(-\alpha,0)$ cannot be both points of $E_{(j)}$, since $j \neq 1728$ implies $B \neq 0$.

    If $\K$ were algebraically closed, both domain and codomain would have precisely three elements by \Cref{lem:multRoots}, hence the considered map would be $1$-to-$1$.
    Therefore, it is sufficient to show that $\alpha^2 \in \K$ implies $\alpha \in \K$.
    This follows by
    \[ 0 = \alpha^3 + A\alpha + B \implies \alpha (\alpha^2+A) = -B, \]
    noticing that $B \neq 0$ implies $\alpha^2+A \neq 0$, therefore $\alpha = -\frac{B}{\alpha^2+A}$.
\end{proof}

\textcolor{black}{There are always curves with $j$-invariant $0$ and $1728$ in the image of $\Hess$ by \Cref{lem:multRoots}, but their order-$2$ $\K$-rational points depend on the chosen twist. We refer the reader to \Cref{App:twists} for a complete analysis of these cases.
For every other $j$-invariant,} \Cref{prop:j2tors,prop:hessianIsom} show that the image of the Hessian map misses precisely the elliptic curves without $\K$-rational order-$2$ points.
Hence, these curves correspond to the \emph{leaves} of Hessian graphs, namely the elements with empty Hessian fiber.
Since the $2$-torsion is always defined over algebraically closed fields, every complex elliptic curve belongs to the Hessian image, as noticed in \cite[Prop.\,5.15]{cilibertoOttaviani:hessianMap}.
\textcolor{black}{Conversely, over finite fields, being a leaf only depends on the parity of the trace (see \Cref{lemma:trace2})}.

We conclude the section by listing the loops of the Hessian graphs $\fungraph{\P^1(\K)}{\Hess}$, namely the elements $j \in \P^1(\K)$ such that $\Hess(j) = j$.

\begin{lemma} \label{lem:loopj}
    Let $j \in \P^1(\K)$. Then $\Hess(j) = j$ if and only if
    \[ j \in \begin{cases}
        \{1728, \infty, \frac{ 2^73^3 }{7} (\pm 3 \sqrt{-3} - 1)\}, &\textnormal{if } {7 \nmid \mathsf{char}(\K)}\textnormal{ and } 3 \in (\K^*)^2, \\
        \{1728, \infty \}, &\textnormal{if } {7 \nmid \mathsf{char}(\K)}\textnormal{ and } 3 \not\in (\K^*)^2, \\
        \{1728, \infty, 4 \}, &\textnormal{if } 7 \ | \ \mathsf{char}(\K). 
    \end{cases}\]
\end{lemma}
\begin{proof}
    We observed in \Cref{rem:degenerate} that $\Hess(\infty)=\infty$, then we can restrict to $j \in \K$.
    For such points, $\Hess(j)=j$ if and only if
    \[ 0 = -28j^3 + 20736j^2 - 143327232j + 330225942528 = 4(j-1728)(-7j^2 - 6912j - 47775744). \]
    If $7 \nmid \mathsf{char}(\K)$, the roots of the above polynomial over $\kbar$ are $\{1728, \frac{ 2^73^3 }{7} (\pm 3 \sqrt{-3} - 1)\}$.
    Conversely, if $7 \mid \mathsf{char}(\K)$, then
    \[ -7x^2 - 6912x - 47775744 = 4x+5 \in \K[x], \]
    therefore the possible values of $j$ are $1728$ and $4 = -\frac{5}{4} \in \K$.
\end{proof}


\section{The Hessian as a Lattès map}
\label{sec:HessianIsLattes}

In this section, we show that the Hessian transformation $\Hess$ is a Lattès map.
\textcolor{black}{We first show that, over $\C$, this result can be proved with the general criterion of \Cref{thm:lattesCrit}.
Notably, the complex Lattès structure of the Hessian was already claimed in~\cite[Rmk.\,5.2]{popescuPampu2008:iterHess}.
However, these arguments provide little insight into the behavior of $\Hess$ on a field $\K$ that is not algebraically closed and/or has positive characteristic.
To fill this gap, in the following sections, we provide an explicit \emph{reduced} Lattès model for $\Hess$ over an arbitrary field, which we will leverage to understand its dynamics.}

\subsection{\textcolor{black}{Complex Lattès structure}}


\begin{proposition}
\label{Sec3:prop-Lattes-C}
    The map $\Hess \colon \P^1(\C) \rightarrow \P^1(\C)$, defined as in \cref{eq:Hessj}, is a Lattès map.
\end{proposition}
\begin{proof}
    By \Cref{lem:multRoots} and \cite[Thm.\,1.6]{silverman2007arithmetic}, one can check that $\Hess$ cannot have exceptional points. 
    Moreover, the first and second derivatives of $\Hess$ are
    \[ H'(j) = - \frac{(j + 13824) (j - 6912)^{2}}{27j^3}, \quad H''(j) = -2^{17}3^4 \frac{j-6912}{j^4}. \]
    Hence, $-13824$ and $6912$  are critical points with $e_{-13824}(\Hess)=2$ and $e_{6912}(\Hess)=3$.
    To check the values $0$ and $\infty$, we consider the \emph{linear conjugate}~\cite[§1.1]{silverman2007arithmetic} of $\Hess$ after $f \colon j \mapsto 6912j /  (j+6912)$:
    \[ \tilde{\Hess}(j) = ( f^{-1}\circ \Hess \circ f)(j) = \frac{2^{32}3^9}{j^{3} + 6912 j^{2} - 2^{24}3^6}, \]
    whose first and second derivatives are
    \[ \tilde{\Hess}'(j) = -\frac{2^{32}  3^{10} j  (j + 2^9 3^2)}{(j^{3} + 6912 j^{2} - 2^{24}3^{6})^{2}}, \quad \tilde{\Hess}''(j) = \frac{2^{34}  3^{10}  (j^4 + 2^{10}3^2j^3 + 2^{15} 3^6j^2 +2^{23}3^6j +2^{31}  3^8)}{(j^3 + 6912j^2 -2^{24} 3^6)^3}.  \]
    Thus, we see that $-6912 = f^{-1}(\infty)$ is not critical for $\tilde{\Hess}$, while $0 = f^{-1}(0)$ is critical with ramification $e_{0}(\tilde{\Hess})=2$.
    Therefore, we conclude that
    \[ e_j(\Hess) = \begin{cases}
        3 & \textnormal{if } j = 6912, \\
        2 & \textnormal{if } j \in \{0, -13824 \}, \\
        1 & \textnormal{otherwise}.
    \end{cases} \]
The thesis follows immediately by \Cref{thm:lattesCrit}, checking that the map 
\[v:\P^1(\C) \to \NN_{>0}, \quad j \mapsto v(j)=\begin{cases}
    2 &\textnormal{if }j=1728,\\
    3 &\textnormal{if }j=0,\\
    6 &\textnormal{if }j=\infty,\\
    1 &\textnormal{otherwise},
\end{cases}\]
is a ramification function.
\end{proof}

The above proof shows how \Cref{thm:lattesCrit} allows us to prove that a complex rational map is Lattès by only inspecting its \emph{post-critical portrait}, namely its behavior on the orbits of its critical points.
The post-critical portrait of $\Hess$ is shown in \Cref{fig:critPort}.
\begin{figure}[ht]
    \centering
   \begin{tikzpicture}
    \node (A) at (0, 0.5) {$-13824$};
    \node (B) at (0, -0.5) {1728};
    \node (C) at (2, 0) {0};
    \node (D) at (2, -1) {$\infty$};
    \node (E) at (2, 1) {$6912$};

    \draw[-latex, bend left=35] (A) to (B);
    \draw[-latex, bend right=35] (A) to (B);
    
    \draw[-latex, bend left=45] (E) to (C);
    \draw[-latex, bend right=45] (E) to (C);
    \draw[-latex] (E) -- (C);

    \draw[-latex, bend left=35] (C) to (D);
    \draw[-latex, bend right=35]  (C) to (D);

    \draw[-latex, looseness=4, out=-135, in=-45] (B) to (B);
    \draw[-latex, looseness=4, out=-135, in=-45] (D) to (D);
\end{tikzpicture}
\caption{Post-critical portrait of $\Hess$. The arrow multiplicity corresponds to the origin's ramification index.}
\label{fig:critPort}
\end{figure}

\begin{remark}
    \Cref{fig:critPort} indicates that $\Hess$ falls within the classification of degree-$3$ complex Lattès maps given in \cite[§8.2]{milnor2006lattes}, as a M{\"o}bius transformation of \cite[eq.\ (17)]{milnor2006lattes}.
\end{remark}

\subsection{The projective Hessian} \label{subsec:Fkl}

\begin{definition}
    For any given $k \in \K^*$, we define the projective maps
    \begin{align*}
        \Hess_k &: \P^1(\K) \to \P^1(\K), \quad [u:v]\mapsto[(u+kv)^3 : -27 u^2v], \\
        \Lambda_k &: \P^1(\K) \to \P^1(\K), \quad [u:v] \mapsto [ u^3 + k v^3 : -3u^2v ].
    \end{align*} 
\end{definition}
Since $k$ is invertible, then both $\Hess_k$ and $\Lambda_k$ are well-defined at every point of $\P^1(\K)$.

\begin{remark} \label{rmk:psiHess}
    The Hessian transformation $\Hess$ (at the level of $j$-invariants) defined in \cref{eq:Hessj} is equal to $\Hess_{-6912}$, while the Hessian  transformation $\Lambda$ (on the Hesse pencil) defined in \cref{eqn:lambda} corresponds to $\Lambda_{108}$.
\end{remark}


\begin{definition}
    For any given $h \in \K^*$, we define the \emph{scaling map} as
    \[ \Phi_{h} : \P^1(\K) \to \P^1(\K), \quad [u:v] \mapsto [hu : v]. \]
\end{definition}

Since $h$ is invertible, then $\Phi_{h}$ is a well-defined projective coordinate change of $\P^1(\K)$. 

\begin{proposition}
\label{prop:FkFamily}
   For every $h,k \in \K^*$, we have
   \[ \Hess_{hk} \circ \Phi_{h} = \Phi_{h} \circ \Hess_k \quad \textnormal{and} \quad \Lambda_{h^3k} \circ \Phi_{h} = \Phi_{h} \circ \Lambda_k. \]
\end{proposition}
\begin{proof} For every $[u:v] \in \P^1(\K)$, we have
\[  \Hess_{hk}\big(\Phi_{h}([u:v])\big) = [ (hu+khv)^3 : -27 h^2u^2v ] = [h(u+kv)^3 : -27 u^2v] = \Phi_{h}\big(\Hess_k([u:v]) \big), \]
and 
\[  \Lambda_{h^3k}\big(\Phi_{h}([u:v])\big) = [ h^3u^3+h^3kv^3 : -3 h^2u^2v ] = [h(u^3+kv^3) : -3 u^2v] = \Phi_{h}\big(\Lambda_k([u:v]) \big), \]
from which the proposition follows.
\end{proof}

\begin{corollary} \label{cor:FkFamily}
   The functional graphs $\big\{ \fungraph{\P^1(\K)}{\Hess_k} \big\}_{k \in \K^*}$ are all isomorphic.
\end{corollary}
\begin{proof}
    By applying \Cref{prop:FkFamily} with $h = k^{-1}$, we have
    \[ \Hess_k = \Phi_{k} \circ \Hess_1 \circ \Phi_{k}^{-1}, \]
    namely the functional graph 
    of $\Hess_k$ is isomorphic to that of $\Hess_1$, independently of $k \in \K^*$.
\end{proof}
\begin{remark}
    $\big\{ \fungraph{\P^1(\K)}{\Lambda_k} \big\}_{k \in \K^*}$ are not isomorphic in general, unless $\K = \K^3$.
\end{remark}


\begin{definition} \label{defn:cubing}
    We define the \emph{cubing map} as
    \[ M_3 : \P^1(\K) \to \P^1(\K), \quad [u:v] \mapsto [u^3:v^3]. \]
\end{definition}

\begin{proposition} \label{prop:commuting}
    Let $k \in \K^*$. Then
    \[ \Hess_k \circ M_3 = M_3 \circ \Lambda_k. \]
\end{proposition}
\begin{proof} By definition we have
\[ \Hess_k\big(M_3( [u:v] )\big) = \Hess_k([u^3:v^3]) = [(u^3+kv^3)^3 : -27 u^6v^3], \]
while
\[ M_3\big( \Lambda_k( [u:v] ) \big) = M_3( [u^3+kv^3 : -3u^2v] ) = [(u^3+kv^3)^3 : (-3u^2v)^3], \]
from which equality follows. 
\end{proof}

\Cref{prop:FkFamily,prop:commuting} show that the diagram of \Cref{fig:CommCubeBase} is commutative. Each vertical arrow of such diagram represents $M_3$.

\begin{figure}[H]
\centering\begin{tikzpicture}
\node[] (P1A) at (0,-2) {$\P^1$};
\node[] (P1B) at (4,-2) {$\P^1$};

\node[] (P1uA) at (1.5,-3) {$\P^1$};
\node[] (P1uB) at (5.5,-3) {$\P^1$};

\node[] (P1uuA) at (1.5,-5) {$\P^1$};
\node[] (P1uuB) at (5.5,-5) {$\P^1$};
\node[] (P1uuC) at (0,-4) {$\P^1$};
\node[] (P1uuD) at (4,-4) {$\P^1$};


\draw[-latex] (P1A)--(P1B) node[midway,above] () {$\Lambda_k$};
\draw[-latex] (P1uA)--(P1uB) node[midway,above] () {$\Lambda_{h^3k}$};

\draw[-latex] (P1A)--(P1uA) node[midway,xshift=-0.1cm,yshift=-0.2cm] () {$\Phi_{h}$};
\draw[-latex] (P1B)--(P1uB) node[midway,xshift=0.15cm,yshift=0.2cm] () {$\Phi_{h}$};

\draw[-latex] (P1uA)--(P1uuA) node[midway,above] () {};
\draw[-latex] (P1uB)--(P1uuB) node[midway,right] () {}; 
\draw[-latex] (P1uuA)--(P1uuB) node[midway,above] () {$\Hess_{h^3k}$};

\draw[-latex] (P1A)--(P1uuC) node[midway,left] () {$M_3$};
\draw[-latex] (P1uuC)--(P1uuA) node[midway,xshift=-0.1cm,yshift=-0.2cm] () {$\Phi_{h^3}$};

\draw[-latex] (P1uuC)--(P1uuD) node[midway,above] () {$\Hess_k$};
\draw[-latex] (P1uuD)--(P1uuB) node[midway,xshift=0.15cm,yshift=0.2cm] () {$\Phi_{h^3}$};
\draw[-latex] (P1B)--(P1uuD) node[midway,above] () {};
\end{tikzpicture}
\caption{Commuting diagram arising from \Cref{prop:FkFamily,prop:commuting}.}
\label{fig:CommCubeBase}
\end{figure}
\begin{remark} \label{rmk:lambdainj}
By specializing the above results with $k = 108$ and $h=-4$, one sees that
\[ M_3 \circ \Phi_{-4} : \P^1(\K) \to \P^1(\K), \quad [u:v] \mapsto [-(4u)^3:v^3] \]
maps $\fungraph{\P^1(\K)}{\Lambda}$ into $\fungraph{\P^1(\K)}{\Hess}$.
When the cubing operation is invertible over $\K$, this is a functional graph isomorphism. 
\end{remark}

\begin{remark} \label{rmk:constantmatters}
In this work, we will need to keep track of the isomorphisms $\Phi_h$, hence we maintain the dependency on~$k$.
However, \Cref{cor:FkFamily} shows that, from a dynamical point of view, the maps $\{\Hess_k\}_{k \in \K^*}$ are indistinguishable.
On the other hand, the constant $-27$ in the definition of $\Hess_k$ plays a crucial role from a dynamical perspective. In fact, the functional graphs of the maps $\{ \Phi_h \circ \Hess_k \}_{h \in \K^*}$, which are obtained by simply scaling one entry of $\Hess_k$, do not display the symmetric structure of $\Hess_k$.
For instance, in Appendix~\ref{App:examples} we compare the Hessian graph $\fungraph{\P^1(\F_{17})}{\Hess_{-6912}}$ 
(\Cref{fig:hessian17}) with the functional graph
$\fungraph{\P^1(\F_{17})}{\Phi_{\frac{27}{8}} \circ \Hess_{-6912}}$ 
(\Cref{fig:randomFuncGraph}).
\end{remark}

\subsection{The curve \texorpdfstring{$E_k$}{Ek} and the endomorphism \texorpdfstring{$\psi_k$}{ψk}} \label{subsec:Ek}

\begin{definition} \label{defn:KK}
    Let $k \in \K^*$. We denote by $E_k$ the elliptic curve defined over $\K$ by
    \begin{equation*}
        y^2 = x^3 + \frac{k}{4}.
    \end{equation*}
    We also denote by $T_k$ the $3$-torsion point $ \big(0,\frac{\sqrt{k}}{2}\big) \in E_k$, for any fixed choice of $\sqrt{k} \in \kbar$.
\end{definition}

\textcolor{black}{Since $E_k$ has $j$-invariant $0$, it is well-known that $\Aut(E_k) \simeq \mu_6$ ~\cite[§III.10]{silv:arithEll}, where $\mu_n$ denotes the group of $n$-th roots of unity in $\K^*$.
Letting $\rho : (x,y) \mapsto (\z x, -y)$ be a generator of $\Aut(E_k)$, we have a ring embedding
\begin{equation} \label{eq:embedding}
    \ZZ[\zeta_6] \to \End(E_k), \quad \zeta_6 \mapsto \rho,
\end{equation}
where $\zeta_6 = -\zeta_3$.
Hence, $\End(E_k)$ is strictly larger than $\ZZ$.
If $\Char(\K) = 0$, this means that $E_k$ has CM by $\ZZ[\zeta_6]$.
In positive characteristic, $\End(E_k)$ may be larger when $E_k$ is \emph{supersingular}.
We now consider the endomorphism corresponding to $\sqrt{-3}$ under the embedding of \cref{eq:embedding}.
}


\begin{definition} \label{defn:psik}
    \textcolor{black}{Let $k \in \K^*$. We define $\psi_k = [2]\rho- [1] \in \End(E_k)$.}
\end{definition}


\begin{lemma} \label{lem:psi2}
    For every $k \in \K^*$, we have $\psi_k \circ \psi_k = [-3] \in \End(E_k)$.
\end{lemma}
\textcolor{black}{\begin{proof}
    Since $\rho^2 = \rho-[1]$, we have $\psi_k \circ \psi_k = ([2] \rho-[1])^2 = [4](\rho-[1]) - [4]\rho + [1] = [-3]$.
\end{proof}}

\begin{lemma} \label{lem:psik}
    \textcolor{black}{For every $k \in \K^*$, the map $\psi_k$ is a degree-$3$ endomorphism of $E_k$, explicitly given by
     \[ \psi_k (x,y) = \begin{cases}
        \left( -\frac{x^3+k}{3x^2}, -y\frac{x^3-2k}{3\sqrt{-3}x^3} \right) \quad &\textnormal{if } (x,y) \neq \pm T_k, \\
        \0 &\textnormal{otherwise.}
    \end{cases} \]}
\end{lemma}
\begin{proof} The above map is the composition of the degree-$3$ isogeny $E_{k} \to E_{-27k}$ of kernel $\{ \0, \pm T_k \}$ described in~\cite[§2.1]{cohenPazuki:3descent} with the curve isomorphism $E_{-27k} \simeq E_{k}$, which is defined (over $\overline{\K}$) by $(x,y) \mapsto \big(\frac{x}{-3}, \frac{y}{-3\sqrt{-3}}\big)$. \textcolor{black}{The equality with $[2]\rho- [1]$ is checked symbolically by the MAGMA \cite{magma} script available at \cite{OurRepo}.}
\end{proof}

\subsection{Lattès structure} \label{subsec:Hess}

In this section, we show that both $\Hess_k$ and $\Lambda_k$ can be \textcolor{black}{lifted to} $\psi_k$ over $E_k$.

\begin{definition} \label{defn:pi}
    \textcolor{black}{For a given elliptic curve $E(\kbar) \subset \P^2(\kbar)$,} we define the projections
    \[ \pi_{\Lambda} : \textcolor{black}{E}(\kbar) \to \P^1(\kbar), \quad [x:y:z] \mapsto \begin{cases} [x:z] &\textnormal{if } [x:y:z] \neq \0,\\
    [1:0] &\textnormal{otherwise}.
\end{cases}  \]
and
    \[ \pi_{\Hess} = M_3 \circ \pi_{\Lambda} :  \textcolor{black}{E}(\kbar) \to \P^1(\kbar), \quad [x:y:z] \mapsto \begin{cases} [x^3:z^3] &\textnormal{if } [x:y:z] \neq \0,\\
    [1:0] &\textnormal{otherwise}.
\end{cases}  \]
\end{definition}

\begin{remark} \label{rmk:ProjAsQuot}
    \textcolor{black}{We notice that $\pi_{\Lambda}$ and $\pi_{\Hess}$ do not depend on the coefficients of $E$, as they may be equally defined over the whole $\P^2(\kbar)$, and further restricted to the given $E$.}
    When \textcolor{black}{this curve is $E_k$ (namely, $j(E) = 0$)}, these projections correspond to
\[   
    \pi_{\Lambda} : E_k \to E_k / \langle \textcolor{black}{[-1]} \rangle \simeq \P^1, \quad
    \pi_{\Hess} : E_k \to E_k / \langle \textcolor{black}{\rho} \rangle \simeq \P^1,
\]
where \textcolor{black}{$\rho (x,y) = (\z x, -y)$} is the generator of $\Aut(E_k)$. 
\end{remark}


\begin{theorem} \label{thm:projectphi3}
    Let $k \in \K^*$. 
    \textcolor{black}{Then, the following}
    \[\begin{tikzcd}
E_k(\overline{\K}) \arrow{r}{\psi_k} \arrow{d}{\pi_{\Lambda}} & E_k(\overline{\K}) \arrow{d}{\pi_{\Lambda}} \\
\P^1(\overline{\K}) \arrow{r}{\Lambda_k} & \P^1(\overline{\K}) 
\end{tikzcd} \qquad
\begin{tikzcd}
E_k(\overline{\K}) \arrow{r}{\psi_k} \arrow{d}{\pi_{\Hess}} & E_k(\overline{\K}) \arrow{d}{\pi_{\Hess}} \\
\P^1(\overline{\K}) \arrow{r}{\Hess_k} & \P^1(\overline{\K}) 
\end{tikzcd}\]
    \textcolor{black}{are reduced Lattès diagrams for $\Lambda_k$ and $\Hess_k$.}
\end{theorem}

\begin{proof} 
\textcolor{black}{Since the projections $\pi_{\Lambda}$ and $\pi_{\Hess}$ arise from the projections by (finite) subgroups of $\Aut(E_k)$ (\Cref{rmk:ProjAsQuot}), it is sufficient to show that the above diagrams are commutative. 
}
If $P \in \{\0, \pm T_k\}$, then we have $\Lambda_k \circ \pi_{\Lambda}(P) = [1:0] = \pi_{\Lambda} \circ \psi_k (P)$.
Otherwise, $P = (x,y) \in E_k(\kbar)$ with $x \neq 0$, and we have
\[ \Lambda_k \circ \pi_{\Lambda}(P) = \Lambda_k([x:1]) = [x^3+k : -3x^2 ] = \left[ -\frac{x^3+k}{3x^2} : 1 \right] = \pi_{\Lambda} \circ \psi_k(P). \]
This proves that the first diagram is commutative, which implies the second by post-composing $\pi_{\Lambda}$ with $M_3$ and using \Cref{prop:commuting}.
\end{proof}

\textcolor{black}{\begin{corollary} \label{cor:rigidlattes}
    For any $k \in \K^*$, the maps $\Lambda_k$ and $\Hess_k$ are rigid Lattès maps.
\end{corollary}
\begin{proof}
    By \Cref{thm:projectphi3}, both $\Lambda_k$ and $\Hess_k$ are degree-$3$ Lattès maps.
    Since the degree of flexible Lattès maps is $m^2$ by \cite[Prop. 6.51-(a)]{silverman2007arithmetic}, for some $m \in \ZZ$, then they are both rigid.
\end{proof}}

We also recall that the curves $\{E_k\}_{k \in \K^*}$ are isomorphic over $\overline{\K}$, since they all have $j$-invariant $0$.
Such isomorphisms are given, for suitable $u \in \overline{\K}^*$, by 
\[ \phi_u : E_k \to E_{u^6k}, \quad (x,y) \to (u^2x,u^3y). \]
We now prove that these isomorphisms respect the considered $3$-endomorphisms $\psi_k$.

\begin{lemma} \label{lem:comDiag}
    For every $k \in \K^*$ and $u \in (\kbar)^*$, we have
    \[ \psi_{u^6k} \circ \phi_u = \phi_u \circ \psi_k. \]
\end{lemma}
\begin{proof}
    For every $P \in \{\0, \pm T_k\}$, we have
    $ \psi_{u^6k} \circ \phi_u(P) = \0 = \phi_u \circ \psi_k(P)$.
    For every other $(x,y) \in E_k$, we have
    \begin{align*}
        \psi_{u^6k} \circ \phi_u (x,y) &= \left(-\frac{u^6x^3+u^6k}{3u^4x^2}, -u^3y\frac{u^6x^3-2u^6k}{3\sqrt{-3}u^6x^3}\right) \\
        &= \left(-u^2\frac{x^3+k}{3x^2}, -u^3y\frac{x^3-2k}{3\sqrt{-3}x^3}\right) = \phi_u \circ \psi_k(x,y).
    \end{align*} 
    Hence, $\psi_{u^6k} \circ \phi_u = \phi_u \circ \psi_k$ for every point of $E_k$.
\end{proof}

\begin{corollary} \label{cor:psiAut}
    For every $k \in \K^*$, $\psi_k$ commutes with every element of $\Aut(E_k)$.
\end{corollary}
\begin{proof}
    \textcolor{black}{
    This follows straightforwardly by \Cref{lem:comDiag}, since 
    $\Aut(E_k) = \langle \rho \rangle$, with $\rho = \phi_{-(\z)^2}$.}
\end{proof}

\begin{remark}
\label{rem:compatibility}
    \textcolor{black}{In general, for any given elliptic curve $E$ over $\K$, an endomorphism $\psi \in \End(E)$ that is \emph{compatible} with $\Gamma \subseteq \Aut(E)$ gives rise to a reduced Lattès diagram by fixing an isomorphism $E/\Gamma \simeq\P^1$.
    By compatible, we mean that for every $\gamma \in \Gamma$ there exists $\gamma' \in \Gamma$ such that 
    \begin{equation} \label{eq:LattèsCond}
        \psi \circ \gamma = \gamma' \circ \psi \in \End(E).
    \end{equation}
    For a finer statement, we refer the reader to \Cref{lem:GammaActionPsi}.
    The above is clearly satisfied when $\End(E)$ is commutative, which is always the case if $\Char(\K)=0$ \cite[Cor. III.9.4]{silv:arithEll} or when $\Char(\K) > 0$ and $E$ is ordinary.
    Furthermore, \cref{eq:LattèsCond} is also satisfied when $\Gamma = \langle [-1] \rangle$, which holds whenever $j(E) \not \in \{0,1728\}$.
    On the other hand, suppose that $j(E) \in \{0,1728\}$, that $E$ is supersingular, and that $\langle[-1]\rangle \subsetneq \Gamma \subseteq \Aut(E)$. Then some endomorphisms of $E$ may nevertheless be compatible with $\Gamma$, as it is the case for $\psi_k$ (\Cref{cor:psiAut}), but this compatibility should not be expected.
    For example, let $\Frob_5 : (x,y) \mapsto (x^5,y^5)$ be the Frobenius endomorphism of the curve $E_4: y^2 = x^3+1$ defined over $\F_{25}$, and $\rho \in \Aut(E)$ as in \cref{eq:embedding}.
    We straightforwardly check \cite{OurRepo} that \cref{eq:LattèsCond} is not satisfied for $\psi = \Frob_5 + [1]$, for any choice of $\gamma' \in \Aut(E)$, hence such a $\psi$ does not fit any reduced Lattès diagram.}
\end{remark}

\begin{figure}[ht]
\centering\begin{tikzpicture}
\node[] (EkA) at (0,0) {$E_k$};
\node[] (EkB) at (4,0) {$E_k$};
\node[] (P1A) at (0,-2) {$\P^1$};
\node[] (P1B) at (4,-2) {$\P^1$};

\node[] (EkuA) at (1.5,-1) {$E_{u^6k}$};
\node[] (EkuB) at (5.5,-1) {$E_{u^6k}$};
\node[] (P1uA) at (1.5,-3) {$\P^1$};
\node[] (P1uB) at (5.5,-3) {$\P^1$};

\node[] (P1uuA) at (1.5,-5) {$\P^1$};
\node[] (P1uuB) at (5.5,-5) {$\P^1$};
\node[] (P1uuC) at (0,-4) {$\P^1$};
\node[] (P1uuD) at (4,-4) {$\P^1$};

\draw[-latex] (EkA)--(EkB) node[midway,above] () {$\psi_k$};
\draw[-latex] (EkuA)--(EkuB) node[midway,above] () {$\psi_{u^6k}$};

\draw[-latex] (EkA)--(EkuA) node[midway,xshift=-0.1cm,yshift=-0.2cm] () {$\phi_{u}$};
\draw[-latex] (EkB)--(EkuB) node[midway,right,yshift=0.1cm] () {}; 

\draw[-latex] (EkA)--(P1A) node[midway,left] () {$\pi_{\Lambda}$};
\draw[-latex] (EkB)--(P1B) node[midway,left] () {};

\draw[-latex] (EkuA)--(P1uA) node[midway,right] () {};
\draw[-latex] (EkuB)--(P1uB) node[midway,right] () {}; 

\draw[-latex] (P1A)--(P1B) node[midway,above] () {$\Lambda_k$};
\draw[-latex] (P1uA)--(P1uB) node[midway,above] () {$\Lambda_{u^6k}$};

\draw[-latex] (P1A)--(P1uA) node[midway,xshift=-0.1cm,yshift=-0.2cm] () {$\Phi_{u^2}$};
\draw[-latex] (P1B)--(P1uB) node[midway,right] () {};

\draw[-latex] (P1uA)--(P1uuA) node[midway,above] () {};
\draw[-latex] (P1uB)--(P1uuB) node[midway,right] () {}; 
\draw[-latex] (P1uuA)--(P1uuB) node[midway,above] () {$\Hess_{u^6k}$};

\draw[-latex] (P1A)--(P1uuC) node[midway,left] () {$M_3$};
\draw[-latex] (P1uuC)--(P1uuA) node[midway,xshift=-0.1cm,yshift=-0.2cm] () {$\Phi_{u^6}$};

\draw[-latex] (P1uuC)--(P1uuD) node[midway,above] () {$\Hess_k$};
\draw[-latex] (P1uuD)--(P1uuB) node[midway,above] () {};
\draw[-latex] (P1B)--(P1uuD) node[midway,above] () {};
\end{tikzpicture}
\caption{Commuting results from \Cref{subsec:Fkl,subsec:Ek,subsec:Hess}.} 
\label{fig:CommCube}
\end{figure}


\section{Functional graphs of prime-kernel group endomorphisms} \label{sec:ConnComp}

\textcolor{black}{
In this section, we prove a structural result on the dynamics of group endomorphisms with a prime kernel (\Cref{thm:structure}).
This result immediately applies to the endomorphism $\psi_k$, whose kernel has $3$ elements by \Cref{lem:psik}.
More generally, it may be applied to study Lattès maps arising from elements of $\End(E)$ corresponding to $\sqrt{\pm \ell}$, for any integer prime $\ell$. However, it would require further refinement to be applied to non-prime-kernel endomorphisms, such as those defining flexible Lattès maps.
An interested reader can find examples of these different dynamics in \cite{DEKLERK2022112691,URIBEVARGAS}. The results in this section are necessary but not sufficient to fully describe the dynamics of a Lattès map over a given non-algebraically closed field $\K$. In fact, it would still be necessary to determine the shape of the fibers of its domain and to understand the action of the projection $\pi$. We will address this in Section~\ref{sec:Sk}.}

\textcolor{black}{Let $G$ be} an additive group, whose identity is denoted by $\0$.
Let also $\phi \in \End(G)$ be an endomorphism of $G$ with finite prime kernel $\ell = \left|\ker \phi\right|$.
It is easy to see that the relation
\begin{equation*} 
    P \sim Q \quad \textnormal{if} \quad \exists \, n,m \in \NN : \phi^{(n)}(P) = \phi^{(m)}(Q)
\end{equation*}
is a well-defined equivalence relation on the points of $G$.
\begin{definition}
    The equivalence classes in $G/{\sim}$ will be referred to as the \emph{connected components} of~$\fungraph{G}{\phi}$.
\end{definition}

We will need the following lemma from group theory, whose proof is reported for completeness.

\begin{lemma} \label{lem:groupchain}
    Let $G$ be a group, $\ell$ be a prime integer, and $\phi \in \End(G)$ be such that $\left|\ker \phi \right| = \ell$. 
    Then the ascending series of groups
    \begin{equation} \label{eq:groupseq}
        \{\0\} = \ker \phi^{(0)} \lneq \ker \phi \leq \ker \phi^{(2)} \leq \dots 
    \end{equation}
    is strictly increasing until it stabilizes.
    Let $m = \inf_{k \in \NN} \big\{\ker \phi^{(k)} = \ker \phi^{(k+1)}\big\}$.
    Then, for every $0 \leq k < m$, we have
    \[ \ker \phi^{(k)} = \phi\big(\ker \phi^{(k+1)}\big). \]
\end{lemma}
\begin{proof}
    We first observe that the series of \cref{eq:groupseq} stabilizes at $\ker \phi^{(m)}$. 
    In fact, for every $k \geq m$, we have
    \[ x \in \ker \phi^{(k+1)} \implies \phi^{(k-m)}(x) \in \ker \phi^{(m+1)} = \ker \phi^{(m)} \implies x \in \ker \phi^{(k)}. \]
    Since $\ker \phi^{(0)} = \{\0\} \neq \ker \phi$, then $m > 0$.
    For every $0 \leq k < m$, we have
    \[ \phi\big( \ker \phi^{(k+1)} \big) \leq \ker \phi^{(k)} \lneq \ker \phi^{(k+1)}. \]
    However, since $\phi$ restricts to an endomorphism of the finite group $\ker \phi^{(k+1)}$, whose kernel is given by $\ker \phi \cap \ker \phi^{(k+1)} = \ker \phi$, then the fundamental theorem on homomorphisms implies
    \[ \bigl[\ker \phi^{(k+1)} : \phi\big( \ker \phi^{(k+1)} \big)\bigr] = \frac{\left| \ker \phi^{(k+1)} \right| }{ \left| \phi\bigl( \ker \phi^{(k+1)} \bigr) \right|} = \left|\ker \phi \right| = \ell. \]
    Therefore, since $\ell$ is prime, we conclude that $\ker \phi^{(k)} = \phi\bigl(\ker \phi^{(k+1)}\bigr)$.
\end{proof}

\begin{definition} \label{def:AncestorTree}
    Let $\phi \in \End(G)$ and $S \subseteq G$ be a non-empty subset.
We define the \emph{ancestor tree} of a point $P \in S$ (with respect to $S$ and $\phi$) as the graph $\tau_{P,S,\phi}$, whose vertices are
\[ \{ Q \in G \ | \ d_S(Q) < \infty \textnormal{ and } \phi^{(d_S(Q))}(Q) = P \}, \]
and whose edges are
\[ \{ Q_1 \to Q_2 \ | \ Q_1 \neq P \textnormal{ and } \phi(Q_1) = Q_2 \}. \]
When $\phi$ is understood, we denote the graph by $\tau_{P,S}$.
\textcolor{black}{Moreover, for brevity we will use the convention
\begin{equation*}
    \tau_P=\begin{cases}
       \tau_{P,\mathcal{\Per(\phi)}} \qquad &\text{if $P\in \Per(\phi)$},\\
       \tau_{P,\{P\}} \qquad &\text{otherwise.} 
    \end{cases}
\end{equation*}
}
The elements in $\tau_{P}$ will be referred to as the \emph{ancestors} of $P$.
\end{definition}

\begin{remark}
    The ancestor tree $\tau_{P,S,\phi}$ is the restriction of the functional graph $\langle G, \phi \rangle$ to the elements of $G$ that eventually reach $P$ without passing through any other element of $S$. 
    In particular, each vertex $Q \neq P$ has a unique path to $P$.
    Hence, $\tau_{P,S,\phi}$ is an \emph{arborescence} rooted in $P$, namely a directed rooted graph such that every vertex different from $P$ has precisely one path to $P$.
\end{remark}


\begin{definition} \label{defn:Tlm}
Let $\ell \in \NN$ and $m \in \NN \cup \{\infty\}$.
We define $\T_\ell^{m}$ as the arborescence constructed as follows: 
\begin{itemize}
    \item Every node that is not a leaf has indegree $\ell$ if it is not the root, $\ell-1$ otherwise.
    \item If $m < \infty$, then $\T_\ell^{m}$ is finite and every leaf has depth $m$ (with respect to the root).
    Conversely, $\T_\ell^{\infty}$ has no leaves.
\end{itemize}
\end{definition}

\begin{remark}
    The above notation is borrowed from \cite{URIBEVARGAS}, although we stress that the groups we consider might well be infinite, and the root vertex has no loops.
\end{remark}
\begin{remark}
    The arborescence $\T_\ell^{m}$ of \Cref{defn:Tlm} is uniquely determined by the choice of $\ell,m$.
    The relevant family for this paper will be $\{ \T_3^m \}_{m \in \NN \cup \{\infty\}}$, which is portrayed in \Cref{fig:T3m}.
\end{remark}

\vspace{-0.3cm}
\begin{figure}[ht]
\centering\begin{tikzpicture}
\node[circle,scale=0.7] (0) at (0,0) {\textbullet};
\node[circle,scale=0.8] (0T) at ($(0)+(0,-0.4)$)  {$\T_3^0$};

\node[circle,scale=0.7] (1) at  ($(0)+(2,0)$) {\textbullet};
\node[circle,scale=0.8] (1T) at ($(1)+(0,-0.4)$)  {$\T_3^1$};
\coordinate (1l) at ($(1)+(0.2,0.8)$);
\coordinate (1r) at ($(1)+(-0.2,0.8)$);

\node[circle,scale=0.7] (2) at  ($(1)+(2.5,0)$) {\textbullet};
\node[circle,scale=0.8] (2T) at ($(2)+(0,-0.4)$)  {$\T_3^2$};
\node[circle,scale=0.7] (2l) at ($(2)+(-0.6, 0.8)$) {};
\node[circle,scale=0.7] (2r) at ($(2)+(0.6, 0.8)$) {};
\coordinate (2ll) at ($(2l)+(-0.4,0.8)$);
\coordinate (2lc) at ($(2l)+(0,0.8)$);
\coordinate (2lr) at ($(2l)+(0.4,0.8)$);
\coordinate (2rl) at ($(2r)+(-0.4,0.8)$);
\coordinate (2rc) at ($(2r)+(0,0.8)$);
\coordinate (2rr) at ($(2r)+(0.4,0.8)$);

\node[circle] (dots) at ($(2)+(2,0)$) {\dots};

\node[circle,scale=0.7] (i) at ($(dots)+(2,0)$) {\textbullet};
\node[circle,scale=0.8] (iT) at ($(i)+(0,-0.4)$)  {$\T_3^{\infty}$};
\node[circle,scale=0.7] (il) at ($(i)+(-0.6, 0.8)$) {};
\node[circle,scale=0.7] (ir) at ($(i)+(0.6, 0.8)$) {};
\coordinate (ill) at ($(il)+(-0.4,0.8)$);
\coordinate (ilc) at ($(il)+(0,0.8)$);
\coordinate (ilr) at ($(il)+(0.4,0.8)$);
\coordinate (irl) at ($(ir)+(-0.4,0.8)$);
\coordinate (irc) at ($(ir)+(0,0.8)$);
\coordinate (irr) at ($(ir)+(0.4,0.8)$);
\node[circle] (v1) at ($(ill)+(0,0.4)$)  {$\vdots$};
\node[circle] (v2) at ($(ilc)+(0,0.4)$)  {$\vdots$};
\node[circle] (v3) at ($(ilr)+(0,0.4)$)  {$\vdots$};
\node[circle] (v4) at ($(irl)+(0,0.4)$)  {$\vdots$};
\node[circle] (v5) at ($(irc)+(0,0.4)$)  {$\vdots$};
\node[circle] (v6) at ($(irr)+(0,0.4)$)  {$\vdots$};

\draw[-latex] ($(1l)+(0,0.05)$) -- (1);
\draw[-latex] ($(1r)+(0,0.05)$) -- (1);

\draw[-latex] ($(2l)+(0,0.05)$) -- (2);
\draw[-latex] ($(2r)+(0,0.05)$) -- (2);
\draw[-latex] (2ll) -- (2l);
\draw[-latex] (2lc) -- (2l);
\draw[-latex] (2lr) -- (2l);
\draw[-latex] (2rl) -- (2r);
\draw[-latex] (2rc) -- (2r);
\draw[-latex] (2rr) -- (2r);

\draw[-latex] ($(il)+(0,0.05)$) -- (i);
\draw[-latex] ($(ir)+(0,0.05)$) -- (i);
\draw[-latex] (ill) -- (il);
\draw[-latex] (ilc) -- (il);
\draw[-latex] (ilr) -- (il);
\draw[-latex] (irl) -- (ir);
\draw[-latex] (irc) -- (ir);
\draw[-latex] (irr) -- (ir);
\end{tikzpicture}
\caption{The arborescences $\T_3^m$.}
\label{fig:T3m}
\end{figure}

\begin{definition}
    We call \emph{oriented line} the infinite graph whose vertices are $\{P_i\}_{i \in \ZZ}$, with $P_i \neq P_j$ for every $i \neq j$, and whose edges are $\{P_i \to P_{i+1}\}_{i \in \ZZ}$.
    Similarly, we call \emph{oriented semiline} the infinite graph whose vertices are $\{P_i\}_{i \in \NN}$, with $P_i \neq P_j$ for every $i \neq j$, and whose edges are $\{P_i \to P_{i+1}\}_{i \in \NN}$.
\end{definition}

\begin{theorem} \label{thm:structure}
Let $G$ be a group, $\ell$ be a prime integer and $\phi \in \End(G)$ such that $\left|\ker \phi \right| = \ell$. Let also
\[ m = \sup_{P \in \tau_{\0}} \big( d_{\{\0\}}(P)\big)\in \NN_{>0} \cup\{\infty\}. \]
Then $\tau_{\0}$ is isomorphic to $\T_\ell^{m}$.
Moreover, every connected component of $\fungraph{G}{\phi}$ is one of the following:
\begin{enumerate}
    \item\label{per} A cycle $\mathcal{C} = \{ P_1, \dots, P_r \}$, with $\tau_{P_i,\mathcal{C}} \simeq \tau_{\0}$ for every $i \in \{1,\dots,r\}$.
    \item\label{line} An oriented line $\mathcal{L} = \{ P_i \}_{i \in \ZZ}$, with $\tau_{P_i,\mathcal{L}} \simeq \tau_{\0}$ for every $i \in \ZZ$.
    \item\label{sline} An oriented semiline $\mathcal{N} = \{ P_i \}_{i \in \NN}$, where $\tau_{P_i,\mathcal{N}}$ is isomorphic to $\T_\ell^{\min(i,m)}$ for every $i \in \NN$.
\end{enumerate}

\end{theorem}
\begin{proof}
    Let $V_\0$ be the set of vertices of $\tau_{\0}$, 
    namely
    \[ V_{\0} = \bigcup_{i \in \NN} \ker \phi^{(i)}. \]
    Thus, $V_\0$ is a subgroup of $G$ containing the whole series of \cref{eq:groupseq}, and $\phi$ restricts to a group endomorphism of $V_\0$, whose kernel has still size $\ell$.
    In particular, every non-leaf element of $\tau_\0$ has indegree $\ell$, except for $\0$ which has indegree $\ell-1$ (since the loop $\phi(\0) = \0$ is not part of $\tau_{\0}$).
    The maximal depth in $\tau_{\0}$ is the (possibly infinite) index $m$ given by \Cref{lem:groupchain}, starting from which the series of \cref{eq:groupseq} stabilizes.
    For every $P \in V_{\0}$, we denote its distance from $\0$ by $d(P) = d_{\{\0\}}(P)$, 
    then we have
    \[ d = d(P) \iff P \in \ker \phi^{(d)} \setminus \ker \phi^{(d-1)}. \]
    If $m < \infty$, then $\tau_{\0}$ is finite, and by \Cref{lem:groupchain} the leaves are precisely
    \[ \tau_{\0} \setminus \phi( \tau_{\0} ) = \ker \phi^{(m)} \setminus \phi\big( \ker \phi^{(m)} \big) = \ker \phi^{(m)} \setminus \ker \phi^{(m-1)}, \]
    i.e., they all have depth $m$.
    On the other hand, when $m = \infty$, by \Cref{lem:groupchain} we know that every element of depth $d$ is the image of $\ell$ elements of depth $d+1$, therefore there can be no leaves.
    Thus, in both cases we conclude that $\tau_{\0}$ is isomorphic to $\T_\ell^{m}$.

    Every connected component either contains periodic elements (hence it contains a cycle -- case \ref{per}) or its points $P$ define infinite semilines $\{\phi^{(i)}(P)\}_{i \in \NN}$.
    In the latter case, either there is a semiline that can be completed to a line (case \ref{line}), or the component contains no oriented lines (case \ref{sline}).
    We discuss these three cases separately.
    
    \ref{per}: Let us assume that the connected component contains a cycle $\mathcal{C}$.
    The map $\phi|_{\mathcal{C}}$ is invertible, and for every $P \in \mathcal{C}$ the map
    \[ \tau_{\0} \to \tau_{P,\mathcal{C}}, \quad u \mapsto \phi|_{\mathcal{C}}^{-d(u)}(P) + u \]
    is an oriented tree isomorphism, which maps $\0$ to $P$, \textcolor{black}{for the same argument of \cite[Thm.\,2.8]{DEKLERK2022112691}}.
    Hence, for every $P \in \mathcal{C}$ this gives an isomorphism of arborescences $\tau_{\0} \simeq \tau_{P,\mathcal{C}}$.

    \ref{line}: Let us assume that the connected component contains an oriented line $\mathcal{L} = \{P_i\}_{i \in \ZZ}$, with $\phi(P_i) = P_{i+1}$.
    The restriction $\phi|_{\mathcal{L}} : P_i \mapsto P_{i+1}$ is invertible, and for any $P_i \in \mathcal{L}$ we define
    \[ \Phi : \tau_{\0} \to \tau_{P_i,\mathcal{L}}, \quad u \mapsto P_{i-d(u)} + u. \]
    By adapting the idea of \cite[Thm.\,2.8]{DEKLERK2022112691}, we prove that $\Phi$ is an isomorphism of arborescences.
    
    $\Phi$ is well-defined: if $d = d(u)$, then clearly $\phi^{(d)}\big( \Phi(u) \big) = P_i$.
    Moreover, let $0 \leq f \leq d$ be an integer such that $\phi^{(f)}\big( \Phi(u) \big) \in \mathcal{L}$, i.e., it is equal to $P_m$ for some $m \in \ZZ$.
    We have
    \[ P_i = \phi^{(d-f)} \circ \phi^{(f)}\big( \Phi(u) \big) =  \phi^{(d-f)} \circ P_m = P_{m+d-f}, \]
    which implies $m = i+f-d$ because $P_i$ is not periodic. Therefore
    \[ \phi^{(f)}(P_{i-d} + u) = \phi^{(f)}\big( \Phi(u) \big) = P_m = P_{i+f-d} \implies \phi^{(f)}(u) = \0, \]
    which proves that $f \geq d$. Hence, for every $u \in \tau_{\0}$, the minimal $f \in \NN$ such that $\phi^{(f)}\big( \Phi(u) \big) \in \mathcal{L}$ is $f=d(u)$, therefore $\Phi(u) \in \tau_{P_i,\mathcal{L}}$.
    
    $\Phi$ is surjective: let $Q \in \tau_{P_i,\mathcal{L}}$ and $r$ be the unique integer such that $\phi^{(r)}(Q) = P_i$. We define
    \[ u = - P_{i-r} + Q. \]
    It is easy to check that $\phi^{(r)}(u) = \0$, hence $u \in \tau_{\0}$.
    We claim that $d(u) = r$, which implies that $\Phi(u) = Q$.
    In fact, if $\phi^{(s)}(u) = \0$ for some integer $0 \leq s \leq r$, then
    \[ \phi^{(s)}(Q) = P_{i+s-r} \in \mathcal{L}. \]
    Since $\phi^{(s)}(Q) \in \mathcal{L}$ only for $s \geq r$, we conclude that $s=r$.
    
    $\Phi$ is injective: let $u_1,u_2 \in \tau_{\0}$ with $d_1 = d(u_1), d_2 = d(u_2)$ and such that $\Phi(u_1) = \Phi(u_2)$, i.e.,
    \begin{equation}\label{eq:uiinterm}
        u_1-u_2 = - P_{i-d_1} + P_{i-d_2}.
    \end{equation}
    Let $d = \max(d_1,d_2)$, then we have
    \[ \0 = \phi^{(d)}(u_1-u_2) = - P_{i+d-d_1} + P_{i+d-d_2}. \]
    Since $P$ is not periodic, this implies $d_1 = d_2$, therefore also $u_1 = u_2$ by \cref{eq:uiinterm}, proving injectivity.
    
    $\Phi$ respects edges: for every $u,v \in \tau_{\0}$ such that $u \neq \0$ and $\phi(u) = v$ we have $d(u) = d(v)+1$, therefore $\phi\big(\Phi(u)\big) = \Phi(v)$.
    The converse also holds: if $\phi\big(\Phi(u)\big) = \Phi(v)$, then
    \[ \phi(u) - v = - P_{i-d(u)+1} + P_{i-d(v)}, \]
    which implies as above that $d(u) = d(v)+1$, hence $\phi(u) = v$.

    \ref{sline}: Let us assume that the connected component contains an oriented semiline $\mathcal{N} = \{ P_i \}_{i \in \NN}$, with $\phi(P_i) = P_{i+1}$.
    Unless $\mathcal{N}$ is part of an oriented line (case \ref{line}), we can assume $\phi^{-1}(P_0) = \emptyset$.
    For every $i \in \NN$, let us consider the arborescence $\tau_i$ obtained by restricting $\tau_{\0}$ to elements of $\ker \phi^{(i)}$.
    Since $\tau_{\0} \simeq \T_\ell^m$, using \Cref{lem:groupchain} as above we conclude that $\tau_i \simeq \T_\ell^{\min(i,m)}$.
    We discuss the finite and infinite cases separately.
    
    $\bullet$ $m = \infty$: We prove that the map 
    \[ \Psi : \tau_i \to \tau_{P_i,\mathcal{N}}, \quad u \mapsto P_{i-d(u)} + u, \]
    is a well-defined isomorphism of arborescences.

    $\Psi$ is well-defined: for every $u \in \ker \phi^{(i)}$ we have $d(u) \leq i$, hence $P_{i-d(u)}$ is a well-defined point of~$\mathcal{N}$.
    Checking that $\Psi(u) \in \tau_{P_i,\mathcal{N}}$ is completely analogous to \ref{line}.

    $\Psi$ is invertible: injectivity follows as in \ref{line}. 
    It is easy to check that, for every $0 \leq j \leq i$, we have
    \[ \ker \phi^{(j)} + P_{i-j} = \phi^{(-j)}(P_i). \]
    Since $| \tau_{P_0,\mathcal{N}} | = |\{ P_0 \}| = 1$, this inductively implies that for every $i \geq 1$ we have
    \[ | \tau_{P_i,\mathcal{N}} | = | \tau_{P_i} | - |\tau_{P_{i-1}}| = \sum_{j = 0}^i \ell^j - \sum_{j = 0}^{i-1} \ell^j = \ell^i = \bigl| \ker \phi^{(i)} \bigr| = |\tau_i|, \]
    from which the surjectivity of $\Psi$ follows.

    Finally, $\Psi$ respects the edges as in \ref{line}, hence it is an isomorphism of arborescences.
    
    $\bullet$ $m < \infty$: We begin by refining the semiline.
    Let $Q \in \tau_{P_m}$ be a point with the longest finite path to $P_m$, which exists since $m$ is finite and we are not in case \ref{line}.
    We consider the new semiline $\mathcal{N}' = \{ Q_i \}_{i \in \NN}$, where $Q_i = \phi^{(i)}(Q)$.
    We prove that for every $i \in \NN$ the map 
    \[ \Psi : \tau_i \to \tau_{Q_i,\mathcal{N}'}, \quad u \mapsto Q_{i-d(u)} + u, \]
    is a well-defined isomorphism of arborescences.

    $\Psi$ is well-defined with the same argument of the case $m = \infty$.

    $\Psi$ is invertible: injectivity follows as in \ref{line}. 
    For every $i < m$, the same counting argument of case $m = \infty$ still holds, and implies $|\tau_i| = |\tau_{Q_i,\mathcal{N}'}|$.
    For $i = m$, there cannot be elements $R \in \tau_{Q_i}$ with $d_{\mathcal{N}'}(R) > m$, by definition of $Q$.
    For $i > m$, there are no elements $R \in \tau_{Q_i}$ with $d(R) = m+1$, otherwise we would have
    \[ R - Q_{i-m-1} \in \ker \phi^{(m+1)} \setminus \ker \phi^{(m)}, \]
    contradicting \Cref{lem:groupchain}.
    Therefore, even for any $i \geq m$ we inductively conclude that
    \[ |\tau_{Q_i,\mathcal{N}'}| = |\tau_{Q_i}| - |\tau_{Q_{i-1}}| = \sum_{j = 0}^{m-1} \ell^j + (i-m+1)\ell^m - \sum_{j = 0}^{m-1} \ell^j - (i-m)\ell^m = \ell^m = 
    \bigl|\ker \phi^{(i)}\bigr| = |\tau_i|, \]
    hence $\Psi$ is surjective. 
    
    Again, $\Psi$ respects the edges as in \ref{line}, hence it is an isomorphism of arborescences.
\end{proof}

The different cases when $\ell=3$, which is the case of interest for the Hessian, are portrayed in \Cref{fig:StructPict}.

\begin{figure}[ht]
    \centering
\begin{tikzpicture}[node distance=2cm]

\node (0) at (0,0) {$\0$};
\node (tau0) at ($(0)+(-0,-1.5)$) {$\tau_{\0}$};
\node[circle,scale=0.7] (0l) at ($(0)+(-0.6, 0.8)$) {};
\node[circle,scale=0.7] (0r) at ($(0)+(0.6, 0.8)$) {};
\coordinate (0ll) at ($(0l)+(-0.4,0.8)$) {};
\coordinate (0lc) at ($(0l)+(0,0.8)$) {};
\coordinate (0lr) at ($(0l)+(0.4,0.8)$) {};
\coordinate (0rl) at ($(0r)+(-0.4,0.8)$) {};
\coordinate (0rc) at ($(0r)+(0,0.8)$) {};
\coordinate (0rr) at ($(0r)+(0.4,0.8)$) {};

\draw[-latex] ($(0l)+(0,0.05)$) -- (0);
\draw[-latex] ($(0r)+(0,0.05)$) -- (0);
\draw[-latex] (0ll) -- (0l);
\draw[-latex] (0lc) -- (0l);
\draw[-latex] (0lr) -- (0l);
\draw[-latex] (0rl) -- (0r);
\draw[-latex] (0rc) -- (0r);
\draw[-latex] (0rr) -- (0r);


\node[circle,scale=0.7] (02) at (3, -0.01) {\textbullet};
\node (i) at ($(02)+(-0,-1.5)$) {Case \ref{per}};
\node (02P) at ($(02)+(0,-0.3)$) {$P_i$};
\node[circle,scale=0.7] (02r2) at ($(02)+(0.8,-0.2)$) {\textbullet};
\node[circle,scale=0.7] (02l2) at  ($(02)+(-0.8,-0.2)$) {\textbullet};

\node[circle,scale=0.7] (02l) at ($(02)+(-0.6, 0.8)$) {};
\node[circle,scale=0.7] (02r) at ($(02)+(0.6, 0.8)$) {};
\coordinate (02ll) at ($(02l)+(-0.4,0.8)$) {};
\coordinate (02lc) at ($(02l)+(0,0.8)$) {};
\coordinate (02lr) at ($(02l)+(0.4,0.8)$) {};
\coordinate (02rl) at ($(02r)+(-0.4,0.8)$) {};
\coordinate (02rc) at ($(02r)+(0,0.8)$) {};
\coordinate (02rr) at ($(02r)+(0.4,0.8)$) {};

\draw[-latex] (02l2) edge[bend left =10] (02);
\draw[-latex] (02) edge[bend left =10] (02r2);
\draw[-latex,dashed] (02r2) edge[bend left=80] (02l2);

\draw[-latex] ($(02l)+(0,0.05)$) -- (02);
\draw[-latex] ($(02r)+(0,0.05)$) -- (02);
\draw[-latex] (02ll) -- (02l);
\draw[-latex] (02lc) -- (02l);
\draw[-latex] (02lr) -- (02l);
\draw[-latex] (02rl) -- (02r);
\draw[-latex] (02rc) -- (02r);
\draw[-latex] (02rr) -- (02r);


\node[circle,scale=0.7] (03) at (6.5, -0.02) {\textbullet};
\node (ii) at ($(03)+(-0,-1.5)$) {Case \ref{line}};
\node (03P) at ($(03)+(0,-0.3)$) {$P_i$};
\node[circle,scale=0.7] (03r2) at ($(03)+(0.8,0)$) {\textbullet};
\node[circle,scale=0.7] (03l2) at  ($(03)+(-0.8,0)$) {\textbullet};
\node[circle,scale=0.7] (03rr2) at ($(03r2)+(0.8,0)$) {};
\node[circle,scale=0.7] (03ll2) at  ($(03l2)+(-0.8,0)$) {};

\node[circle,scale=0.7] (03l) at ($(03)+(-0.6, 0.8)$) {};
\node[circle,scale=0.7] (03r) at ($(03)+(0.6, 0.8)$) {};
\coordinate (03ll) at ($(03l)+(-0.4,0.8)$) {};
\coordinate (03lc) at ($(03l)+(0,0.8)$) {};
\coordinate (03lr) at ($(03l)+(0.4,0.8)$) {};
\coordinate (03rl) at ($(03r)+(-0.4,0.8)$) {};
\coordinate (03rc) at ($(03r)+(0,0.8)$) {};
\coordinate (03rr) at ($(03r)+(0.4,0.8)$) {};

\draw[-latex] (03l2) edge (03);
\draw[-latex] (03) edge (03r2);
\draw[-latex,dashed] (03r2) edge (03rr2);
\draw[-latex,dashed] (03ll2) edge (03l2);

\draw[-latex] ($(03l)+(0,0.05)$) -- (03);
\draw[-latex] ($(03r)+(0,0.05)$) -- (03);
\draw[-latex] (03ll) -- (03l);
\draw[-latex] (03lc) -- (03l);
\draw[-latex] (03lr) -- (03l);
\draw[-latex] (03rl) -- (03r);
\draw[-latex] (03rc) -- (03r);
\draw[-latex] (03rr) -- (03r);


\node[circle,scale=0.7] (04) at (11, -0.02) {\textbullet};
\node (iii) at ($(04)+(-0.7,-1.5)$) {Case \ref{sline}};
\node (04P) at ($(04)+(0,-0.3)$) {$P_2$};
\node[circle,scale=0.7] (04r2) at ($(04)+(0.8,0)$) {};
\node[circle,scale=0.7] (04l2) at  ($(04)+(-1.4,0)$) {\textbullet};
\node (04P1) at ($(04l2)+(0,-0.3)$) {$P_1$};
\node[circle,scale=0.7] (04ll2) at  ($(04l2)+(-0.8,0)$) {\textbullet};
\node (04P2) at ($(04ll2)+(0,-0.3)$) {$P_0$};

\node[circle,scale=0.7] (04l) at ($(04)+(-0.6, 0.8)$) {};
\node[circle,scale=0.7] (04r) at ($(04)+(0.6, 0.8)$) {};
\coordinate (04ll) at ($(04l)+(-0.4,0.8)$) {};
\coordinate (04lc) at ($(04l)+(0,0.8)$) {};
\coordinate (04lr) at ($(04l)+(0.4,0.8)$) {};
\coordinate (04rl) at ($(04r)+(-0.4,0.8)$) {};
\coordinate (04rc) at ($(04r)+(0,0.8)$) {};
\coordinate (04rr) at ($(04r)+(0.4,0.8)$) {};

\coordinate (04l2r) at ($(04l2)+(0.2,0.8)$) {};
\coordinate (04l2l) at ($(04l2)+(-0.2,0.8)$) {};

\draw[-latex] (04l2) edge (04);
\draw[-latex,dashed] (04) edge (04r2);
\draw[-latex] (04ll2) edge (04l2);

\draw[-latex] ($(04l)+(0,0.05)$) -- (04);
\draw[-latex] ($(04r)+(0,0.05)$) -- (04);
\draw[-latex] (04ll) -- (04l);
\draw[-latex] (04lc) -- (04l);
\draw[-latex] (04lr) -- (04l);
\draw[-latex] (04rl) -- (04r);
\draw[-latex] (04rc) -- (04r);
\draw[-latex] (04rr) -- (04r);

\draw[-latex] ($(04l2r)+(0,0.05)$) -- (04l2);
\draw[-latex] ($(04l2l)+(0,0.05)$) -- (04l2);
\end{tikzpicture}
    \caption{Possible connected components of a group endomorphism with a finite kernel of cardinality $\ell=3$. The cases correspond to those of \Cref{thm:structure}.}
    \label{fig:StructPict}
\end{figure}

\begin{remark}
\label{rem:northcott}
  When $G$ is a subgroup of $E_k(\kbar)$, the finite connected components (if any) of $\fungraph{G}{\phi}$ are fully described by case~\ref{per} of \Cref{thm:structure}.
  In particular, if $G$ is finite, this is the only case that appears.
  On the other hand, when $G \leq E_k(\K)$ for a number field $\K$, and $\phi$ is the restriction of an endomorphism of $E_k$, the number of preperiodic elements is finite (see \cite[Thm.\,5]{Northcott} or \cite[Thm.\,3.12]{silverman2007arithmetic}), hence case~\ref{per} of \Cref{thm:structure} covers only a finite number of elements.
  We remark that every case described in \Cref{thm:structure} can actually occur for suitable choices of $\fungraph{G}{\phi}$, as discussed in \Cref{App:tightnessOfStructureThm}. 
\end{remark}


\section{\textcolor{black}{Reduced Lattès diagrams}} \label{sec:Sk}

{\color{black}
In this section, we provide general results that can be employed to study Lattès maps fitting in a reduced diagram, as in \Cref{defn:Lattes}, with a particular focus on non-algebraically closed fields.
We will further assume that the projection $\pi$ is one of those given in \Cref{rmk:CMforRigid}.
In particular, \Cref{thm:StructureS} shows how the fibers of $\pi$ can be expressed as unions of subgroups of $E$, and \Cref{prop:newdisjpartsSk} shows that these subgroups intersect each other at (at most) six points.
In \Cref{subsec:Hessianpifibers}, we specialize these results to the Hessian transformation. Finally, in \Cref{subsec:piVSpsi}, \Cref{prop:GenProjectingSk} illustrates how the functional graph of $\psi$ is transformed under $\pi$.}

{\color{black}{
\subsection{The fiber of \texorpdfstring{$\pi$}{π}}

When $\K = \kbar$, the dynamics of a Lattès map $\phi:\P^1(\K) \to \P^1(\K)$ fitting in a reduced diagram
 can be easily read by quotienting $\fungraph{E(\K)}{\psi}$ for the action of $\Gamma \subseteq \Aut(E)$.
 The situation is different when $\K \neq \kbar$, as the dynamics of $\phi$ depends on the action of $\psi$ restricted to 
 \[ \S_{\pi} = \pi^{-1}\big( \P^1(\K) \big) \subseteq E(\kbar). \]
 Hereafter, we will consider only curves $E$ defined by short Weierstrass equations, and with one of the four projections $\pi:E(\kbar) \to \P^1(\kbar)$ described in \Cref{rmk:CMforRigid}.
 This assumption is fulfilled by the Hessian (cf. \Cref{thm:projectphi3}), but does not hold for general Lattès maps since the conjugation over $\kbar$ does not imply functional graph isomorphism over $\K$ \cite[§7]{Silverman2012}. 
 
 With these assumptions, we have $E(\K) \subseteq \S_{\pi} \subseteq E(\kbar)$, although both inclusions may be strict.
}
We will employ the notation of twists introduced in \Cref{lemma:twists}.
In particular, let $n = |\Aut(E)|$ and $\iota_{D} : E^{(D)} \xrightarrow{\simeq} E$ be the $\kbar$-isomorphism corresponding to a fixed choice of $D^{1/n}$.
To resolve the ambiguity in the choice of the root, we consider
\[ \TT_{\pi} = \{ \sigma \in \Aut(E) \ | \ \sigma(\S_{\pi}) = \S_{\pi} \}. \]
It is easy to see that $\TT_{\pi}$ is a group, which fits in the chain of subgroups
\[ \langle \Gamma, \Aut_{\K}(E) \rangle \leq \TT_{\pi} \leq \Aut(E). \]
In fact, by inspecting the possible choices for $\Gamma$ and $j(E)$, one explicitly finds that
\begin{equation} \label{eq:TT}
    \TT_{\pi} = \begin{cases} 
    \Aut_{\K}(E) & \textnormal{if } j(E)=0 \textnormal{ and } |\Gamma| = 2, \\
    \Aut(E) & \textnormal{otherwise}.
\end{cases}
\end{equation} 
The next lemma further characterizes the elements of $\TT_{\pi}$. 
\begin{lemma} \label{lem:TT}
    Let $\TT_{\pi}$ be defined as above.
    For every $\sigma \in \Aut(E)$, we have $\sigma \in \TT_{\pi}$ if and only if $\exists \varphi_{\sigma} \in \Aut_{\K}(\P^1)$ such that $\pi \circ \sigma = \varphi_{\sigma} \circ \pi$.
\end{lemma}
\begin{proof}
    Let $\sigma \in \TT_{\pi}$. For every $Q \in \P^1(\kbar)$, we consider any $P_Q \in E(\kbar)$ such that $\pi(P_Q) = Q$, and set
    \[ \varphi_{\sigma} : \P^1(\kbar) \to \P^1(\kbar), \quad Q \mapsto \pi\big( \sigma(P_Q) \big). \]
    This map is well-defined: since $\Aut(E)$ is commutative and $\pi$ is invariant on $\Gamma$-orbits, the map is independent of the choice of $P_Q$.
    Since $\Aut(E)$ acts linearly on the coordinates of $(x,y) \in E(\kbar)$, then we conclude that $\varphi_{\sigma}$ is a scalar map, which satisfies $\pi \circ \sigma = \varphi_{\sigma} \circ \pi$ by construction.
    Moreover, a direct inspection of all the possible choices for $\Gamma$ and $j(E)$ shows that it is defined over $\K$, thus $\varphi_{\sigma} \in \Aut_{\K}(\P^1)$.

    On the other hand, for each $P \in \S_{\pi}$, we have $\pi\big( \sigma(P) \big) = \varphi_{\sigma}\big( \pi(P ) \big) \in \P^1(\K)$,
    therefore $\sigma(P) \in \S_{\pi}$, so $\sigma(\S_{\pi}) \subseteq \S_{\pi}$. 
    Conversely, for every $P \in \S_{\pi}$, we have
    \[ \pi \circ \sigma^{-1}(P) = \varphi_{\sigma}^{-1} \circ \pi (P) \in \P^1(\K), \]
    hence $\sigma^{-1}(P) \in \S_{\pi}$, i.e., $\S_{\pi} \subseteq \sigma(\S_{\pi})$.
\end{proof}

Although $\S_{\pi}$ is defined over a field that is potentially larger than $\K$, we prove that it is completely determined by isomorphic copies of $E^{(D)}(\K)$, for certain twists $E^{(D)}$ of $E$.
\textcolor{black}{\begin{remark} \label{rmk:chooseroot}
    Each twist intrinsically contains the information about \emph{any} choice among the $n$-th roots of $D$, as it corresponds to a 1-cocycle in $H^{1}\big(\Gal(\kbar,\K),\Aut(E)\big)$ \cite[§X.2-5]{silv:arithEll}.
    We now 
     need to make this choice semi-explicit. 
    Namely, for any integer factorization $n = cm$, we define
    \[ \iota_{D^m} : E^{(D^m)} \to E, \quad (x,y) \mapsto \big(D^{-2/c}x,\;D^{-3/c}y\big). \]
    This corresponds to fixing $(D^{m})^{1/n} = D^{1/c}$, for any choice of $c$-th root of $D$.
    Concretely, this means that among the $m$-th roots of $D^m$, we always pick $D$ itself.
    The choice of $D^{1/c} \in \kbar$ is arbitrary, but it is fixed once and for all for every $D \in \K^*/(\K^*)^c$, which means that we associate a unique embedding $\iota_{D^m} : E^{(D^m)} \to E$ to each twist $E^{(D^m)}$.
    Choosing a different root is the same as composing $\iota_{D^m}$ by a suitable automorphism of $E$.
\end{remark}}
\begin{theorem} \label{thm:StructureS}
    Let $E(\kbar) \subset \P^2(\kbar)$ be defined by $y^2 = x^3+Ax+B$, with $A,B \in \K$.
    Let $\Gamma \subseteq \Aut(E)$ and $\pi : E \to E/\Gamma \simeq \P^1$ be one of the four maps of \Cref{rmk:CMforRigid}.
    Let also 
    $c = |\Gamma|$ and $m = [\Aut(E):\Gamma]$.
    Then
    \[ \S_{\pi} = \bigcup_{\substack{\sigma \in \TT_{\pi}/\Aut_{\K}(E)\\D \in \K^*/(\K^*)^{c}}} \sigma \circ\iota_{D^{m}}\big(E^{(D^m)}(\K)\big). \]
\end{theorem}

\begin{proof}
    Let $n = |\Aut(E)| = cm$. We prove the inclusions separately.

    $\supseteq )$ Let $D \in \K^*$ and $Q = (x,y) \in E^{(D^m)}(\K)$. 
    Let $P = \iota_{D^m}(Q) = (\frac{x}{D^{2/c}}, \frac{y}{D^{3/c}})$ (\Cref{rmk:chooseroot}).
    Since $E \simeq E^{(D)}$ over $\kbar$, we have $\Aut(E) = \Aut(E^{(D)})$, hence the projection $\pi^{(D)} : E^{(D)} \to E^{(D)}/\Gamma \simeq \P^1$ is defined by the same rational function of $\pi$.
    A straightforward computation then shows that
    \[ \pi(P) = \frac{ \pi^{(D^m)}(Q) }{D} \in \P^1(\K) \quad \implies \quad P \in \S_{\pi}. \]
    For every $\sigma \in \TT_{\pi}$, we have $\sigma(P) \in \S_{\pi}$, hence the post-composition with $\sigma$ preserves the inclusion.

    $\subseteq )$ Let $(x,y) \in \S_{\pi}$. In all possible cases, we explicitly exhibit a $D \in \K^*$ such that the point 
    \[ Q = (xD^{2/c}, yD^{3/c}) \in E^{(D^m)}(\K) \]
    satisfies $(x,y) = \sigma \circ\iota_{D^{m}}(Q)$, for a suitable choice of $\sigma \in \TT_{\pi}$.

    \begin{itemize}[leftmargin=1.1cm]
        \item[$c=2$:] When the projection is unramified (i.e., $y \neq 0$), let $D = y^2 \in \K^*$.
        Since $x = \pi(x,y) \in \K$, then $y^2 \in \K^*$, and we directly verify that $Q = (xy^{2},y^{4}) \in E^{(D^m)}(\K)$.
        In this case, we obtain $\iota_{D^{m}}(Q) = (x,\pm y)$, depending on the choice of the square root, hence we retrieve $(x,y)$ up to the action of $[-1] \in \TT_{\pi}$.
        If $y=0$, we have $(x,0) \in E(\K)$, hence the choice $D=1$ works.
        \item[$c=3$:] We necessarily have $j(E) = 0$, therefore $(n,m)=(6,2)$.
        If $x \neq 0$, then $D = x^3 = y^2-B \in \K^*$, since $y = \pi(x,y) \in \K$.
        We immediately check that $Q = (x^{3},x^{3}y) \in E^{(D^2)}(\K)$, and that $\iota_{D^{2}}(Q) \in \{(\zeta_3^i x, y)\}_{i \in \{0,1,2\}}$, hence there exists an element of $\Aut(E) = \TT_{\pi}$ (\cref{eq:TT}) that brings it to $(x,y)$.
        If $x=0$, then we choose $D=1$ and $Q = (x,y)$.
        \item[$c=4$:] In this case, we have $j(E) = 1728$, therefore $(n,m)=(4,1)$.
        Since $x^2 = \pi(x,y) \in \K$, if $y \neq 0$ we consider $D = y^4 = (x^3+Ax)^2 \in \K^*$.
        We then directly verify that $Q = (xy^{2},y^{4}) \in E^{(D)}(\K)$ is the requested point, as $\iota_{D^{m}}(Q) \in \{\sigma(x,y)\}_{\sigma \in \Aut(E)}=\{\sigma(x,y)\}_{\sigma \in \TT_{\pi}}$ (\cref{eq:TT}).
        If $y = 0$, then we have two possibilities: either $x = 0$, hence we consider $D=1$ and $Q = (0,0)$, or $x \neq 0$, hence we consider $D = x^2 \in \K^*$ and $Q = (x^2,0)$.
        \item[$c=6$:] This case occurs with $j(E) = 0$, therefore $(n,m)=(6,1)$ and $x^3 = \pi(x,y) \in \K$, then also $y^2 \in \K$.
        When $xy \neq 0$, we consider $D = x^6y^6 \in \K^*$ and $Q = (x^3y^2, x^3y^4) \in E^{(D)}(\K)$, so that $\iota_{D}(Q) \in \{\sigma(x,y)\}_{\sigma \in \Aut(E)}=\{\sigma(x,y)\}_{\sigma \in \TT_{\pi}}$ (\cref{eq:TT}).
        When $x = 0 \neq y$, we consider $D=y^2\in \K^*$ and $Q = (0,y^2)$, while if $x \neq 0 = y$, we consider $D=x^6 \in \K^*$ and $Q = (x^3,0)$.
    \end{itemize}

    Finally, we note that $\sigma \circ \iota_{D^{m}}\big(E^{(D^m)}(\K) \big)$ does not change by varying $D$ by any element in $(\K^*)^c$, and $\sigma$ by any element in $\Aut_{\K}(E)$.
    In fact, if $D_u = Du^c$ for some $u \in \K^*$, then 
    \[ \iota_{D_u^{m}}\big(E^{(D_u^m)}(\K) \big) = \iota_{D^{m}} \circ \iota_{u^{n}} \big(E^{(D^mu^n)}(\K) \big) = \iota_{D^{m}}\big(E^{(D^m)}(\K) \big), \]
    hence the points in $\iota_{D_u^{m}}\big(E^{(D_u^m)}(\K) \big)$ were already covered by $\iota_{D^{m}}\big(E^{(D^m)}(\K) \big)$.
    Similarly, when $\sigma$ is replaced by $\tau \circ \sigma$ for some $\tau \in \Aut_{\K}(E)$, we have
    \[ \tau \circ \sigma \circ \iota_{D^{m}}\big(E^{(D^m)}(\K) \big) = \sigma \circ \iota_{D^{m}}\Big( \tau\big(E^{(D^m)}(\K) \big)\Big) = \sigma \circ \iota_{D^{m}} \big(E^{(D^m)}(\K) \big), \]
    where the last equality follows from $\Aut_{\K}(E) = \Aut_{\K}(E^{(D)})$.
\end{proof}

\begin{remark} \label{rmk:firstunion}
    By \Cref{thm:StructureS}, $\S_{\pi}$ depends on the base field $\K$ in two, essentially independent, ways.
    The first union adjusts the (arbitrary) choice of $n$-th roots in $\kbar$, while the second selects a suitable set of twists of $E$.
    Since $|\TT_{\pi}/\Aut_{\K}(E)| \leq 3$, the first union is typically small, and it is trivial for $j(E) \neq 0,1728$.
    By \Cref{lem:TT}, we know that this union only contributes to a multiple covering of $\P^1(\K)$, hence it is not essential from a dynamical perspective (although it is needed for the set equality). 
    The second union, instead, defines a typically irredundant covering of $\P^1(\K)$, and might well consist of infinitely many components (e.g., when $\K = \QQ$).    
\end{remark}

\begin{remark}
    \Cref{thm:StructureS} establishes a general decomposition that can be used to study Lattès fitting \textcolor{black}{a reduced Lattès diagram, with $E$ in (short) Weierstrass form}.
    In particular, \Cref{cor:Spi4Hessian} provides an algebraic interpretation of the sets $A_n, B_n$ defined in \cite[§3]{ugolini2014} and \cite[§3]{ugolini2018functional}.
    However, we remark that in our setting, the considered curve does not need to be ordinary. In fact, when $E_k$ is supersingular, this implies interesting dynamical contraints (see, e.g., \Cref{sec:q2mod3}).
\end{remark}

{\color{black}
The following proposition shows that the union of \Cref{thm:StructureS} is essentially minimal, as two distinct components intersect at most in points with a zero coordinate.


\begin{proposition} \label{prop:newdisjpartsSk}
     Let $E$, $c$, and $m$ be as in \Cref{thm:StructureS}.
     Let $D_1, D_2 \in \K^*/(\K^*)^c$, $\sigma_1, \sigma_2 \in \TT_{\pi}/\Aut_{\K}(E)$, and
     \[ [x:y:z] \in \sigma_1 \circ\iota_{D_1^{m}}\big(E^{(D_1^m)}(\K)\big) \cap \sigma_2 \circ\iota_{D_2^{m}}\big(E^{(D_2^m)}(\K)\big). \]
     Then either $xyz=0$, or $D_1 = D_2$ and $\sigma_1 = \sigma_2$.
\end{proposition}
\begin{proof}
    For $z \neq 0$, we can restrict the proof to affine points.
    Since $\sigma_i \in \Aut(E)$, then there exists $u_i \in \kbar^*$ such that $\sigma_i(x,y) = (u_i^2x,u_i^3y)$.
    Therefore, there are $(x_i,y_i) \in E^{(D_i^m)}(\K)$ such that
    \[ ( u_1^2D_1^{-2/c}x_1, u_1^3D_1^{-3/c}y_1 ) = (x,y) = ( u_2^2D_2^{-2/c}x_2, u_2^3D_2^{-3/c}y_2 ). \]
    If $xy \neq 0$, then $x_1x_2y_1y_2 \neq 0$ and
    \[ \frac{D_1^{2/c}}{D_2^{2/c}} \left(\frac{u_2}{u_1}\right)^2 = \frac{x_1}{x_2} \in \K^*,
    \qquad
    \frac{D_1^{3/c}}{D_2^{3/c}} \left(\frac{u_2}{u_1}\right)^3 = \frac{y_1}{y_2} \in \K^*. 
    \]
    Dividing the second equation by the first, we get $\frac{D_1^{1/c} u_2}{D_2^{1/c} u_1} \in \K^*$.
    We straightforwardly check that, for every possible value of $c$, we have $(\frac{u_2}{u_1})^c = (\pm 1)^c$, therefore $\frac{D_2}{D_1} \in (\K^*)^c$, namely $D_1 = D_2 \in \K^*/(\K^*)^c$.
    By \Cref{rmk:chooseroot}, this implies that $\iota_{D_1^{m}} = \iota_{D_2^{m}}$, i.e., $D_1^{1/c} = D_2^{1/c}$.
    Therefore, we conclude that also $\frac{u_2}{u_1} \in \K^*$, which implies that
    $\sigma_2 \circ \sigma_1^{-1} \in \Aut_{\K}(E)$,
    namely $\sigma_1 = \sigma_2 \in \TT_{\pi}/\Aut_{\K}(E)$.
\end{proof}
}

\subsection{The fibers of \texorpdfstring{$\pi_{\Lambda}$}{πΛ} and \texorpdfstring{$\pi_{\Hess}$}{πH}} \label{subsec:Hessianpifibers}

We now apply \Cref{thm:StructureS} specifically to the study of the Hessian transformation.}

\begin{definition} \label{defn:SlSh}
    Let $k \in \K^*$. We define
    \[ \S_{\Lambda_k} = \pi_{\Lambda}^{-1}\big( \P^1(\K) \big) \cap E_k(\kbar) \quad \textnormal{and} \quad \S_{\Hess_k} = \pi_{\Hess}^{-1}\big( \P^1(\K) \big) \cap E_k(\kbar). \]
\end{definition}
\begin{remark} \label{rmk:graph_epi}
\textcolor{black}{It is clear that $\S_{\Lambda_k}$ and $\S_{\Hess_k}$ are the specializations of $\S_{\pi}$ arising from the reduced Lattès diagrams of \Cref{thm:projectphi3}.
Thus, they define the functional graph epimorphisms}
\[ \pi_{\Lambda} : \langle \S_{\Lambda_k}, \psi_k \rangle \to \langle \P^1(\K), \Lambda_k \rangle \quad \textnormal{and} \quad \pi_{\Hess} : \langle \S_{\Hess_k}, \psi_k \rangle \to \langle \P^1(\K), \Hess_k \rangle. \]
Therefore, to describe the functional graph of $\Lambda_k$ (resp.\ $\Hess_k$), it is sufficient to understand the functional graph of $\psi_k$ on $\S_{\Lambda_k}$ \big(resp.\ $\S_{\Hess_k}$\big) and how it projects under $\pi_{\Lambda}$ (resp.\ $\pi_{\Hess}$).
\end{remark}

\begin{remark} \label{rmk:extensions} In general, we do not need to consider the whole $\kbar$: the points in $\S_{\Lambda_k}$ are defined over
    \[ \K\left( \sqrt{x^3+k/4} \right)_{x \in \K}  \subseteq \overline{\K}, \]
while the coordinates of those of $\S_{\Hess_k}$ lie in
    \[ \K\left( \sqrt{x}, \sqrt[3]{x} \right)_{x \in \K}  \subseteq \overline{\K}. \]
Since $\S_{\Lambda_k}$ and $\S_{\Hess_k}$ are closed with respect to $\psi_k$, both the above fields must contain $\sqrt{-3}$. In fact, one can straightforwardly check that, \textcolor{black}{by defining $x_t = -\frac{t^3+k}{3t^2}$ for any $t \in \K^*$, we have
\[ -3 = \frac{\left(t^3+\frac{k}{4}\right)}{\left(x_t^3+\frac{k}{4}\right)}\Big(\frac{t^3-2k}{3t^2}\Big)^2, \]
whenever it is defined, i.e., when $t^3 \notin \{2k,-\frac{k}{4}\}$.}
Hence, the whole $\Aut(E_k)$ is defined over these fields, since they contain $\z$.
\end{remark}
\textcolor{black}{\begin{corollary} \label{cor:Spi4Hessian}
    Let $k \in \K^*$. Then 
    \[ \S_{\Lambda_k} = \bigcup_{u \in \K^*/(\K^*)^2} \phi_{u^{-\frac{1}{2}}}\big(E_{u^3k}(\K)\big) \qquad \text{and} \qquad \S_{\Hess_k} = \begin{cases}
        \bigcup_{\substack{i \in \{0,1,2\} \\ u \in \K^*/(\K^*)^6}} \phi_{\z^iu^{-\frac{1}{6}}}\big(E_{uk}(\K)\big) \quad &\textnormal{if } \z \notin \K, \\[3ex]
        \bigcup_{u \in \K^*/(\K^*)^6} \phi_{u^{-\frac{1}{6}}}\big(E_{uk}(\K)\big) \quad &\textnormal{if } \z \in \K.
    \end{cases} 
    \]
\end{corollary}
\begin{proof}
    It follows by \Cref{thm:StructureS} applied to the reduced Lattès diagrams given by \Cref{thm:projectphi3}, where $\TT_{\pi}$ is given explicitly by \cref{eq:TT}.
\end{proof}}

\textcolor{black}{\begin{remark} \label{rmk:groupsrespectedbypsi}
The fiber $\S_{\pi}$ of the projection $\pi$ is a union of groups, by \Cref{thm:StructureS}.
The dynamics of a Lattès map with such a projection is dictated by the action of $\psi$ on those groups.
If $\psi$ is a scalar multiplication (which implies that the Lattès map is flexible), these groups are all closed under the action of $\psi$, therefore the group structure of the twists of $E$ completely determines the underlying dynamics.
However, this need not hold for a general $\psi$.
For instance, these groups may (\Cref{sec:q1mod3}) or may not (\Cref{sec:q2mod3}) be closed under the lifted Hessian $\psi = \psi_k$.
\end{remark}}

\textcolor{black}{By \Cref{rmk:groupsrespectedbypsi}, the decompositions of $\S_{\Lambda_k}$ and $\S_{\Hess_k}$ are respected by $\psi_k^{(2)} = [-3]$ (\Cref{lem:psi2}). }
The next lemma shows that $\psi_k$ is actually closed on the union of (possibly identical) pairs of these groups.

\begin{lemma} \label{lem:Spartsclosed}
    Let $k,u \in \K^*$. Then, for any fixed choice of $u^{\frac{1}{6}} \in \kbar$, we have
    \[ \psi_k\Big(\phi_{u^{-\frac{1}{6}}}\big(E_{uk}(\K)\big)\Big) \subseteq \phi_{(-\frac{u}{27})^{-\frac{1}{6}}}\big(E_{-\frac{u}{27}k}(\K)\big) \quad \textnormal{and} \quad \psi_k\Big(\phi_{(-\frac{u}{27})^{-\frac{1}{6}}}\big(E_{-\frac{u}{27}k}(\K)\big)\Big) \subseteq \phi_{u^{-\frac{1}{6}}}\big(E_{uk}(\K)\big). \]
\end{lemma}
\begin{proof}
    Let $(x,y) \in E_{uk}(\K)$. A straightforward computation shows that
    \[ \psi_k \big( \phi_{u^{-\frac{1}{6}}} (x,y) \big) = \left(-\frac{x^3+uk}{3x^2u^{\frac{1}{3}}}, -y\frac{x^3-2ku}{3\sqrt{-3}x^3u^{\frac{1}{2}}}\right) = \phi_{(-\frac{u}{27})^{-\frac{1}{6}}}\left(\frac{x^3+uk}{9x^2}, y\frac{x^3-2ku}{27x^3}\right), \]
    which belongs to $\phi_{(-\frac{u}{27})^{-\frac{1}{6}}}\big(E_{-\frac{u}{27}k}(\K)\big)$.
    The other inclusion can be deduced by substituting $u$ with $-\frac{u}{27}$:
    \[ \psi_k\Big(\phi_{(-\frac{u}{27})^{-\frac{1}{6}}}\big(E_{-\frac{u}{27}k}(\K)\big)\Big) \subseteq \phi_{(\frac{u}{3^6})^{-\frac{1}{6}}}\big(E_{\frac{u}{3^6}k}(\K)\big) = \phi_{u^{-\frac{1}{6}}}\big(E_{uk}(\K)\big), \]
    where the last equality follows by \textcolor{black}{\Cref{thm:StructureS}, since they correspond to the same embedding of the same twist of $E_k$.}
\end{proof}

\textcolor{black}{\subsection{Projecting the graph} \label{subsec:piVSpsi} }

{\color{black} 
In this section, we describe how, for a reduced Lattès diagram, the functional graph $\fungraph{\S_{\pi}}{\psi}$ projects to $\fungraph{\P^1(\K)}{\phi}$. We stress that the projected graph is not necessarily as regular as those arising from group endomorphisms, since $\S_{\pi}$ need not be a group, and $\pi$ can identify some of its parts.

The following lemma refines the notion of compatibility between $\psi$ and $\Gamma \subseteq \Aut(E)$, which was introduced in \Cref{rem:compatibility}.
\begin{lemma} \label{lem:GammaActionPsi}
   There exists $\epsilon \in \{\pm 1\}$ such that, for every $\gamma \in \Gamma$ and $d\geq 0$, we have
    \[ \psi^{(d)} \circ \gamma = \gamma^{\epsilon^d} \circ \psi^{(d)}. \]
\end{lemma}
\begin{proof}
For $d=0$ the statement is trivial.
Let us now prove the case $d=1$.

We claim that, for every $\gamma \in \Gamma$, there exists a unique $\theta_\gamma \in \Gamma$ such that
    \[ \psi \circ \gamma = \theta_\gamma \circ \psi. \]
    To prove this, we observe that the reduced Lattès structure implies that every $R \in E(\kbar)$ 
    is sent to $\0$ by one of the 
    morphisms $\{\psi \circ \gamma - \gamma' \circ \psi\}_{\gamma' \in \Gamma}$. In particular, one of them, say $\theta_\gamma$, must map infinitely many points to $\0$ and is therefore the zero morphism. To conclude the proof of our claim, we need to check that $\theta_\gamma$ is unique: if $\psi \circ \gamma = \theta_\gamma \circ \psi = \theta'_\gamma \circ \psi$, one has $(\theta_\gamma - \theta'_\gamma)\circ \psi=0$, which implies $\theta_\gamma = \theta'_\gamma$ by surjectivity of $\psi$~\cite[Thm.\,I.2.3]{silv:arithEll}. 
    
    One can directly check that 
    \[ \theta: \Gamma \to \Gamma, \quad \gamma \mapsto \theta_{\gamma}, \]
    is a group homomorphism.
    A similar argument to that used above proves that $\theta$ is also  injective, as
    \[ \theta_{\gamma} = 1 \quad \implies \quad \psi \circ \gamma = \psi \quad \implies \quad \psi \circ (\gamma - 1) = 0 \quad \implies \gamma=1.\]
    Therefore, $\theta \in \Aut(\Gamma)$. Since $\Gamma \simeq \mu_i$ for some ${i \in \{2,3,4,6\}}$, and $\Aut(\mu_i) = \langle x \mapsto x^{-1} \rangle$, this implies that either
    $\theta(\gamma)=\gamma$ or $\theta(\gamma)=\gamma^{-1}$
    for all $\gamma \in \Gamma$, which is the case $d=1$.
    The case $d>1$ is proven by induction, since
    \[ \psi^{(d+1)} \circ \gamma = \psi \circ (\psi^{(d)} \circ \gamma) = (\psi \circ \gamma^{\epsilon^d}) \circ \psi^{(d)} = (\gamma^{\epsilon^d})^\epsilon \circ \psi\circ \psi^{(d)} = \gamma^{\epsilon^{d+1}} \circ \psi^{(d+1)}, \]
    which is the inductive step. 
\end{proof}

As a consequence, the action of $\Gamma$ respects periodic elements and the depth in ancestor trees.

\begin{lemma} \label{lem:GammapreservesPer}
    For every $\gamma \in \Gamma$ and $P \in E(\kbar)$, we have
    \[ P \in \Per(\psi) \quad \iff \quad \gamma(P) \in \Per(\psi). \]
\end{lemma}
\begin{proof}
    If $P \in \Per(\psi)$, then there is $k\in \NN$ such that $\psi^{(k)}(P) = P$.
   \Cref{lem:GammaActionPsi} yields
    \[ \psi^{(2k)} \circ \gamma(P) = \gamma^{\epsilon^{2k}} \circ \psi^{(2k)}(P) = \gamma(P), \]
    therefore $\gamma(P) \in \Per(\psi)$.
    The reverse implication follows by replacing $P$ by $\gamma(P)$, and $\gamma$ by $\gamma^{-1}$.
\end{proof}

\begin{lemma}
\label{lem:PhiPreservesDist}
    Let $\epsilon \in \{\pm1\}$ be defined as in \Cref{lem:GammaActionPsi}.
    For every $\gamma \in \Gamma$, $P \in E(\kbar)$ and $Q \in \tau_P$, let $d = d_{P}(Q)$.
    Then we have
    \[ \gamma^{\epsilon^d}(Q) \in \tau_{\gamma(P)} \quad \textnormal{and} \quad d_{\gamma(P)}\big( \gamma^{\epsilon^d}(Q) \big) = d. \]
    Moreover, $\gamma(P)\in \tau_P$ if and only if $\gamma(P)=P$.
\end{lemma}
\begin{proof}
    The statement trivially holds when $d=0$, so we assume $d>0$.  
     \Cref{lem:GammaActionPsi} yields
    \begin{equation} \label{eq:reachtheroot}
         \psi^{(d)} \circ \gamma^{\epsilon^d}(Q) = \gamma^{\epsilon^{2d}} \circ \psi^{(d)}(Q) = \gamma(P).
    \end{equation}
    By \Cref{def:AncestorTree}, none of $\{ \psi^{(k)}(Q) \}_{0 \leq k < d}$ is periodic, then by \Cref{lem:GammapreservesPer} none of $\{ \psi^{(k)} \circ \gamma' (Q) \}_{0 \leq k < d}$ can be periodic, for any choice of $\gamma' \in \Gamma$.
    Thus, we conclude that $\gamma^{\epsilon^d}(Q) \in \tau_{\gamma(P)}$.
    
    We now show that $\gamma(P) \notin \{ \psi^{(k)} \circ \gamma^{\epsilon^d} (Q) \}_{0 \leq k < d}$, which implies the equality between distances.
    Let us assume by contradiction that $\psi^{(k)} \circ \gamma^{\epsilon^d} (Q) = \gamma(P)$ for some $0 \leq k < d$. 
    Then by \cref{eq:reachtheroot} we have    
    \[ \gamma(P) = \psi^{(d)} \circ \gamma^{\epsilon^d}(Q) = \psi^{(d-k)} \circ \psi^{(k)} \circ \gamma^{\epsilon^d}(Q) = \psi^{(d-k)} \circ \gamma(P), \]
    therefore $\gamma(P) = \psi^{(k)} \circ \gamma^{\epsilon^d} (Q) = \gamma^{\epsilon^{d+k}} \circ \psi^{(k)}(Q)$ is periodic.
    By \Cref{lem:GammapreservesPer}, we conclude that also $\psi^{(k)}(Q)$ was periodic, which contradicts the fact that $\tau_{P} \setminus \{P\}$ cannot contain periodic points.

For the last statement, let $\gamma(P)\in\tau_P$ and
$ n=d_P(\gamma(P))$.
Let us assume by contradiction that $n>0$. \Cref{lem:GammaActionPsi} gives
$\psi^{(n)}\circ\gamma = \gamma^{\epsilon^n}\circ\psi^{(n)}$,
which implies
\[
    \psi^{(n)}(P) = \psi^{(n)} \big( \gamma (P) \big) = \gamma^{-\epsilon^n}(P).
\]
 We also have
\begin{align*}
    \psi^{(2n)}(P)
    =
    \psi^{(n)}\bigl(\gamma^{-\epsilon^n}(P)\bigr)=
    \gamma^{-\epsilon^{2n}}\bigl(\psi^{(n)}(P)\bigr)
    =
    \gamma^{-(1+\epsilon^n)}(P).
\end{align*}
Since $\psi^{(2n)}$ commutes with $\gamma$ by \Cref{lem:GammaActionPsi}, iterating the above identity $\ord(\gamma)$ times yields
\[
    \psi^{(2n\cdot\ord(\gamma))}(P)
    =
    \gamma^{-\ord(\gamma)(1+\epsilon^n)}(P)
    =
    P.
\]
Hence $P$ is periodic, and therefore so is $\gamma(P)$ by \Cref{lem:GammapreservesPer}.
This implies that $\gamma(P)\in\tau_P\setminus\{P\}$ is a periodic point, contradicting the definition of $\tau_P$.
Thus we conclude that $n=0$, and consequently $\gamma(P)=P$.
The converse implication is trivial.
\end{proof}

\begin{lemma}
\label{lem:pipreservesdep}  
Let $P \in \S_{\pi}$. Then
\begin{itemize}
    \item if $Q \in \tau_P \cap \S_{\pi}$, then $d_P(Q)=d_{\pi(P)}(\pi(Q))$.
    \item \( 
P \in \Per(\psi) \iff \pi(P) \in \Per(\phi).\) If this is the case, let $k$ be the period of $\pi(P)$ and $\gamma \in \Gamma$ such that $\psi^{(k)}(P)=\gamma (P)$. Then, letting $\epsilon \in \{\pm1\}$ be defined as in \Cref{lem:GammaActionPsi}, the period of $P$ is $Nk$ with
\[N=\begin{cases}
    1 \qquad &\text{if $\gamma(P)=P$}\\
    2 \qquad &\text{if $\gamma(P)\neq P$ and $\epsilon^k=-1$}\\
    \min_{m \in \NN_{>0}}\{ \gamma^m(P) = P \} \qquad &\text{otherwise}.\\
\end{cases}\]
\end{itemize}
\end{lemma}

\begin{proof} It is clear that $d_P(Q) \geq d_{\pi(P)}(\pi(Q))$.
By contradiction, let us assume that the inequality is strict, i.e. $\phi^{(d)}(\pi(Q))=\pi(P)$ for some $d<d_P(Q)$.
By commutativity of the Lattès diagram, this is equivalent to saying that $\pi(\psi^{(d)}(Q))=\pi(P)$, that is $\psi^{(d)}(Q)=\gamma(P)\in \tau_P$ for some $\gamma \in \Gamma$ such that $\gamma(P) \neq P$. This cannot happen, as it would  contradict \Cref{lem:PhiPreservesDist}. Therefore, $d_P(Q)=d_{\pi(P)}(\pi(Q))$.

If $P \in \Per(\psi)$, then $\psi^{(k)}(P)=P$ for some positive integer $k$. Hence $\phi^{(k)}(\pi(P))=\pi(\psi^{(k)}(P))=\pi(P)$. Vice versa, suppose that $\pi(P)$ is periodic of period $k$. Then $\psi^{(k)}(P)=\gamma (P)$ for some $\gamma \in \Gamma$. If $\gamma(P)=P$, the statement follows immediately. Otherwise, if $\epsilon^k=-1$, \Cref{lem:GammaActionPsi} yields
\[\psi^{(2k)}(P)=\psi^{(k)}(\gamma(P))=\gamma^{\epsilon^k+1}(P)=P.\]
 If $\epsilon^k=1$,  \Cref{lem:GammaActionPsi} implies that, for every $m \in \NN$,
 \[\psi^{(mk)}(P)=\gamma^m(P),\]
 which completes the proof.
\end{proof}

\textcolor{black}{\begin{remark} \label{rmk:NdividesorderGamma}
    We note that, in the above lemma, we have
    \[ \min_{m \in \NN_{>0}}\{ \gamma^m(P) = P \} = \ord_{\Gamma/\textnormal{Stab}_{\Gamma}(P)}([\gamma]) \ | \ |\Gamma|. \]
    Thus, $N$ divides $|\Gamma|$ whenever $|\Gamma| \neq 3$ or $\epsilon = 1$.
    The latter holds for the lifted Hessian $\psi_k$ (\Cref{cor:psiAut}).
\end{remark}}

\begin{remark}
\Cref{lem:pipreservesdep} applies also to any point $P \in E(\kbar)$, by replacing $\phi$ with $\bar{\phi}:\P^1(\kbar) \to \P^1(\kbar)$, according to \Cref{defn:Lattes}. 
This can also be viewed as a consequence of the fact that $\pi$ is a surjective $|\Gamma|$-to-one semiconjugacy~\cite[§6]{milnor2006lattes}. 
\end{remark}

\begin{proposition} \label{prop:identifyancestors}
    Let $P,Q \in \S_{\pi}$ with $\pi(P) = \pi(Q)$.
    Then we have an isomorphism of ancestor trees $\tau_{P} \simeq \tau_{Q}$ in $E(\kbar)$, which restricts to an isomorphism of arborescences
    \[ \tau_{P} \cap \S_{\pi} \simeq \tau_{Q} \cap \S_{\pi}. \]
    Moreover, they project to the same arborescence $\pi(\tau_{P} \cap \S_{\pi})=\pi(\tau_{Q} \cap \S_{\pi})$ in $\P^1(\K)$.
\end{proposition}
\begin{proof}
Since $\pi(P)=\pi(Q)$, there exists $\gamma \in \Gamma$ such that $Q = \gamma (P)$. By \Cref{lem:GammaActionPsi}, there exists $\epsilon \in \{\pm 1\}$ such that $\psi \circ \gamma = \gamma^{\epsilon} \circ \psi.$ 
Letting  $d(H)=d_P(H)$ for brevity, we now show that 
    \[ \Phi : \tau_{P} \to \tau_Q, \quad H \mapsto \gamma^{\epsilon^{d(H)}}(H) \]
   is a well-defined isomorphism of ancestor trees. The map $\Phi$ is well defined by \Cref{lem:PhiPreservesDist} with  $Q=\gamma(P)$. For the same reason, we have a well-defined set-theoretical inverse
   \[ \Phi^{-1} : \tau_{Q} \to \tau_P, \quad H \mapsto \gamma^{-\epsilon^{d(H)}}(H).\]
   We now prove that $\Phi$ respects the edges. Let us consider $P_1,P_2 \in \tau_P$, with $\psi(P_1)=P_2$ and $d(P_2)=d(P_1)-1$. Hence, we obtain
   \[ \Phi(P_2)=\gamma^{\epsilon^{d(P_2)}}(P_2)=\gamma^{\epsilon^{d(P_1)-1}}\circ \psi(P_1)=\psi \circ \gamma^{\epsilon^{d(P_1)}}(P_1)=\psi \circ \Phi(P_1).\]
   Vice versa, if $\psi \circ \Phi(P_1)=\Phi(P_2)$ then 
   \[ \gamma^{\epsilon^{d(P_2)}}(P_2)=\psi \circ \gamma^{\epsilon^{d(P_1)}}(P_1)=\gamma^{\epsilon^{d(P_1)-1}} \circ \psi (P_1).\]
   Since $\Phi$ preserves the distance from the root by \Cref{lem:PhiPreservesDist}, we have $d(P_1)=d(P_2)+1$, from which we conclude $\psi(P_1)=P_2$.

   Since $\pi \circ \Phi(H)=\pi(H)$, the isomorphism $\Phi$ restricts to $\S_{\pi}$, i.e.\ $H \in \tau_P \cap \S_\pi$ if and only if $\Phi(H) \in \tau_Q \cap \S_\pi$. Finally, $\pi(\tau_{P} \cap \S_{\pi})=\pi(\tau_{Q} \cap \S_{\pi})$ and it is an arborescence because $\phi$ is a Lattès map and no cycles/loops are created by the projection of $\tau_{P} \cap \S_{\pi}$ by $\pi$ thanks to \Cref{lem:pipreservesdep}.
\end{proof}
\begin{remark}
When $\psi$ commutes with $\Gamma$, we have $\epsilon = 1$ in \Cref{lem:GammaActionPsi}, hence the isomorphism of arborescences in the above proof is simply $\Phi=\gamma$. In particular, by \Cref{cor:psiAut}, this is the case for the lifted Hessian $\psi_k$.
    \end{remark}
}

\textcolor{black}{\begin{proposition} \label{prop:GenProjectingSk}
    Let $\Gcal$ be a connected component of $\fungraph{\S_{\pi}}{\psi}$ containing periodic points.
    Then
    \begin{enumerate}
        \item \label{itm:epimorph} $\pi$ projects $\Gcal$ onto a connected component of $\fungraph{\P^1(\K)}{\phi}$.
        \item \label{itm:lemma} This projection preserves the depth of elements and maps the cycle of periodic points in $\Gcal$ into a cycle whose length is $N$ times shorter, where $N$ is given by \Cref{lem:pipreservesdep} for any periodic point $P$ of $\Gcal$.
        \item \label{itm:arbor} This projection is an isomorphism of arborescences rooted in periodic points, with the following possible exception: the branches arriving at points in $\ker(\psi\circ(\gamma-1))$ for $\gamma \in \Gamma \setminus \{1\}$ might be identified.
    \end{enumerate}
\end{proposition}
\begin{proof}
    \ref{itm:epimorph}: The reduced Lattès structure of $\phi$ implies that $\pi$ is a graph epimorphism  $\pi : \langle \S_{\pi}, \psi \rangle \to \langle \P^1(\K), \phi \rangle$ mapping connected components into connected components.
    In particular, the image of $\Gcal$ is surjectively covered precisely by
    \[ \{ \pi \circ \gamma (\Gcal) \}_{\gamma \in \Gamma} . \]
    \ref{itm:lemma}: It follows directly from \Cref{lem:pipreservesdep}.
    \\ \ref{itm:arbor}: The isomorphism of arborescences follows if the projection preserves the indegree of vertices.
    If the indegree is not preserved for a certain vertex $R$, then there are two distinct vertices $P, Q \in \mathcal{S}_\pi$ such that 
    \[\psi(P)=\psi(Q)=R \quad \text{and} \quad \pi(P)=\pi(Q).\]
    In particular, the latter equality implies that there exists $\gamma \in \Gamma \setminus\{1\}$ such that $Q=\gamma(P)$, so that the former equality  implies $P \in \ker(\psi\circ(\gamma-1))$. The whole branches arriving at $P$ and $Q$ are therefore identified by \Cref{prop:identifyancestors}.
\end{proof}
}

\section{The functional graph of \texorpdfstring{$\psi_k$}{ψk}}
\label{sec:Functionalpsik}
\textcolor{black}{Since $\Lambda_k$ and $\Hess_k$ fit into reduced Lattès diagrams by \Cref{thm:projectphi3}, the results in \Cref{sec:Sk} can be applied to $\psi_k$ to describe the Hessian action.}
In this section, we prove that the functional graph of this curve endomorphism enjoys several additional properties: its loops correspond to $2$-torsion points (\Cref{subsec:selfloops}) and the depth of its points only depends on their order (\Cref{subsec:CyclesDepth}). 
\textcolor{black}{Moreover, in \Cref{sec:projthegraph}, we refine \Cref{prop:GenProjectingSk} by characterizing those branches that are identified via $\pi_{\Hess}$ and $\pi_{\Lambda}$. }



\subsection{Loops} \label{subsec:selfloops}
The next lemma characterizes the \emph{loops} (i.e., cycles of length $1$) in $\fungraph{E_k(\overline{\K})}{\psi_k}$.

\begin{lemma} \label{lem:selfloops}
  Let $k \in \K^*$ and $P \in E_k(\overline{\K})$. Then 
    \[ \psi_k(P) = P \quad \iff \quad P \in E_k[2] = \left\{\0, \left(-\z^i\sqrt[3]{\frac{k}{4}},0\right)\right\}_{i \in \{0,1,2\}}.
    \]
 The corresponding loops in the Hessian graphs are $\infty$ and $-\frac{k}{4}$.
\end{lemma}
\textcolor{black}{\begin{proof}
    The condition $\psi_k(P) = P$ is equivalent to $P \in \ker(\psi_k-[1])$. By \Cref{defn:psik}, we have
    \[ \psi_k-[1] = [2](\rho-[1]) = [2] \rho^2 \in \End(E_k), \]
    and $\ker([2]\rho^2) = \ker([2])$, since $\rho \in \Aut(E_k)$.
    It is clear that $E_k[2] = \{\0, (-\z^i\sqrt[3]{k/4},0)\}_{i \in \{0,1,2\}}$, and these points are mapped to $\infty$ or $-\frac{k}{4}$ after $\pi_{\Hess}$.
\end{proof}}

\begin{remark}
    We note that there may be loops in $\fungraph{\P^1(\K)}{\Hess}$ (\Cref{lem:loopj}) that do not arise from loops in $\fungraph{E_k(\kbar)}{\psi}$ (\Cref{lem:selfloops}).
    This happens when the cubing map $M_3$ is not injective over $\P^1(\K)$, and $\pi_{\Hess}$ sends non-trivial cycles of $\fungraph{\S_{\Hess_k}}{\psi}$ into loops of $\fungraph{\P^1(\K)}{\Hess}$.
\end{remark}


\subsection{Cycles and depth} \label{subsec:CyclesDepth}
A point $P \in E_k(\kbar)$ is referred to as a \emph{torsion} point if $\ord(P)<\infty$, i.e., it is $m$-torsion for some $m \in \ZZ_{>0}$.

\begin{proposition} \label{prop:torsion}
    Let $k \in \K^*$. A point $P \in E_k(\overline{\K})$ is periodic in $\fungraph{E_k(\overline{\K})}{\psi_k}$ if and only if it is a torsion point and
    \[ {\gcd}\big(\ord(P),3\big)=1. \]
\end{proposition}
\begin{proof}
    By \Cref{lem:psi2} we have $\psi_k^{(2)} = [-3]$, then 
    \begin{equation} \label{eq:ord-3}
        \ord\big(\psi_k \circ \psi_k(P)\big) = \begin{cases}
        \ord(P)/3 &\textnormal{if } \ord(P) < \infty \textnormal{ and } 3 \ | \ \ord(P), \\
        \ord(P) &\textnormal{otherwise. } 
        \end{cases} 
    \end{equation}
    If $P$ is periodic there exists $e\in \NN_{>0}$ such that $\psi^{(e)}(P) = P$, then also $[(-3)^e]P = P$, so $P$ is a torsion point.
    If $e=1$, then $\psi_k(P) = P$, hence $[2]P = \0$ by \Cref{lem:selfloops}, so ${\gcd}\big(\ord(P),3\big) = 1$.
    If $e \geq 2$, then $P = \psi_k^{(e)}(P) = \psi_k^{(e-2)}([-3]P)$. 
    Since $\psi_k$ is a group endomorphism, then $\ord\big(\psi_k(P)\big) \ | \ \ord(P)$, so in particular we must have $\ord(P) = \ord([-3]P)$, hence $3 \nmid \ord(P)$ by \cref{eq:ord-3}.
    
    On the other side, if $P$ is \textcolor{black}{a torsion point} and ${\gcd}\big(\ord(P),3\big) = 1$, then there exists the multiplicative order $e \in \NN_{>0}$ of $-3$ modulo $\ord(P)$, namely $(-3)^e \equiv 1 \bmod \ord(P)$. Therefore
    \[ \psi_k^{(2e)}(P) = [(-3)^e]P = P, \]
    hence $P$ is periodic.
\end{proof}


\begin{proposition} \label{prop:Ord3}
    Let $k \in \K^*$, $P \in E_k(\overline{\K})$ be a preperiodic point and $d = d_{\textnormal{Per}(\psi_k)}(P)$. Then
    \[ \ord(P) = 3^{\lceil d/2 \rceil} N, \]
    where $N = \ord\big(\psi_k^{(d)}(P)\big)$ and ${\gcd}(N,3) = 1$.
\end{proposition}
\begin{proof}
    Since $\psi_k^{(d)}(P)$ is periodic, by \Cref{prop:torsion} it is sufficient to prove that
    \begin{equation} \label{eq:ordP}
        \ord(P) = 3^{\lceil d/2 \rceil} \ord\big(\psi_k^{(d)}(P)\big).
    \end{equation} 
    We recall again that $\psi_k^{(2)} = [-3]$ by \Cref{lem:psi2}.
    If $d = 2d'$ is even, then by the minimality of $d$ and \cref{eq:ord-3}, it follows that
    \[ \ord\big(\psi_k^{(d)}(P)\big) = \ord\big([(-3)^{d'}]P\big) = \ord(P)/3^{d'}, \]
    which proves \cref{eq:ordP}.
    If $d = 2d'+1$ is odd instead, the same argument shows that
    \[ \ord\big(\psi_k^{(2d')}(P)\big) = \ord\big((-3)^{d'}P\big) = \ord(P) / 3^{d'}. \]
    However, by \Cref{prop:torsion} and the minimality of $d$, we also have that $3\nmid\ord\big(\psi_k^{(d)}(P)\big)$, while ${3 \mid \ord\big(\psi_k^{(d-1)}(P)\big)}$. Therefore
    \[ \ord\big(\psi_k^{(d)}(P)\big) = \ord\big(\psi_k^{(2d')}(P)\big) / 3 = \ord(P) / 3^{d'+1}, \]
    which is \cref{eq:ordP} for $d$ odd.
\end{proof}


\subsection{Projecting the Hessian} \label{sec:projthegraph}
\textcolor{black}{In this section we specialize \Cref{prop:GenProjectingSk} to describe the projection of $\fungraph{\S_{\Lambda_k}}{\psi_k}$ (resp.\ $\fungraph{\S_{\Hess_k}}{\psi_k}$) through $\pi_{\Lambda}$ (resp.\ $\pi_{\Hess}$).
The main improvement we obtain for the Hessian map, compared to a generic Lattès map fitting into a reduced Lattès diagram, is the full characterization of the arborescences that are identified after the projection.}


\begin{lemma} \label{lem:ProjPi}
    Let $k \in \K^*$ and $P,Q \in \S_{\Lambda_{k}}$ such that $P \neq Q$ and $P \neq \psi_k(P) = \psi_k(Q) \neq Q$. Then
    \[ \pi_{\Lambda}(P) = \pi_{\Lambda}(Q) \iff [2] \psi_k(P) = \0.\]
\end{lemma}
\begin{proof}    
    If $\pi_{\Lambda}(P) = \pi_{\Lambda}(Q)$ then $P = - Q$, hence $\psi_k(P) = \psi_k(Q) = - \psi_k(P)$, so $[2] \psi_k(P) = \0$.
    Conversely, if $[2] \psi_k(P) = \0$, then $\psi_k(P) = -\psi_k(P) = \psi_k(-P)$, hence $P$ and $-P$ have the same image under $\psi_k$. However, by \Cref{lem:selfloops} we also have that $\psi_k\big( \psi_k(P) \big) = \psi_k(P)$ and $P \neq -P$. 
    Since \textcolor{black}{$\deg \psi_k = 3$}, then 
    \[ Q \in \psi_k^{-1}\big( \psi_k(P) \big) = \{P, - P, \psi_k(P)\}. \]
    Thus, since $Q \neq \psi_k(P)$, we conclude that $Q = -P$, so $\pi_{\Lambda}(P) = \pi_{\Lambda}(Q)$.
\end{proof}

\begin{lemma} \label{lem:ProjM3Pi}
    Let $k \in \K^*$ and $P,Q \in \S_{\Hess_k}$ such that $P \neq Q$ and $P \neq \psi_k(P) = \psi_k(Q) \neq Q$. Then
    \[ \pi_{\Hess}(P) = \pi_{\Hess}(Q) \iff [2] \psi_k(P) = \0 \quad \textnormal{or} \quad \psi_k^{(2)}(P) = \0.\]
\end{lemma}
\begin{proof} 
    If $\pi_{\Hess}(P) = \pi_{\Hess}(Q)$ then there is $\Phi \in \Aut(E_k)$ such that $P = \Phi(Q)$. By \Cref{cor:psiAut} we have
    \[ \Phi \big( \psi_k (P) \big) = \psi_k\big(\Phi(P) \big) = \psi_k(Q) = \psi_k(P). \]
    Since $P \neq Q$, the point $\psi_k(P)$ is then fixed by a non-trivial element of $\Aut(E_k)$, namely it has a zero coordinate. Thus, it is either a $2$-torsion point or $\pm T_k$.
    Conversely, if $[2] \psi_k(P) = \0$, then $\pi_{\Lambda}(P) = \pi_{\Lambda}(Q)$ as in \Cref{lem:ProjPi}.
    If, instead, $P =(x,y)$  with $x \neq 0 \neq y$ and $\psi_k^{(2)}(P) = \0$, then $\{ (\z^{i} x, \pm y) \}_{i \in \{0,1,2\}}$ are six distinct points.
    Since \textcolor{black}{$\deg \psi_k = 3$}, we obtain
    \[ Q \in \psi_k^{-1}\big( \psi_k(P) \big) \subseteq \psi_k^{-1}(\{\pm T_k\}) = \{ (\z^{i} x, \pm y) \}_{i \in \{0,1,2\}}, \]
    which proves that $Q = \Phi(P)$ for some $\Phi \in \Aut(E_k)$, hence $\pi_{\Hess}(P) = \pi_{\Hess}(Q)$.
\end{proof}

The coverings of $\S_{\Lambda_k}$ and $\S_{\Hess_k}$ established in \Cref{cor:Spi4Hessian} are almost disjoint, \textcolor{black}{as proved in \Cref{prop:newdisjpartsSk}}.
The next lemma shows that this remains true even after projecting to $\P^1(\K)$.
We will prove it only for $\pi_{\Hess}$, which immediately implies an analogous result for $\pi_{\Lambda}$.

\begin{lemma} \label{lem:disjointTwistsProjs}
    Let $k,u,v \in \K^*$ with $\frac{u}{v} \not\in (\K^*)^6$.
    Then, for any choice of $u^{\frac{1}{6}}, v^{\frac{1}{6}} \in \kbar$, we have
    \[ \pi_{\Hess} \Big( \phi_{u^{-\frac{1}{6}}}\big( E_{uk}(\K)\big) \Big) \cap \pi_{\Hess} \Big( \phi_{v^{-\frac{1}{6}}}\big( E_{vk}(\K)\big) \Big) \subseteq \left\{ 0,-\frac{k}{4},\infty \right\}.
    \]
\end{lemma}
\begin{proof}
    If $P \in \phi_{u^{-\frac{1}{6}}}\big( E_{uk}(\K)\big) \cap \phi_{v^{-\frac{1}{6}}}\big( E_{vk}(\K)\big)$, then one of the coordinates of $P$ is $0$ by \textcolor{black}{\Cref{prop:newdisjpartsSk}}, hence $\pi_{\Hess}(P) \in \{0, -\frac{k}{4}, \infty\}$ by \Cref{lem:selfloops}.
    Let us assume by contradiction that there are points
    \[ P = (x,y) \in \phi_{u^{-\frac{1}{6}}}\big( E_{uk}(\K)\big) \quad \textnormal{and} \quad Q \in \phi_{v^{-\frac{1}{6}}}\big( E_{vk}(\K)\big), \]   
    having all non-zero coordinates and such that $\pi_{\Hess}(P) =\pi_{\Hess}(Q)$, i.e., $Q \in \{(\z^i x, \pm y)\}_{i \in \{0,1,2\}}$. 
    Thus, we have
    \[\begin{cases}
        \phi_{u^{\frac{1}{6}}}(P) \in E_{uk}(\K) &\implies u^{\frac{1}{3}}x, u^{\frac{1}{2}}y \in \K^*, \\
        \phi_{v^{\frac{1}{6}}}(Q) \in E_{vk}(\K) &\implies v^{\frac{1}{3}}\z^i x, \pm v^{\frac{1}{2}}y \in \K^*. 
    \end{cases} \]
    The above conditions imply $\frac{u}{v} \in (\K^*)^2 \cap (\K^*)^3 = (\K^*)^6$, which contradicts the hypothesis.
\end{proof}
\begin{proposition} \label{prop:projectingSk}
    Let $\Gcal$ be a connected component of $\fungraph{\S_{\Lambda_k}}{\psi_k}$ (resp.\ $\fungraph{\S_{\Hess_k}}{\psi_k}$) containing periodic points.
    Then $\pi_{\Lambda}$ (resp.\ $\pi_{\Hess}$) projects $\Gcal$ onto a connected component of $\fungraph{\P^1(\K)}{\Lambda_k}$ (resp.\ $\fungraph{\P^1(\K)}{\Hess_k}$).
    This projection maps the cycle of periodic points in $\Gcal$ into a cycle whose length is $1$ or $2$ (resp.\ $1,2,3$, or $6$) times shorter, and preserves the depth of elements.
    Finally, it is an isomorphism of arborescences rooted in periodic points, with the following exceptions:
    \begin{itemize}
        \item the branches arriving at $2$-torsion points are identified, and
        \item (only for $\pi_{\Hess}$) the branches arriving at $\pm T_k$ are identified.
    \end{itemize}
\end{proposition}
\begin{proof}
    The first part of the statement follows immediately from the reduced Lattès structure of $\Lambda_k$ and $\Hess_k$ (\Cref{thm:projectphi3}) and parts \ref{itm:epimorph} and \ref{itm:lemma} of \Cref{prop:GenProjectingSk}. 
    \textcolor{black}{In particular, the integer $N$ from part \ref{itm:lemma} arises from \Cref{lem:pipreservesdep}, and always divides $|\Gamma|$ by \Cref{rmk:NdividesorderGamma}.}
    Finally, the isomorphism of arborescences follows if the projection preserves the indegree of vertices.
    By \Cref{lem:ProjPi} (resp.\ \Cref{lem:ProjM3Pi}) this holds for all vertices of $\S_{\Lambda_k}$ (resp.\ $\S_{\Hess_k}$) not arriving at $2$-torsion points (resp.\ or $\pm T_k$), while the ancestors of these exceptional points are identified and, consequently, their whole branches are identified by \Cref{prop:identifyancestors}.
\end{proof}


Some concrete examples of the projection properties prescribed by \Cref{prop:projectingSk} are displayed in \Cref{App:examples} (\Cref{fig:hessian29psi,fig:hessian29,fig:hessian31psi,fig:hessian31}).


\section{Hessian graphs over finite fields}
\label{sec:HessianFF}

In this section, $\K$ will be a finite field $\F_q$, for a given prime power $q=p^r \in \ZZ$, with $p \neq 2,3$.
In this case, to every elliptic curve $E$ over $\F_q$ we can associate its \emph{trace} (of Frobenius) \cite[Rmk.\,V.2.6]{silverman2007arithmetic}:
\[ \tr(E) = q+1-\left|E(\F_q)\right|. \]
Moreover, as observed in \Cref{rmk:extensions}, \textcolor{black}{both $\fungraph{\P^1(\K)}{\Hess_k}$ and $\fungraph{\P^1(\K)}{\Lambda_k}$} can be understood by means of small extensions of $\F_q$.
Indeed, according to \Cref{defn:SlSh}, one can directly check that
\begin{align*}
    \S_{\Lambda_k} &= \{ (x,y) \in E_k(\F_{q^2}) \ | \ x \in \F_q \} \cup \{\0\}, \\
    \S_{\Hess_k} &= \{ (x,y) \in E_k(\F_{q^6}) \ | \ x^3 \in \F_q \} \cup \{\0\}.
\end{align*}
In this setting, we have effective criteria for counting the roots of cubics \cite[§2-3]{dickson:critIrrFinFld}, which we can employ to characterize the indegree (with respect to $\psi_k$) of points whose first entry belongs to a given field.
Such an indegree notably depends only on their second entry.

\begin{proposition} \label{prop:indegree}
    Let $k \in \F_q^*$ and $\textcolor{black}{P \in } \S = \S_{\Lambda_k} \setminus\{\0\}$. 
    Then
    \[ |(\psi_k|_{\S})^{-1}(P)| = 1 \iff [2]P\neq\0 \quad \text{and} \quad  \begin{cases}
            P \notin E_k(\F_q) &\quad \text{if $-3k \in (\F_q)^2$},\\
            P \in E_k(\F_q) &\quad \text{if $-3k \notin (\F_q)^2$}.\\
        \end{cases} \] 
\end{proposition}
\begin{proof}
    Let $P = (x_P,y_P)$.
    Since $\psi_k$ is linear in the second entry, the fiber $(\psi_k|_{\S})^{-1}(P)$ only depends on the number of solutions of
    \[ x^3+3x_Px^2+k \in \F_q[x]. \]
    By~\cite[§2]{dickson:critIrrFinFld}, such an equation has precisely one solution over $\F_q$ if and only if the discriminant of this cubic polynomial, which is
    \[ -108 k \left(x_P^3+\frac{k}{4}\right) = (6y_P)^2(-3k), \]
    is not a square in $\F_q$.
    This happens if and only if
    \begin{equation*}
    y_P\neq 0 \quad \text{and} \quad
        \begin{cases}
            y_P \notin \F_q &\quad \text{if $-3k$ is a square in $\F_q$},\\
            y_P \in \F_q &\quad \text{if $-3k$ is not a square in $\F_q$},
        \end{cases}
    \end{equation*}
    which proves the statement.
\end{proof}

In the remaining, we aim at detailing the structure of $\langle \P^1(\F_q), \Hess \rangle$, therefore we fix $k = -6912$ (\Cref{rmk:psiHess}), and we drop it from our notation.
The Lattès structure is then given by the curve
\[ E \ : \ y^2 = x^3-1728, \]
and its endomorphism $\psi$ with kernel generated by $T = (0,24\sqrt{-3})$.
Besides, $-3k = 144^2$ is a square regardless of $\F_q$.
Thus, by \Cref{prop:indegree}, for every $(x,y) \in \S_{\Lambda}$, we have
\[ (\psi|_{\S_{\Lambda}})^{-1}(x,y) = 1 \iff y \not\in \F_q. \]

\begin{remark}
\label{rmk:selfLoops}
    By \Cref{lem:selfloops}, with our choice of $k$, the loops in $\fungraph{E(\F_{q^2})}{\psi}$ are
    \[ \0, \quad (12,0), \quad (-6+6\sqrt{-3},0), \quad (-6-6\sqrt{-3},0), \]
    which correspond to the loops $\textcolor{black}{\infty}, 1728$ in $\fungraph{\P^1(\F_q)}{\Hess}$. 
\end{remark}

\begin{remark} 
    We investigate the dynamics of $\Hess$ instead of $\Lambda$ because it potentially contains more information, as observed in \Cref{rmk:lambdainj}.
    More precisely, we have a cubic immersion
    \[ \fungraph{\P^1(\F_q)}{\Lambda} \xrightarrow[]{M_3 \circ \Phi_{-4}} \fungraph{\P^1(\F_q)}{\Hess}, \]
    which is an isomorphism if and only if $q \equiv 2 \bmod 3$ (see \Cref{sec:q2mod3}).
\end{remark}

Hessian graphs $\fungraph{\P^1(\F_q)}{\Hess}$ behave differently depending on $q \bmod 3$.
In \Cref{subsec:leavesloops} we give general results over finite fields, while in the following sections we examine the two cases separately.


\subsection{Indegrees} \label{subsec:leavesloops}

We prove that the leaves of Hessian graphs precisely correspond to curves with odd traces.
Moreover, concerning the covering of $\S_{\Hess}$ as in \Cref{cor:Spi4Hessian}, \textcolor{black}{these curves} always arise (via $\pi_{\Hess}$) from points belonging to (the isomorphic image of) cubic twists of $E(\F_q)$.

\begin{lemma} \label{lemma:trace2}
    Let $j \in \F_q^*$ and $E_{(j)}$ be any elliptic curve with that $j$-invariant.
    Then
    \[ 2 \ | \ \tr(E_{(j)}) \quad \iff \quad \Hess^{-1}(j) \neq \emptyset. 
    \]
\end{lemma}
\begin{proof}
    The equivalence follows for every $j \neq 1728$ by \Cref{prop:j2tors}.
    However, any $E_{(1728)}$ has at least one $2$-torsion point (hence, even order), and $\Hess^{-1}(1728) = \{1728, -13824\}$ by \Cref{lem:multRoots}, so the equivalence also holds for $j = 1728$.
\end{proof}

\begin{theorem} \label{thm:indeg1}
    Let $j \in \F_q \setminus \{0,1728\}$ and $u$ be a generator of $\F_q^* / (\F_q^*)^6$. Then
    \[ |\Hess^{-1}(j)| = 1 \quad \iff 
    \quad j-1728 \not\in (\F_q^*)^2 \quad \iff 
    \quad j \in \bigcup_{h \in \{1,3,5\}} \pi_{\Hess} \left( \phi_{u^{-\frac{h}{6}}}\big( E_{-6912u^h}(\F_q)\big) \right). \]
\end{theorem}
\begin{proof}
    Let $R = (x,y) \in \S = \S_{\Hess}$ be such that $\pi_{\Hess}(R) = x^3 = j$, hence ${y = \sqrt{j - 1728} \in \F_{q^2}}$.
    Since $j \not\in \{ 1728, \infty\}$, we have $[2]R \neq \0$ and $\psi(R) \neq R$ by \Cref{lem:selfloops}.
    From $j \neq 0$ we also get ${\psi(R) \neq \0}$.
    Thus, by applying \Cref{lem:ProjM3Pi}, we obtain $|\Hess^{-1}(j)| \geq |(\psi|_{\S})^{-1}(R)|$.
    We notice that $|(\psi|_{\S})^{-1}(R)|$ does not depend on the choice of $R \in \pi_{\Hess}^{-1}(j)$ by \Cref{prop:identifyancestors} and \textcolor{black}{by the fact that $\pi$ preserves periodicity (\Cref{lem:pipreservesdep}).
    By \Cref{prop:GenProjectingSk}}, for every other $R' \in \pi_{\Hess}^{-1}(j)$, we have $\pi_{\Hess}(\tau_{R'} \cap \mathcal{S}) =\pi_{\Hess}(\tau_{R} \cap \mathcal{S})$, 
    hence in particular $|\Hess^{-1}(j)| \leq |(\psi|_{\S})^{-1}(R)|$.
    %
    Thus, we proved that $|\Hess^{-1}(j)| = |(\psi|_{\S})^{-1}(R)|$.
    
    We now show that
    \begin{equation} \label{eq:iffj1728}
        |(\psi|_{\S})^{-1}(R)| = 1 \quad \iff \quad j-1728 \not\in (\F_q^*)^2.
    \end{equation} 
    Suppose first that $j-1728 \in (\F_q^*)^2$, i.e., $y \in \F_q^*$.
    If $(\psi|_{\S})^{-1}(R) = \emptyset$, there is nothing to prove.
    Given a preimage $(\alpha,\beta) \in \S$ of $R$, by definition of $\psi$ we have
    \[ y = -\beta \frac{\alpha^3+13824}{3\sqrt{-3}\alpha^3} \in \F_q^* \quad \implies \quad  \sqrt{-3}\beta \in \F_q^*. \]
    From the addition formulae of $E$, one can straightforwardly check that the first entry of $(\alpha,\beta) \pm T$ is
    \textcolor{black}{\[ \mp \frac{48 \sqrt{-3}\beta}{\alpha^2}, \]}
    whose cube is in $\F_q$, since $\alpha^3 \in \F_q^*$.
    Therefore, $(\alpha,\beta)\pm T \in (\psi|_{\S})^{-1}(R)$, proving that ${|(\psi|_{\S})^{-1}(R)| > 1}$.
    To prove the other implication in \cref{eq:iffj1728}, \textcolor{black}{let us assume that $y \notin \F_q^*$.}
    By \cite[§2]{dickson:critIrrFinFld}, there is a unique root $\alpha \in \F_{q^3}$ of the polynomial $g(z)=z^3+3xz^2-6912$, as in \Cref{prop:indegree}.
    Since $\psi$ is an endomorphism of $E(\F_{q^6})$, there is a unique $\beta \in \F_{q^6}$ such that $(\alpha,\beta) \in E(\F_{q^6})$ and $\psi(\alpha,\beta) = R$.
    We only need to prove that \textcolor{black}{$(\alpha,\beta) \in \S$, namely} $\alpha^3 \in \F_q$.
    \textcolor{black}{Since $\F_q \subseteq \F_{q^3}$ is determined by the fixed points of the $q$-Froebenius, for every} $\gamma \in \F_{q^3}$ we have
\begin{equation}
\label{eqn:cubeFq}
         \gamma^3 \in \F_q \quad \iff \quad \textcolor{black}{(\gamma^q)^3 = (\gamma^3)^q = \gamma ^3 \quad \iff} \quad  \exists \ i \in \{0,1,2\} \ : \ \gamma^q = \z^i \gamma.
\end{equation}
    Therefore, we have
    \[ g(\alpha)^q = \alpha^{3q}+3\z^ix\alpha^{2q}-6912 = (\z^{2i}\alpha^q)^{3}+3x (\z^{2i}\alpha^q)^{2}-6912=g(\z^{2i}\alpha^q), \]
    so that $\alpha$ and $\z^{2i}\alpha^q$ are both roots of $g$. Two cases can occur (depending on $q \bmod 3$):
    \begin{itemize}
        \item if $\z\in \F_q$, then $\z^{2i}\alpha^q\in \F_{q^3}$ and, since $\alpha$ was the unique root of $g$ in $\F_{q^3}$, then $\alpha^q =\z^i \alpha$. Therefore, by \cref{eqn:cubeFq}, we conclude $\alpha^3 \in \F_q$.
        \item if $\z \notin \F_q$, then \textcolor{black}{the cubing is invertible over $\F_q$, hence $t^3-x^3 = (t-x)(t-\z x)(t-\z^2 x) \in \F_q[t]$ has precisely one root in $\F_q$.
        This implies that $x \in \F_{q^2}$ and $g(z)$ has its coefficients over $\F_{q^2}$. Therefore, $\alpha \in \F_{q^2}\cap \F_{q^3}=\F_q$} by \cite[§2]{dickson:critIrrFinFld}.
    \end{itemize}
    In both cases $(\alpha, \beta)\in \S$, proving~\eqref{eq:iffj1728}.
    
    Finally, we prove the following equivalence:
    \begin{equation} \label{eq:lastequiv}
        y \notin \F_q^* \quad \iff \quad (x,y) \in \bigcup_{ \substack{i \in \{0,1,2\}\\h \in \{1,3,5\}}} \phi_{\z^iu^{-\frac{h}{6}}}\big( E_{-6912u^h}(\F_q)\big).
    \end{equation}
    \textcolor{black}{By \Cref{cor:Spi4Hessian}, there is} $h \in \{0,\dots,5\}$ and $i \in \{0,1,2\}$ such that $R \in \phi_{\z^iu^{-\frac{h}{6}}}\big( E_{-6912u^h}(\F_q)\big)$, i.e.,\ there is a point $(x',y') \in E_{-6912u^h}(\F_q)$ such that $(x,y)=(\z^{2i}u^{-\frac{h}{3}}x', u^{-\frac{h}{2}}y')$.
    Since $y'\in \F_q$ and $u$ generates $\F_q^* / (\F_q^*)^6$, clearly $y=u^{-\frac{h}{2}}y'$ lies in $\F_q$ if and only if $h$ is even \textcolor{black}{(regardless of $i$)}, which proves \textcolor{black}{\cref{eq:lastequiv}.
    This yields the last equivalence in the statement, since
    \[ \pi_{\Hess}^{-1}\left( \bigcup_{h \in \{1,3,5\}} \pi_{\Hess} \left( \phi_{u^{-\frac{h}{6}}}\big( E_{-6912u^h}(\F_q)\big) \right) \right) = \bigcup_{ \substack{i \in \{0,1,2\}\\h \in \{1,3,5\}}} \phi_{\z^iu^{-\frac{h}{6}}}\big( E_{-6912u^h}(\F_q)\big), \]
    which concludes the proof.}
\end{proof}


\subsection{\texorpdfstring{$q \equiv 1 \bmod{3}$}{q=1 mod 3}} \label{sec:q1mod3}

In this section, we consider Hessian graphs over finite fields $\F_q$ with $q \equiv 1 \bmod 3$.
In this case, we crucially have $\z \in \F_q$, hence $\sqrt{-3} \in \F_q$ and $\psi$ is defined over $\F_q$. 

\begin{lemma} \label{lem:OrdTwists}
    Let $q \equiv 1 \bmod 3$ and $u$ be a generator of $\F_q^* / ( \F_q^* )^6$. 
    For every $i \in \{0,\dots,5\}$, we have
    \[
        \begin{cases}
            |E^{(u^i)}(\F_q)| \equiv 0 \bmod{3} &\textnormal{if } 2 \ | \ i, \\
            |E^{(u^i)}(\F_q)| \equiv 1 \bmod{3} &\textnormal{otherwise}.
        \end{cases}
    \]
\end{lemma}
\begin{proof}
    By \Cref{lemma:twists}, the cubic twists of $E$ are defined, for $j \in \{0,1,2\}$, by 
    \[ E^{(u^{2j})} \ : \ y^2 = x^3 - 1728 u^{2j}. \]
    In particular, each $E^{(u^{2j})}$ has a $3$-torsion point $(0,24\sqrt{-3}u^{j})$, therefore $|E^{(u^{2j})}(\F_q)| \equiv 0 \bmod 3$, so $\tr(E^{(u^{2j})}) \equiv 2 \bmod 3$.
    The other twists are the proper quadratic twists of the above curves, hence their traces are $-2 \bmod 3$ by~\cite[Prop.\,13.1.10]{husemoller:ellipticCurves}, therefore their order modulo $3$ is $1$.
\end{proof}

\begin{lemma} \label{lem:Skcoverq1}
    Let $q \equiv 1 \bmod 3$.
    For every $u \in \F_q^* / ( \F_q^* )^6$, we have
    \[ \fungraph{ \phi_{u^{-\frac{1}{6}}}\big( E_{-6912 u}(\F_q)\big) }{\psi} \simeq \fungraph{ E_{-6912 u}(\F_q) }{\psi_{-6912 u}}. \]
\end{lemma}
\begin{proof}
    By \Cref{lem:comDiag}, the curve isomorphism $\phi_{u^{-\frac{1}{6}}} : E_{-6912 u} \to E$ defines an isomorphism of functional graphs
    \[ \fungraph{ E_{-6912 u} }{\psi_{-6912 u}} \simeq \fungraph{ E }{\psi}. \]
    Since $\sqrt{-3} \in \F_q$, then $-27 \in (\F_{q}^*)^6$, therefore the above isomorphism restricts to $\F_q$-rational points by \Cref{lem:Spartsclosed}.    
\end{proof}

\begin{theorem} \label{thm:structureq=1}
    Let $q \equiv 1 \bmod 3$.
    The Hessian graph $\fungraph{\P^1(\F_q)}{\Hess}$ is the union of six subgraphs, which can intersect only in $0,1728,\infty \in \P^1(\F_q)$.    
    For each of these subgraphs, there exists $m \in \NN$ such that its connected components consist of cycles, whose vertices are the roots of arborescences $\T_3^m$, with the following modifications:
    \begin{itemize}
        \item The arborescences $\T_3^m$ rooted in $\infty$ are pruned of one node at depth $1$ if $m \geq 1$, and of two additional nodes at depth $2$ if $m \geq 2$.
        \item The arborescences $\T_3^m$ rooted in $1728$ are pruned of one node at depth $1$ if $m \geq 1$.
    \end{itemize}
    Moreover, three of these subgraphs have $m=0$, two of them have $m=1$, and the last one has $m>1$.
\end{theorem}
\begin{proof}
    By \Cref{lem:Spartsclosed}, since $\sqrt{-3} \in \F_q$, the map $\psi$ is respects the covering of $\S_{\Hess}$ given by \Cref{cor:Spi4Hessian}.
    Therefore, we have
    \[ \fungraph{\S_{\Hess}}{\psi} = \bigcup_{u \in \K^*/(\K^*)^6} \fungraph{\phi_{u^{-\frac{1}{6}}}\big(E_{-6912u}(\F_q)\big)}{\psi}. \]
    Moreover, by \Cref{lem:Skcoverq1}, we have
    \[ \fungraph{\phi_{u^{-\frac{1}{6}}}\big(E_{-6912u}(\F_q)\big)}{\psi} \simeq \fungraph{ E_{-6912 u}(\F_q) }{\psi_{-6912 u}}. \]
    By \Cref{thm:structure}, 
    each $\fungraph{ E_{-6912 u}(\F_q) }{\psi_{-6912 u}}$ is made of cycles, whose vertices are roots of an arborescence $\T_3^m$ for some $m \in \NN$.
    By \Cref{lem:OrdTwists} and \Cref{prop:Ord3}, we necessarily have $m=0$ for three of them, and $m>0$ for the remaining ones.

    We now project this structure to the underlying Hessian graph.
    By \Cref{prop:projectingSk}, we know that 
    \[ \fungraph{\P^1(\F_q)}{\Hess} = \pi_{\Hess}\big( \fungraph{\S_{\Hess}}{\psi} \big) = \bigcup_{u \in \F^*/(\F^*)^6} \pi_{\Hess}\Big( \fungraph{\phi_{u^{-\frac{1}{6}}}\big(E_{-6912u}(\F_q)\big)}{\psi} \Big). \]
    Since $|\F_q^*/(\F_q^*)^6|=6$, the six claimed subgraphs are $\pi_{\Hess}\big( \fungraph{\phi_{u^{-\frac{1}{6}}}\big(E_{-6912u}(\F_q)\big)}{\psi} \big)$ for $u \in \F_q^*/(\F_q^*)^6$.
    They can intersect only in $0, 1728, \infty \in \P^1(\F_q)$ by \Cref{lem:disjointTwistsProjs}.
    Moreover, by \Cref{prop:projectingSk}, the projection $\pi_{\Hess}$ respects the structure of the arborescences $\T_3^m$, except for the identifications occurring over $2$-torsion points (which project to $1728$ and $\infty$) and the identification of $\pm T$ (which project to $0$).
    The structure around the special points is given by \Cref{lem:multRoots}.
    In particular, since the depth of $\tau_{\0} \subseteq E_{-6912u}(\F_q)$ determines $m$ for each subgraph (\Cref{thm:structure}-\ref{per}), this implies that $m>1$ for exactly one subgraph, namely the unique one containing $6912 \in \P^1(\F_q)$.
\end{proof}

Explicit examples of the six components appearing in the Hessian graphs over $\F_q$ with $q \equiv 1 \bmod 3$ are given in \Cref{App:examples} (\Cref{fig:hessian31,fig:hessian172}).


\subsection{\texorpdfstring{$q \equiv 2 \bmod{3}$}{q=2 mod 3}} \label{sec:q2mod3}

In this section, we investigate Hessian graphs over finite fields $\F_q$ with $q \equiv 2 \bmod 3$.
In this case, we know the group structure of $E$.
\begin{lemma} \label{lem:grpstruct}
    Let $q \equiv 2 \bmod 3$ and $E_{(0)}$ be an elliptic curve over $\F_q$ of $j$-invariant $0$. Then
    \[ E_{(0)}(\F_q) \simeq \ZZ/(q+1)\ZZ. \]
\end{lemma}
\begin{proof}
    Let $p = \mathsf{char}(\F_q)$.
    We know that $E_{(0)}(\F_p)$ is supersingular \cite[Prop.\,4.33]{washington:ellipticCurves}, hence its trace over $\F_p$ is $0$.
    This implies that its Frobenius endomorphism $\pi \colon (x,y) \mapsto (x^p,y^p)$ over $\F_p$ satisfies $\pi^2=-p$.
    If we denote by $t_r$ the trace of $E_{(0)}$ over $\F_{p^r}$, then the characteristic equation of the Frobenius endomorphism $\pi^{r}$ of $E_{(0)}$ over $\F_{p^r}$ gives
\[ 0 = \pi^{2r}-t_r \pi^{r}+p^{r} = (-p)^{r}-t_r \pi^{r} + p^{r}.
\]
Since $q$ must be an odd power of $p$, then the above equation implies $t_r = 0$.
Thus, the trace of $E_{(0)}$ over $\F_q$ is $0$ and the group structure is given by~\cite[Lem.\,4.8]{schoof:nonsingplanecubicfinitefld}.
\end{proof}

\begin{remark} \label{rmk:Senough}
    As observed in \Cref{rmk:lambdainj}, for $q \equiv 2 \bmod 3$ the cubing map $M_3$ is invertible. Then \Cref{prop:commuting} gives
\[ \fungraph{\P^1(\F_q)}{\Hess} \simeq \fungraph{\P^1(\F_q)}{\Lambda}. \]
    In particular, we can understand Hessian graphs by only employing elements in $\S_{\Lambda} \subseteq E(\F_{q^2})$.
\end{remark}



\begin{proposition} \label{prop:structS}
    Let $q \equiv 2 \bmod 3$. Then
    \[ \psi|_{E(\F_q)} : E(\F_q) \to \phi_{\sqrt{-3}} \big( E_{256}(\F_q) \big), \]
    is a well-defined group isomorphism. Conversely, the map
    \[ \psi|_{\phi_{\sqrt{-3}} ( E_{256}(\F_q) )} : \phi_{\sqrt{-3}} \big( E_{256}(\F_q) \big) \to E(\F_q), \]
    is a well-defined $3$-to-$1$ group homomorphism.
\end{proposition}
\begin{proof}
    The above maps are both well-defined by \Cref{lem:Spartsclosed}.
    They are restrictions of a group homomorphism, hence they are also homomorphisms.
    Since $T \in E(\F_{q^2}) \setminus E(\F_{q})$, we have $\ker\psi|_{E(\F_q)} = \{\0\}$, hence $\psi|_{E(\F_q)}$ is injective.
    Surjectivity follows by size comparison: we notice that~$E_{256}$ is the proper quadratic twist of $E$, since $E = E_{256}^{(-27)}$.
    Therefore, we have $\tr(E) = -\tr(E_{256})$, and they are both $0$ by \Cref{lem:grpstruct}, which implies $|E(\F_q)| = q+1 = |E_{256}(\F_q)| = |\phi_{\sqrt{-3}} \big( E_{256}(\F_q) \big)|$.
    
    By \Cref{lem:psi2}, we have
    \[ \psi|_{\phi_{\sqrt{-3}} ( E_{256}(\F_q) )}  \circ \psi|_{E(\F_q)}= [-3]|_{E(\F_q)}, \]
    and $[-3]|_{E(\F_q)}$ is a $3$-to-$1$ group homomorphism.
    Since we proved that $\psi|_{E(\F_q)}$ is a group isomorphism, then $\psi|_{\phi_{\sqrt{-3}} ( E_{256}(\F_q) )}$ needs to be $3$-to-$1$.
\end{proof}


\begin{corollary} \label{cor:structS}
    Let $q \equiv 2 \bmod 3$. Then we have
    \[ \S_{\Lambda} = E(\F_q) \cup \psi\big( E(\F_q) \big), \]
    and
     \[ E(\F_q) \cap \psi\big( E(\F_q) \big) = \{ (12,0), \0 \}. \]
\end{corollary}
\begin{proof}

By \Cref{cor:Spi4Hessian}, we know that $\S_{\Lambda}$ is covered by $E(\F_q)$ and the image of its quadratic twist, which is $\psi( E(\F_q) )$ by \Cref{prop:structS}.
They intersect in \textcolor{black}{points with zero coordinates by \Cref{prop:newdisjpartsSk}, but $T \notin E(\F_q)$ since $\sqrt{-3} \notin \F_q$. Thus, they intersect only in} $2$-torsion points, which are only $\{ (12,0), \0 \}$ by \Cref{rmk:selfLoops}.
\end{proof}

\begin{theorem} \label{thm:structureq=2}
    Let $q \equiv 2 \bmod 3$ and let $N,d \in \NN$ such that $q+1 = 3^dN$, with ${\gcd}(3,N) = 1$.
    The Hessian graph $\fungraph{\P^1(\F_q)}{\Hess}$ enjoys the following properties.
    \begin{enumerate}
        \item\label{i} There are $N$ periodic elements. Two of them ($1728$ and $\infty$) constitute loops and have indegree~$2$, while the others lie on cycles of even length, where elements of indegree $1$ and $3$ alternate.
        \item\label{ibis} The length of every cycle divides the length of a maximal cycle, which is
        \[ \begin{cases}
            \ord_{N}(-3) & \textnormal
            {if $ -1 \in \{ (-3)^n \bmod N \}_{n \in \NN}$}, \\
            2\,\ord_{N}(-3) & \textnormal{otherwise}.
        \end{cases} \]
        \item\label{ii} There are isomorphic arborescences attached to every periodic element of indegree greater than $1$. 
        All the leaves have depth $2d$ relative to periodic elements $\Per(\Hess)$, while every non-leaf $P$ satisfies 
        \[|\Hess^{-1}(P)| = \begin{cases}
            3 &\textnormal{if $d_{\Per(\Hess)}(P)$ is even}, \\ 
            1 &\textnormal{if $d_{\Per(\Hess)}(P)$ is odd}.
        \end{cases}\]
    \end{enumerate}
\end{theorem}
\begin{proof} Let us define
\[ \S_1 = E(\F_q), \quad \S_2 = \phi_{\sqrt{-3}} \big( E_{256}(\F_q) \big), \quad \S = \S_1 \cup \S_2.\]
By \Cref{prop:structS}, we have $\S_2 = \psi\big (E(\F_q) \big)$.
As observed in \Cref{rmk:Senough}, the graph epimorphism
\[ \pi_{\Hess} : \fungraph{\S}{\psi} \to \fungraph{\P^1(\K)}{\Hess} \]
completely determines the considered Hessian graph.

\ref{i}: By \Cref{prop:projectingSk}, the periodic points of $\fungraph{\P^1(\F_q)}{\Hess}$ are precisely the projections of periodic points of $\fungraph{\S}{\psi}$.
By \Cref{prop:structS}, they arise from the periodic points of $\fungraph{\S_1}{[-3]}$ and those of $\fungraph{E_{256}(\F_q)}{[-3]}$.
Since these curves are both cyclic of order $3^dN$ (\Cref{lem:grpstruct}), 
then both $\S_1$ and $\S_2$ contain precisely $N$ periodic points.
As $M_3$ is invertible, the projection $\pi_{\Hess}$ is $2$-to-$1$ over every $j$ different from $1728 = \pi_{\Hess}(12,0)$ and $\infty = \pi_{\Hess}(\0)$.
Since $\{(12,0),\0\} = \S_1 \cap \S_2$ by \Cref{cor:structS}, then the projected periodic points are
\[ 2 + \frac{N-2}{2} + \frac{N-2}{2} = N. \]
Furthermore, $1728$ and $\infty$ are the unique loops in $\fungraph{\P^1(\F_q)}{\Hess}$ by \Cref{lem:loopj}, and they have indegree~$2$ by \Cref{lem:multRoots}.
Given \Cref{prop:structS}, every $j \in \P^1(\K) \setminus \{ 1728, \infty\}$ jumps (under $\Hess$) from $\pi_{\Hess}(\S_1)$ to $\pi_{\Hess}(\S_2)$, and vice-versa.
Hence, the periodic points that are not loops lie on cycles of even length.
Moreover, by \Cref{thm:indeg1}, the indegree of periodic elements in $\pi_{\Hess}( \S_1 )$ is $3$, while the indegree of elements of $\pi_{\Hess}( \S_2 )$ is $1$.

\ref{ibis}: By part \ref{i}, every cycle in $\fungraph{\P^1(\F_q)}{\Hess}$ contains an element 
\begin{equation}\label{eq:PEq}
    \pi_{\Hess} (P), \quad \textnormal{with} \quad P \in \S_1.
\end{equation}
By \Cref{lem:grpstruct}, every maximal cycle in $\fungraph{\P^1(\F_q)}{\Hess}$ contains elements $\pi_{\Hess}(P)$ as in \cref{eq:PEq} of order~$N$.
Since $\psi^{(2)}(P) = [-3]P$ by \Cref{lem:psi2}, then such a $P$ has order $2 \,\ord_{N}(-3)$ in $\fungraph{\S}{\psi}$.
This order is halved after $\pi_{\Hess}$ if and only if $P$ and $-P$ belong to the same cycle in $\fungraph{\S}{\psi}$, namely if and only if $-1 \in \{ (-3)^n \bmod N \}_{n \in \NN}$.
Furthermore, since $P$ generates the subgroup of order $N$ in $\S_1$, every periodic element of type \eqref{eq:PEq} is equal to $[m]P$ for some $m \in \NN$.
Since $\psi$ commutes with the scalar multiplication $[m] \in \End(E)$, the cycle containing $[m]P$ arises from the multiplication of the cycle containing $P$ by the scalar $m$.
This operation folds the maximal cycle ${\gcd}(m,N)$ times, hence the length of a maximal cycle is divisible by the length of every other cycle.

\ref{ii}: By \Cref{prop:structS}, $\pi_{\Hess}$ induces a functional graph embedding
\[ \fungraph{ \S_1 }{[-3]} \hookrightarrow \fungraph{\P^1(\F_q)}{\Hess^{(2)}}. \]
This embedding completely determines 
$\fungraph{\P^1(\F_q)}{\Hess}$ because every element in
$\pi_{\Hess}(\S_2) \setminus \pi_{\Hess}(\S_1)$ has indegree~$1$ by \Cref{thm:indeg1}.
In particular, $\fungraph{\P^1(\F_q)}{\Hess}$ can be obtained from $\fungraph{ \S_1 }{[-3]}$ by replacing every edge $P \to [-3]P$ such that $P \neq -P$ with two edges
\[ \pi_{\Hess}(P) \xrightarrow[]{\Hess} \pi_{\Hess}\circ\psi(P) \xrightarrow[]{\Hess} \pi_{\Hess}([-3]P). \]
Since all non-periodic elements of $\fungraph{ \S_1 }{[-3]}$ are organized in isomorphic arborescences by \Cref{thm:structure}, the same holds for non-periodic elements of $\fungraph{\P^1(\F_q)}{\Hess}$ by \Cref{prop:projectingSk}.
Moreover, the depth of leaves in $\fungraph{ \S_1 }{[-3]}$ is the maximal $e \in \NN$ such that $3^e$ divides the exponent of $E(\F_q) \simeq \ZZ/(3^dN)\ZZ$, i.e., it equals $d$.
Therefore the \textcolor{black}{depth} of all leaves in $\fungraph{\P^1(\F_q)}{\Hess}$ is $2d$.
Finally, the root of every non-trivial arborescence of $\fungraph{\P^1(\F_q)}{\Hess}$ is a periodic element of indegree greater than $1$, hence it belongs to $\pi_{\Hess} (\S_1)$ by \Cref{thm:indeg1}. This implies that elements of even depth have indegree $3$, except for leaves, whose indegree is $0$, and $0, 1728 \in \P^1(\F_q)$, which have indegree $2$. Conversely, all the elements of odd depth belong to $\pi_{\Hess} (\S_2) \setminus \{0,1728\}$, therefore they have indegree $1$ by \Cref{thm:indeg1}.
\end{proof}

\begin{remark}
    Differently from the case $q \equiv 1 \bmod 3$, the arborescences prescribed by \Cref{thm:structureq=2} are not isomorphic to $\T_3^m$ (e.g.\ see \Cref{fig:hessian31psi,fig:hessian89} in \Cref{App:examples}).
    However, their structure is completely determined by the conditions \ref{i} and \ref{ii} of \Cref{thm:structureq=2}.
\end{remark}


\section{A glimpse into cryptography} \label{sec:crypto}

The regular structure of Hessian graphs over finite fields, as described in \Cref{sec:HessianFF}, might remind the reader of \emph{(oriented) $\ell$-isogeny graphs}, which are widely studied in the context of post-quantum cryptography. We refer to~\cite{kohel:endomRings,sutherland:isogVolc} 
for the study of their structure and to~\cite{castryckEtAl:CSIDH,scallop} 
for their use in cryptography.
Finding a cryptographic application for Hessian graphs and understanding their relation with isogeny graphs -- if any -- is beyond the scope of this paper. However, in this section, we highlight some consequences of our previous results that might motivate future research in this direction.


\subsection{Efficient computation of the iterated Hessian}

Isogeny-based protocols are typically performed over finite fields of cryptographic size, i.e., $\F_p$ or $\F_{p^2}$ where $p$ is a prime of at least $256$ bits.
While evaluating many iterations of a rational map over such fields is computationally hard in general, this task has a polynomial complexity for the Hessian transformation because of its Lattès structure.
\begin{proposition}
\label{prop:iterHess}
    Let $N\in \NN$ and $j \in \F_q$. Then $\Hess^{(N)}(j)$ can be computed in polynomial time in the sizes of $q$ and $N$. 
\end{proposition}
\begin{proof}
    By \Cref{prop:commuting} and \Cref{thm:projectphi3}, one can compute $\Hess^{(N)}(j)$ by considering a preimage
    \[ P = \pi_{\Hess}^{-1}(j) \in \S_{\Hess} \subseteq E(\F_{q^6}), \]
   computing $\psi^{(N)}(P)$, and projecting back to $\P^1(\F_q)$ via $\pi_{\Hess}$.
    All these operations can be performed in polynomial time.
    Specifically, $\pi_\Hess^{-1}(j)$ only requires finding a cubic root in $\F_{q^3}$ and a square root in $\F_{q^6}$, which can be done in polynomial time~\cite{adlemanEtAl:roots77}.
    Moreover, up to computing an additional iteration of~$\Hess$, we can assume $N$ to be even.
    By \Cref{lem:psi2}, we have $\psi^{(N)}(P)=\left[-3^{N/2}\bmod |E(\F_{q^6})|\right]P$, which can be computed in polynomial time, both in $N$ and $q$.
    In fact, $|E(\F_{q^6})| \sim q^6$ can be found via Schoof's algorithm~\cite{schoof:countingPtsEC}, \textcolor{black}{which can then be used to reduce $-3^{N/2}$ to an integer of size $\sim q^6$}.
    Such a scalar multiplication of $P$ can be performed via double-and-add algorithms (see e.g.\ \cite[§XI.1]{silv:arithEll}).
\end{proof}

\begin{remark}
    \Cref{cor:Spi4Hessian} and \Cref{lem:comDiag} can be leveraged to avoid computations on extensions of $\F_q$, by replacing $P$ with the corresponding $\F_q$-rational point on a suitable twist of $E$. 
    This efficient Hessian computation is detailed in Appendix~\ref{App:iterHess}, \Cref{alg:iterHess}.
\end{remark}

\textcolor{black}{\begin{remark}
    The same ideas in \Cref{prop:iterHess} can be applied to efficiently compute the $N$-th iteration of any Lattès map $\phi$ descending from $\psi \in \End(E)$ on a given $j \in \F_q$, as $\phi^{(N)}(j) = \pi \circ \psi^{(N)} \circ \pi^{-1}(j)$.
    This method requires:
    \begin{itemize}
        \item computing $P=\pi^{-1}(j)$, i.e., finding a root of a degree-$\deg(\pi)$ univariate polynomial in $\F_q$,
        \item finding the minimal polynomial $f(x) \in \ZZ[x]$ of $\psi$, which boils down to computing its trace (see, e.g., \cite[Thm.\,81]{kohel:endomRings}),
        \item finding $\lambda,\mu \in \ZZ$ such that $\lambda x + \mu \equiv x^N \bmod f(x)$,
        \item evaluating $(\lambda \psi + \mu)(P)$ (see, e.g., \cite[Thm.\,3]{robert:efficientRep}) and projecting it via $\pi$.
    \end{itemize}
    In the Hessian case covered by \Cref{prop:iterHess}, $f(x) = x^2+3$ and $x^N \equiv [-3]^{\frac{N}{2}}$ or $x^N \equiv [-3]^{\frac{N-1}{2}} x$, depending on $N$ being even or odd, respectively.
    Similar methods may be applied to other finite quotients of affine maps (Chebyshev and power maps \cites{milnor2006lattes}).
\end{remark}}

\subsection{Supersingular points in the Hessian graph} \label{subsec:supersing}

An open problem in isogeny-based cryptography is to produce a (random) supersingular $j$-invariant in such a way that its endomorphism ring remains unknown~\cite{booherEtAl:failToHash,SRS}. This motivates the question of whether supersingular $j$-invariants \textcolor{black}{exhibit a structured distribution in the Hessian graphs}.
As we will now see, there is in fact some non-trivial necessary condition for a connected component of the Hessian graph to be \emph{supersingular}, i.e., to contain at least one supersingular $j$-invariant.
\begin{lemma}
\label{lem:supersingCubSquare}
    Let $j$ be a supersingular $j$-invariant. Then $j \in (\F_{p^2})^3$ and $j-1728 \in (\F_{p^2})^2$.
\end{lemma}
\begin{proof}
     The fact that $j$ is a cube in $\F_{p^2}$ follows from CM-theory~\cite[Thm.\,2.4]{Morton2011}.
     By~\cite[§3]{adjEtAl:isogenyGraphsSupersing}, there is an elliptic curve $E_{(j)}: y^2 = x^3+Ax + B$ defined over $\F_{p^2}$ with that $j$-invariant and such that $\tr(E_{(j)}) = \pm 2p$.
     By~\cite[Lem.\,4.8]{schoof:nonsingplanecubicfinitefld}, its affine $2$-torsion points $\{(x_i,0)\}_{i \in \{1,2,3\}}$ are defined over $\F_{p^2}$.
     By~\cite[206]{cox:primesOfFormX2+nY2}, the discriminant $\Delta$ of $E_{(j)}$ is
     \[\Delta=16\prod_{1 \leq i<k \leq 3}(x_i-x_k)^2,\]
    which is a square in $\F_{p^2}$.
    A straightforward computation gives
    $j-1728=\frac{2^{10}3^6 B^2}{\Delta} \in (\F_{p^2})^2$.
\end{proof}
\begin{proposition} \label{prop:locatess}
    The supersingular vertices of $\fungraph{\P^1(\overline{\F_p})}{\Hess}$ all belong to the subgraph $\pi_{\Hess}(\fungraph{E(\F_{p^2})}{\psi})$.
\end{proposition}
\begin{proof}
By \Cref{lem:supersingCubSquare}, $j$ has a cubic root $x\in\F_{p^2}$ and $j-1728$ has a square root $y\in\F_{p^2}$.
Therefore $(x,y) \in E(\F_{p^2})$ and $j=\pi_{\Hess}([x\colon y \colon 1])$.
Since $\sqrt{-3} \in \F_{p^2}$, then $\psi$ is closed on $E(\F_{p^2})$, hence it projects to a subgraph of $\fungraph{\P^1(\overline{\F_p})}{\Hess}$.
\end{proof}

\begin{corollary}
    Every supersingular vertex $j \in \F_{p^2} \setminus \{0, 1728\}$ has indegree $3$ in $\fungraph{\P^1({\F_{p^2}})}{\Hess}$.
\end{corollary}
\begin{proof}
    By \Cref{thm:indeg1} and \Cref{prop:locatess}, such a $j$ cannot have indegree $1$.
    Thus, by \Cref{lem:multRoots}, it is either $0$ or $3$.
    Let $E_{(j)}$ be any elliptic curve with $j$-invariant $j$.
    Since $\tr(E_{(j)})$ is even \cite[§3]{adjEtAl:isogenyGraphsSupersing}, then $E_{(j)}$ has an order-$2$ point. Therefore, by \Cref{prop:j2tors}, the indegree of $j$ must be $3$.
\end{proof}

Finally, since the Hessian preserves the $3$-torsion of elliptic curves, all the curves in a supersingular component share the same trace modulo $3$.

\begin{lemma}
\label{prop:traceMod3}
    Let $E$ be an elliptic curve over $\F_q$ with $j(E) \neq 0$. Then
    \[ \tr(E) \equiv \tr\big(\Hess(E)\big) \bmod 3. \]
\end{lemma}
\begin{proof}
    By Remark~\ref{rem:flexPoints}, $E$ and $\Hess(E)$ have the same $3$-torsion points.
    Therefore, their Frobenius endomorphisms coincide over $E[3]$, meaning that they are represented (as automorphisms of $E[3]$) by the same matrix over $\mathrm{GL}_2(\ZZ/3\ZZ)$.
    In particular, the corresponding characteristic polynomials are equal.
\end{proof}
\begin{proposition} \label{prop:supersingTrace}
    Let $j \in \F_p^*$ (resp.\ $j \in \F_{p^2}^*$) with $j \neq 1728$ be a vertex in a supersingular component of $\fungraph{\P^1(\F_p)}{\Hess}$ (resp.\ $\fungraph{\P^1(\F_{p^2})}{\Hess}$), and let $E_{(j)}$ be a curve defined over $\F_p$ (resp.\ $\F_{p^2}$) with that $j$-invariant.
    Then $\tr(E_{(j)}) \equiv 0 \bmod 3$ (resp.\ $\tr(E_{(j)}) \not\equiv 0 \bmod 3$).
\end{proposition}
The above proposition follows by \Cref{prop:traceMod3}, but its complete proof requires a finer analysis of the local behavior of the Hessian action around $0$ and $1728$.
We defer these results and the proof of \Cref{prop:supersingTrace} to Appendix~\ref{App:twists}.

\section{Conclusion}
\label{sec:conclusion}
\textcolor{black}{In this paper, we provided a novel interpretation of the Hessian transformation for plane projective cubics as a rigid Lattès map fitting into a reduced diagram.
This result allows for a description of the global symmetries of Hessian graphs, which had eluded all previous local analysis.}

\textcolor{black}{The tools presented in \Cref{Sec:Prel,sec:HessianIsLattes,sec:ConnComp,sec:Sk,sec:Functionalpsik} are proven over any field $\K$, provided that $\mathsf{char}(\K) \neq 2,3$, while the results of \Cref{sec:HessianFF,sec:crypto} are finer but require $\K$ to be finite.}
\textcolor{black}{It is conceivable to expect Hessian graphs over infinite fields to be wilder, as they may involve connected components of types \ref{line} or \ref{sline} from \Cref{thm:structure}.
Nonetheless, the reduced Lattès structure still poses strong constraints on such infinite dynamical systems, which would deserve further investigation.}

\textcolor{black}{The tools developed in \Cref{sec:ConnComp,sec:Sk} are quite general and may be readily applied to study reduced Lattès diagrams coming from endomorphisms $\psi$ of prime degree, e.g., arising from elliptic curves with CM by $\sqrt{-\ell}$, for some prime $\ell$.
We do not foresee major obstructions as long as $\mathsf{char}(\K)$ is coprime to~$\ell$.
To widen the applicability of the proposed methods, we envision two possible research directions: either factoring large degree projections $\pi$ through reduced diagrams, or extending the classification of the dynamics of group endomorphisms to those with composite kernel.}

The arithmetic properties of $E_k$ could also provide \textcolor{black}{heuristics about special families of curves}.
For instance, inspecting the distribution of supersingular curves on Hessian graphs (as in \Cref{subsec:supersing}) might be relevant for cryptographic applications.

Finally, since the Hesse pencil can be viewed as a model of the level-$3$ modular curve $X(3)$ (see e.g.\ \cite[§2]{artebaniDolgachev:hassePencil}), one can aim at rephrasing our approach in the language of modular curves, and consider the Hessian transformation on higher-level \textcolor{black}{genus-$1$ moduli}, such as $X(6)$.

\section*{Acknowledgements}

We want to express our gratitude to Luca De Feo, who pointed us to functional graphs of group endomorphisms, to Wouter Castryck for his valuable comments on a preliminary version of this article, to Marc Houben for introducing us to the framework of Lattès maps, to Andrea Ferraguti for the proof of \Cref{Sec3:prop-Lattes-C}, and to Lea Terracini for helpful suggestions regarding dynamical systems over $\mathbb{Q}^{\mathrm{ab}}$.
We also wish to thank the anonymous reviewer for their valuable comments, which helped broaden the scope of the paper to reduced Lattès diagrams.
Finally, we thank Massimiliano Sala for seminal discussions on the topic,  
Edoardo Ballico for his insights on Hessian varieties, Giulio Codogni, Guido Lido, Sachi Hashimoto, and Samuel Le Fourn for fruitful discussions on modular curves.

\section*{Funding}

DT has been supported by the Research Foundation - Flanders (FWO: 12ZZC23N) and by the BOF project C16/21/002 by the Internal Funds KU Leuven.
FP acknowledges support from Ripple’s University Blockchain Research Initiative.
All authors thank the INdAM group GNSAGA for support, and CIRM-FBK (Trento) for the hospitality through the \emph{Research in Pairs} program.

\printbibliography

\appendix
\section{Isomorphism classes in the Hesse pencil}
\label{App:IsoClasses}

We define the \emph{Hessian group} $G_{216}$ as the subgroup of $\Aut\big(\P^2(\kbar)\big)$ fixing the set $\{p_0, \dots, p_8\}$.
As shown in~\cite[§4]{artebaniDolgachev:hassePencil}, it is generated by
\begin{align*}
    g_1\colon [x:y:z] &\mapsto [y : z : x],\\  
    g_2\colon [x:y:z] &\mapsto [x : \z y : \z^2 z],\\
    g_3\colon [x:y:z] &\mapsto [x + y + z : x + \z y + \z^2 z : x + \z^2 y + \z z],\\ 
    g_4\colon [x:y:z] &\mapsto [ x: \z y: \z z].
\end{align*}
One can check that the action of $G_{216}$ respects the Hesse pencil, namely for every $g \in G_{216}$ there is $\hat{g} \in \Aut\big(\P^1(\kbar)\big)$ such that
\[ F_{\lambda} \big( g([x:y:z]) \big) = F_{\hat{g}(\lambda)} ( [x:y:z] ). \]
In particular, one can verify that
\begin{align*}
\hat{g}_1&=\hat{g}_2 = \mathrm{Id},& \hat{g}_3\colon &[\lambda:1] \mapsto \left[ 3(6-\lambda):\lambda + 3\right], & \hat{g}_4\colon [\lambda:1] &\mapsto [\z^2\lambda:1]. \\
& & &[1:0] \mapsto [-3:1]  & [1:0] &\mapsto [1:0]
\end{align*}
The relations $\hat{g}_3^2=\hat{g}_4^3=(\hat{g}_4 \circ \hat{g}_3)^3=\mathrm{Id}$ show that $\widehat{G_{216}}$ is isomorphic to the alternating group $A_4$ (see also~\cite[§2]{cataneseSernesi2024:hesseModSpace}).
We are now ready to prove \Cref{Sect2:Prop-IsoClasses}.

\begin{proof}[Proof of \Cref{Sect2:Prop-IsoClasses}]
One can directly check by \cref{eqn:jInvPencil} that $\mathcal{J}^{-1}(\infty) = \{ -3, -3\z, -3\z^2, \infty \}$, and
\begin{align*}
    j(\lambda) &=-\frac{ \lambda^3  (\lambda - 6)^3 (\lambda - 6 \z)^3  (\lambda -6\z^2)^3}{(\lambda + 3)^3 (\lambda + 3\z)^3 (\lambda +3\z^2)^3},\\
    j(\lambda)-1728 &=- \prod_{i=0}^2 \frac{ \big(\lambda +3\z^i(1+\sqrt{3}) \big)^{2} \big(\lambda +3\z^i(1-\sqrt{3}) \big)^{2} }{(\lambda + 3\z^i)^3}.
\end{align*}
Therefore, we have $\mathcal{J}^{-1}(0) = \{0, 6,6\z,6\z^2\}$, while $\mathcal{J}^{-1}(1728)
$ is non-empty if and only if $\sqrt{3} \in \K$. In this case, $j(\lambda) = 1728$ admits precisely two solutions if $\z \not\in \K$, or six distinct solutions otherwise.

Since the $j$-invariant characterizes the curves up to $\Aut\big(\P^1(\kbar)\big)$, then the fibers of $\mathcal{J}$ are closed under the action of $\widehat{G_{216}}$, whose fixed points (over $\kbar$) are the solutions of the following equations:
\begin{align*}
\mathrm{Id}(\lambda)-\lambda &=0,& \hat{g}_3(\lambda)-\lambda &=- \frac{ \lambda^{2} + 6 \lambda - 18}{\lambda + 3},\\
    \hat{g}_4 \circ \hat{g}_3(\lambda)-\lambda &=- \frac{ (\lambda +3\z^2)  (\lambda-6\z^2)}{\lambda + 3},&
     \hat{g}_4^2 \circ \hat{g}_3(\lambda)-\lambda &=- \frac{ (\lambda + 3 \z)  (\lambda - 6 \z)}{\lambda + 3},\\
\hat{g}_3 \circ \hat{g}_4 \circ \hat{g}_3(\lambda)-\lambda &=- \frac{ (\lambda - 6)  (\lambda + 3)}{\lambda + 3 \z},&
    \hat{g}_4(\lambda)-\lambda &=\left(\z - 1\right)  \lambda,\\
\hat{g}_4^2(\lambda)-\lambda &=-\left(\z + 2\right) \lambda,&   \hat{g}_3 \circ \hat{g}_4(\lambda)-\lambda &=- \frac{ (\lambda + 3 \z)  (\lambda - 6 \z)}{\lambda +3\z^2},\\
    \hat{g}_3 \circ \hat{g}_4^2(\lambda)-\lambda &=- \frac{ (\lambda +3\z^2)  (\lambda-6\z^2)}{\lambda + 3 \z},&  
\hat{g}_4 \circ \hat{g}_3 \circ \hat{g}_4(\lambda)-\lambda &=- \frac{ (\lambda - 6)  (\lambda + 3)}{\lambda +3\z^2},\\
\hat{g}_4^2 \circ \hat{g}_3 \circ \hat{g}_4(\lambda)-\lambda &=- \frac{ (\lambda^{2} +6\z^2 \lambda - 18 \z)}{\lambda +3\z^2},&
    \hat{g}_4 \circ \hat{g}_3 \circ \hat{g}_4^2(\lambda)-\lambda &=- \frac{ \lambda^{2} + 6 \z \lambda -18\z^2}{\lambda + 3 \z}.
\end{align*}
These solutions all belong to $\mathcal{J}^{-1}(\{0,1728, \infty\})$.
Hence, for $j \in \K^*\setminus\{1728\}$, the action of $\widehat{G_{216}}$ on $|\mathcal{J}^{-1}(j)|$ is free.
In particular, since by \cref{eqn:jInvPencil} we have
\[ |\mathcal{J}^{-1}(j)| \leq 12 = |\widehat{G_{216}}|, \]
we conclude that this action is also transitive, hence regular.
Therefore, if $\z \in \K$ and $\lambda \in \mathcal{J}^{-1}(j)$, then
\[ \mathcal{J}^{-1}(j) = \{\hat{g}(\lambda)\}_{g \in G_{216}}. \]
Otherwise, if $\z \not\in \K$, one can easily check that
\[ \mathcal{J}^{-1}(j) = \{ \lambda, \hat{g_3}(\lambda) \}, \]
proving part \ref{propjlambdaii}.
%
%
\end{proof}

\begin{remark}
    We note that the above proof also shows that $\widehat{G_{216}}$ acts transitively on the fibers of $\mathcal{J}$.
    In fact, we observed that the action is even free on $\mathcal{J}^{-1}(j)$ for every $j \in \K^* \setminus \{1728\}$, while there is a unique orbit also for $j \in \{0,1728,\infty\}$, since
    \[ \hat{g}_3(-3) = \infty, \quad \hat{g}_3(0) = 6, \quad \textnormal{and} \quad \hat{g}_4^{-1} \circ \hat{g}_3 \circ \hat{g}_4 (-3-\sqrt{3}) = -3+\sqrt{3}. \]
\end{remark}

\section{Hessian of twists}
\label{App:twists}
    
Passing from curves to $\kbar$-isomorphisms classes of curves does lose some information.
A finer version of the Hessian graph can be obtained considering the Hessian of $\K$-isomorphism classes -- rather than $\kbar$-isomorphism classes -- of elliptic curves.
We will call this graph the  \emph{Hessian graph with twists}.

\begin{proposition} 
    \label{prop:hessianIsom}
Let $E\colon y^2 = x^3 + Ax + B$ be an elliptic curve over $\K$ with $j(E) \neq 0$.
Let $D \in \K^*$ and define $E^{(D)}$ as in \Cref{lemma:twists}.
We denote by $E'$ the elliptic curve $\K$-isomorphic to $\Hess(E)$ and defined by \cref{eqn:hessianSWForm}, and define ${E'}^{(D)}$ accordingly. Then $\Hess(E^{(D)})$ is $\K$-isomorphic to
\begin{align}
\label{eqn:hessianTwists}
    \begin{cases}
       {E'}^{(D^{-1})} &\quad \text{if $j(E')\notin\{0,1728\}$ or $j(E')=1728=j(E)$},\\
       {E'}^{(D^{-2})}&\quad \text{if $j(E')= 1728\neq j(E)$},\\
        {E'}^{(D^{-3})}&\quad \text{if $j(E')=0$}.
    \end{cases}
\end{align}
In particular, $\Hess$ descends to a map $\Twist(E) \rightarrow \Twist(E')$, which corresponds to
\[ \begin{cases} \K^*/(\K^*)^2 \to \K^*/(\K^*)^2, \quad D \mapsto D^{-1}, &\textnormal{if } j(E') \not\in\{0,1728\}, \\
\K^*/(\K^*)^4 \to \K^*/(\K^*)^4, \quad D \mapsto D^{-1}, &\textnormal{if } j(E') = 1728 = j(E), \\
\K^*/(\K^*)^2 \to \K^*/(\K^*)^4, \quad D \mapsto D^{-2}, &\textnormal{if } j(E') = 1728 \neq j(E),\\
\K^*/(\K^*)^2 \to \K^*/(\K^*)^6, \quad D \mapsto D^{-3}, &\textnormal{if } j(E') = 0.
\end{cases} \]
\end{proposition}

\begin{proof}
Let $E' \colon y^2=x^3+A'x+B'$.
If $j(E)\notin \{0,1728\}$, then from \cref{eqn:hessianSWForm} we get
 \begin{equation*}
        \Hess(E^{(D)}) \colon y^2=x^3 + \frac{1}{D^2}A' x + \frac{1}{D^3}B'=\begin{cases}
           E'^{(D^{-1})} &\quad \text{if $j(E')\notin \{0,1728\}$},\\
           E'^{(D^{-2})} &\quad \text{if $j(E')=1728$},\\
           E'^{(D^{-3})} &\quad \text{if $j(E')=0$}.
        \end{cases}
\end{equation*}
Similarly, if $j(E)=1728$, \cref{eqn:hessianSWForm} yields
\begin{equation*}
    \Hess(E^{(D)}) \colon y^2=x^3 +\frac{1}{D}A'x = {E'}^{(D^{-1})}.
\end{equation*}
In this case, by \Cref{lem:multRoots} we have $j(E')=1728$, hence \cref{eqn:hessianTwists} covers every possible case.

Finally, the last statement follows by considering the Hessian transformation $\Twist(E) \to \Twist(E')$ under the canonical isomorphism given in \cite[Prop.\,X.5.4]{silv:arithEll}.
\end{proof}
Since the inversion is an isomorphism of $\K^*$, then $\Hess$ induces almost always an isomorphism between $\Twist(E)$ and $\Twist\big(\Hess(E)\big)$ by \Cref{prop:hessianIsom}.
As a consequence, when $j\big(\Hess(E)\big) \notin \{0,1728\}$, the proper twist of $\Hess(E)$ arises as the Hessian of the proper twist of $E$.
This suggests that the Hessian graph with twists is essentially a \lq doubled\rq\ Hessian graph. However, the cases in which extra proper twists might appear need some special care.

\begin{proposition}
\label{prop:specialTwists}
The Hessian graph with twists over $\F_q$, around curves of $j$-invariant $0$ and $1728$, is locally described by Tables~\ref{tab:0} and~\ref{tab:1728}.
\begin{table}[ht]
    \centering
    \begin{tabular}
    {@{}Sc|Sc@{}}
        \hline
        $q \equiv 1 \bmod 3$ & $q \equiv 2 \bmod 3$\\ \hline
       \begin{tikzpicture}
    \node[circle,draw,fill=gray!100,scale=0.5] (B) at (-0.25,0.17) {};
    \node[circle,draw,fill=gray!40,scale=0.5] (A) at (0.25,0.17) {};
    \node[circle,scale=0.5] (C) at (0,-0.5) {$\infty$};
     \node[circle,draw,scale=0.5] (D) at (-0.75,-0.4) {};
      \node[circle,draw,scale=0.5] (E) at (0.75,-0.4) {};
       \node[circle,draw,scale=0.5] (F) at  (-0.58,-0.05) {};
        \node[circle,draw,scale=0.5] (G) at (0.58,-0.05) {};

    \draw[-latex] (A) -- (C);
    \draw[-latex] (B) -- (C);

    \draw[-latex] (-0.46,0.75) -- (B);
    \draw[-latex] (0.46,0.75) -- (A);

    \draw[-latex] (D) -- (C);
    \draw[-latex] (F) -- (C);
    \draw[-latex] (G) -- (C);
    \draw[-latex] (E) -- (C);
    \draw[-latex] (C) to[out=300,in=250,loop] node[] {} ();
 \end{tikzpicture}
           
& 
        \begin{tikzpicture}
    \node[circle,draw,fill=gray!100,scale=0.5] (B) at (-0.25,0.17) {};
    \node[circle,draw,fill=gray!40,scale=0.5] (A) at (0.25,0.17) {};
    \node[circle,scale=0.5] (C) at (0,-0.5) {$\infty$};
      \phantom{\node (D) at (0,0.8) {$\infty$};}

    \draw[-latex] (A) -- (C);
    \draw[-latex] (B) -- (C);

    \draw[-latex] (-0.46,0.75) -- (B);
    \draw[-latex] (0.46,0.75) -- (A);

    \draw[-latex] (C) to[out=300,in=250,loop] node[] {} ();
 \end{tikzpicture} \\
        \hline
    \end{tabular}
    \caption{Vertices of $j$-invariant $0$ in the Hessian graph with twists over $\F_q$.
    The darker vertex represents the curve of equation $y^2=x^3+1$, while the gray vertex identifies its proper quadratic twist.}
    \label{tab:0}
\end{table}
\begin{table}[ht]
    \centering
    \begin{tabular}{@{}Sr|Sc|Sc|Sc@{}}
        \hline
        & $-3 \in (\F_q^*)^4$ &  $-3 \in (\F_q^*)^2\setminus  (\F_q^*)^4$ & $-3 \notin (\F_q^*)^2$ \\
        \hline
        \begin{tikzpicture}
            \phantom{\node[]  at (0,1) {$q \equiv 3 \bmod 4$};}
            \node[]  at (0,0.5) {$-1 \notin (\F_q^*)^2$};
            \phantom{\node[draw] at (0,0){};}
        \end{tikzpicture} & 
        \begin{tikzpicture}
            \node[circle,draw,fill=gray!40,scale=0.5] (A) at (1,0) {};
            \node[circle,draw,fill=gray!100,scale=0.5] (B) at (0,0) {};
            \draw[-latex] (A) to[out=310,in=240,loop] node[] {} ();
            \draw[-latex] (B) to[out=310,in=240,loop] node[] {} ();
            \draw[-latex] (B) ++(0.25,0.5) -- (B);
            \draw[-latex] (B) ++(-0.25,0.5) -- (B);
        \end{tikzpicture} & 
         \begin{tikzpicture}
            \node[]  at (0,0.38) {$(\F_q^*)^2 = (\F_q^*)^4$};
            \phantom{\node[]  at (0,0) {(never occurs)};}
        \end{tikzpicture}
        &
        \begin{tikzpicture}
            \node[circle,draw,fill=gray!100,scale=0.5] (A) at (0,0) {};
            \node[circle,draw,fill=gray!40,scale=0.5] (B) at (1,0) {};
            \draw[-latex] (A) to[out=45,in=135] (B);
            \draw[-latex] (B) to[out=-135,in=-45] (A);
            \draw[-latex] (A) ++(0.25,0.5) -- (A);
            \draw[-latex] (A) ++(-0.25,0.5) -- (A);
            \phantom{ \draw[-latex] (A) to[out=310,in=240,loop] node[] {} ();}
        \end{tikzpicture}\\
        \hline
        \begin{tikzpicture}
            \phantom{\node[] at (0,1) {$q \equiv 1 \bmod 4$};}
            \node[]  at (0,0.42) {$-1 \in (\F_q^*)^2$};
            \phantom{\node[] at (0,0) {R2};}
        \end{tikzpicture} & 
        \begin{tikzpicture}
            \node[circle,draw,fill=gray!100,scale=0.5] (A) at (0,0) {};
            \node[circle,draw,fill=gray!40,scale=0.5] (B) at (1,0) {};
            \draw[-latex] (A) to[out=310,in=240,loop] node[] {} ();
            \draw[-latex] (B) to[out=310,in=240,loop] node[] {} ();
            \draw[-latex] (A) ++(0,0.53) -- (A);
            \draw[-latex] (B) ++(0,0.53) -- (B);
            \node[circle,draw,scale=0.5] (C) at (2,0) {};
            \node[circle,draw,scale=0.5] (D) at (3,0) {};
            \draw[-latex] (C) to[out=45,in=135] (D);
            \draw[-latex] (D) to[out=-135,in=-45] (C);
        \end{tikzpicture} & 
          \begin{tikzpicture}
            \node[circle,draw,scale=0.5] (A) at (0,0) {};
            \node[circle,draw,scale=0.5] (B) at (1,0) {};
            \draw[-latex] (A) to[out=310,in=240,loop] node[] {} ();
            \draw[-latex] (B) to[out=310,in=240,loop] node[] {} ();
            \draw[-latex] (C) ++(0,0.53) -- (C);
            \draw[-latex] (D) ++(0,0.53) -- (D);
            \node[circle,draw,fill=gray!100,scale=0.5] (C) at (2,0) {};
            \node[circle,draw,fill=gray!40,scale=0.5] (D) at (3,0) {};
            \draw[-latex] (C) to[out=45,in=135] (D);
            \draw[-latex] (D) to[out=-135,in=-45] (C);
        \end{tikzpicture} &
         \begin{tikzpicture}
            \node[circle,draw,fill=gray!100,scale=0.5] (A) at (0,0) {};
            \node[circle,draw,scale=0.5] (B) at (1,0) {};
            \draw[-latex] (A) to[out=45,in=135] (B);
            \draw[-latex] (B) to[out=-135,in=-45] (A);
            \draw[-latex] (A) ++(0,0.53) -- (A);
            \node[circle,draw,fill=gray!40,scale=0.5] (C) at (2,0) {};
            \node[circle,draw,scale=0.5] (D) at (3,0) {};
            \draw[-latex] (C) to[out=45,in=135] (D);
            \draw[-latex] (D) to[out=-135,in=-45] (C);
            \draw[-latex] (C) ++(0,0.53) -- (C);
             \phantom{\draw[->] (A) to[out=310,in=240,loop] node[] {} ();}
        \end{tikzpicture}\\
        \hline
    \end{tabular}
    \caption{Vertices of $j$-invariant $1728$ in the Hessian graph with twists over $\F_q$.
    The darker vertex represents the curve of equation $y^2=x^3-x$, while the gray vertex identifies its proper quadratic twist.}
    \label{tab:1728}
\end{table}

\end{proposition}
\begin{proof}
    We first consider curves with $j$-invariant $0$, i.e., those defined for $B \in \F_q^*$ by 
    \[ E_{0}(B) : y^2 = x^3+B. \]

    We directly verify with \cref{eqn:hessianSWForm} that $\Hess^{-1}\big(E_{0}(1)\big)$ always contains an elliptic curve $\F_q$-isomorphic to that defined by $y^2 = x^3-\frac{1}{9}x + \frac{1}{81}$.
    Moreover, the map
    \[ \F_q^*/(\F_q^*)^2 \to \F_q^*/(\F_q^*)^6, \qquad D \mapsto D^{-3} \]
    is always injective, hence the Hessian of the proper quadratic twist of $y^2 = x^3-\frac{1}{9}x + \frac{1}{81}$ is $\F_q$-isomorphic to the proper quadratic twist of $E_{0}(1)$.
    When $q \equiv 2 \bmod 3$, we have $\F_q^* = (\F_q^*)^3$ and $-3 \not\in (\F_q^*)^2$, hence the $\F_q$-isomorphism classes of curves with $j$-invariant $0$ are only $\{E_{0}(1),E_{0}(-3)\}$, while for $q \equiv 1 \bmod 3$ there are six $\F_q$-isomorphism classes of curves with this $j$-invariant.
    Finally, the Hessian of any such curves consists of three independent lines by \Cref{prop:HessofE}-\ref{eni}.
    
    We now consider curves with $j$-invariant $1728$, i.e., those defined for $A \in \F_q^*$ by 
    \[ E_{1728}(A) : y^2 = x^3+Ax. \]
    Eq.\,\eqref{eqn:hessianSWForm} yields the $\F_q$-isomorphisms $\Hess\big(E_{1728}(A)\big) \simeq E_{1728}(-\frac{1}{3A})$ and $\Hess^{(2)}\big(E_{1728}(A)\big) \simeq E_{1728}(A)$, therefore
    \begin{equation} \label{eq:loopequiv}
        \Hess\big(E_{1728}(A)\big) \simeq E_{1728}(A) \quad \iff \quad -3A^2 \in (\F_q^*)^4.
    \end{equation}
    Furthermore, we straightforwardly verify with \cref{eqn:hessianSWForm} that the Hessian of the elliptic curve defined by $y^2 = x^3-\frac{1}{6}x + \frac{1}{36}$ is  $\F_q$-isomorphic to $E_{1728}(-1)$, hence we always have $\Hess^{-1}\big(E_{1728}(-1)\big) \neq \emptyset$.
    \begin{itemize}
        \item $q \equiv 3 \bmod 4$: In this case we have $-1 \notin (\F_q^*)^2$, therefore $(\F_q^*)^2 = (\F_q^*)^4$.
        In particular, by \Cref{lemma:twists}, the  $\F_q$-isomorphism classes of curves with $j$-invariant $1728$ are $\{E_{1728}(1), E_{1728}(-1)\}$.
        Moreover, the map 
        \[ \F_q^*/(\F_q^*)^2 \to \F_q^*/(\F_q^*)^4, \qquad D \mapsto D^{-2}, \]
        is $2$-to-$1$, hence by \Cref{prop:hessianIsom} the two curves of $j$-invariant $-13824$ (\Cref{lem:multRoots}) are mapped to the same element of $j$-invariant $1728$, namely $E_{1728}(-1)$.
        Finally, by \cref{eq:loopequiv}, we conclude $\Hess\big(E_{1728}(1)\big) \simeq E_{1728}(1)$ if and only if $-3 \in (\F_q^*)^2$ (i.e., $q \equiv 1 \bmod 3$), otherwise $\Hess\big(E_{1728}(1)\big) \simeq E_{1728}(-1)$.
        
        \item $q \equiv 1 \bmod 4$: In this case we have four curves of $j$-invariant $1728$, and the map
        \[ \F_q^*/(\F_q^*)^2 \to \F_q^*/(\F_q^*)^4, \qquad D \mapsto D^{-2}, \]
        is injective.
        Thus, by \Cref{prop:hessianIsom}, the curve defined by $y^2 = x^3-\frac{1}{6}x + \frac{1}{36}$ is the unique element in $\Hess^{-1}\big(E_{1728}(-1)\big)$, while its proper quadratic twist belongs to the fiber of the proper quadratic twist of $E_{1728}(-1)$.
        We now consider two cases separately.
        \begin{itemize}
            \item If $-3 \notin (\F_q^*)^2$ (i.e., $q \equiv 2 \bmod 3$), then $-3A^2 \notin (\F_q^*)^4$ for any $A \in \F_q^*$, hence by \cref{eq:loopequiv} all curves with $j$-invariant $1728$ belong to cycles of length $2$.
            Moreover, in this case, the quadratic twist of $E_{1728}(-1)$ is $E_{1728}(3)$, which is not $\F_q$-isomorphic to $\Hess\big(E_{1728}(-1)\big) \simeq E_{1728}(\frac{1}{3})$, because $9 \in \F_q^*$ is not a fourth-power.
            
            \item If $-3 \in (\F_q^*)^2$ (i.e., $q \equiv 1 \bmod 3$), then by \cref{eq:loopequiv} we have $\Hess\big(E_{1728}(A)\big) \simeq E_{1728}(A)$ precisely when $A^2 = -3k^4$, for some $k\in \F_q^*$.
            In particular, the four elements of $j$-invariant $1728$ form one cycle of length $2$ and two loops.
            Let $D$ be a generator of $\F_q^*/(\F_q^*)^4$.
            Two cases may occur:
            \begin{itemize}
                \item if $-3 \in (\F_q^*)^4$, then both $E_{1728}(-1)$ and its quadratic twist $E_{1728}(-D^2)$ satisfy \cref{eq:loopequiv}, therefore they belong to loops.
                \item if $-3 \notin (\F_q^*)^4$, then neither $E_{1728}(-1)$ nor $E_{1728}(-D^2)$ satisfy \cref{eq:loopequiv}, therefore they form a cycle of length $2$.
            \end{itemize}
        \end{itemize}
    \end{itemize}
    Collecting the above properties, the structure of Hessian graphs with twists is locally described by \Cref{tab:0} around $j=0$, and by \Cref{tab:1728} around $j=1728$.
\end{proof}

Employing the above properties of Hessian graphs with twists, we can now prove \Cref{prop:supersingTrace}.

\proof[Proof of \Cref{prop:supersingTrace}] Let $E$ be a supersingular curve such that $j(E)$ belongs to the connected component of the considered $j$ in $\fungraph{\P^1(\F_p)}{\Hess}$ (resp. $\fungraph{\P^1(\F_{p^2})}{\Hess}$).
If $j(E) \notin \{0,1728\}$, then either $E$ or its proper quadratic twist lies in the same connected component of $E_{(j)}$ in the Hessian graph with twists.
By~\cite[§3]{adjEtAl:isogenyGraphsSupersing}, we have
\begin{equation} \label{eq:tr02p}
    \tr(E) = \begin{cases} 
        0 &\textnormal{over } \F_p, \\
        \pm 2p &\textnormal{over } \F_{p^2}.
    \end{cases}
\end{equation}
By \Cref{prop:traceMod3}, we conclude that $\tr(E_{(j)}) \equiv \tr(E) \equiv 0 \textnormal{ (resp. $\not\equiv 0$)} \bmod 3$, as desired.

We now want to prove that \cref{eq:tr02p} holds even when $j(E) \in \{0,1728\}$.
Since $j \notin \{0,1728\}$, up to quadratic twists, we can assume that $\Hess^{(m)}(E_{(j)}) = E$ for some $m \in \NN_{>0}$.
Supersingular curves over $\F_p$ always have trace $0$ \cite[Cor. 4.32]{washington:ellipticCurves}, hence we only need to examine these special $j(E)$ over $\F_{p^2}$.
Up to quadratic twists, by \Cref{prop:specialTwists}, we can assume $E$ to be defined by
\[ \begin{cases}
    y^2 = x^3+1 &\textnormal{if } j(E) = 0, \\
    y^2 = x^3-x &\textnormal{if } j(E) = 1728.
\end{cases} \]
Since $E$ is supersingular and defined over $\F_p$, its Frobenius $\pi$ over $\F_p$ yields $\pi^2 = -p$ \cite[Prop.\,4.32]{washington:ellipticCurves}.
Thus, its Frobenius $\pi^2$ over $\F_{p^2}$ satisfies
\[ 0 = (\pi^{2})^2 - \tr(E)\, \pi^2 + p^2 = 2p^2 + \tr(E)\, p, \]
hence $\tr(E) = -2p$, which (up to quadratic twists) agrees with \cref{eq:tr02p}.
\endproof

\section{Tightness of \texorpdfstring{\Cref{thm:structure}}{Theorem 4.10}}
\label{App:tightnessOfStructureThm}

In this section, we gather examples of groups $G$ and endomorphisms $\phi\in\End(G)$ with finite prime kernels, whose functional graphs realize every configuration listed by \Cref{thm:structure}, regardless of the finiteness of $m$.
Since $\0$ always constitutes a loop in $\fungraph{G}{\phi}$, then $\tau_{\0}$ is always described by case~\ref{per}. 

We first consider the (cyclic) multiplicative group $K^*$ of a field $K$, whose identity element is $\0 = 1$.
For a given prime $\ell \in \NN_{>0}$, we consider the group endomorphism
\[ \phi_\ell : K^* \to K^*, \quad  x \mapsto x^\ell. \]
\begin{itemize}
    \item If $K =\F_q$ is finite, every connected component is finite, hence it falls under case~\ref{per}. Note that $\left|\ker \phi_\ell \right| = \ell$ if and only if $\ell \mid |\F_q^*|$.
    Indeed, if we denote by $u$ a multiplicative generator of $\F_q^*$, then $\ker \phi_\ell$ is generated by $u^\frac{q-1}{\ell}$.
    
    \item If $K=\mathbb{R}$ and $\ell=2$, then $\tau_{\0} = \{1,-1\}$.
    Since every positive real number has precisely two square roots,
    all connected components other than $\tau_{\0}$ fall under case~\ref{line}.
    
    \item Let $K$ be a number field.
    If $K$ contains an $\ell$-th primitive root 
    of $1$, then $\left|\ker \phi_{\ell} \right| = \ell$.
    Since $\textnormal{Norm}_{K|\mathbb{Q}}$ commutes with $\phi_\ell$, every element of $K^*$ has finitely many ancestors in $\fungraph{K^*}{\phi_{\ell}}$, so case~\ref{line} never occurs.
    Finally, since a number field contains only finitely many roots of unity,
    then the number of periodic elements is finite.
    
    Notice that the finiteness of preperiodic elements also follows from the \emph{Northcott property} of number fields (see e.g.\ \cite[Thm.\,5]{Northcott} or \cite[Thm.\,3.12]{silverman2007arithmetic}, as in \Cref{rem:northcott}).
     
   
    \item If $K=\C$ (or, more generally, any algebraically closed field of characteristic different from $\ell$), $\phi_{\ell}$ is everywhere $\ell$-to-1, hence $\tau_{\0}$ is infinite and every other non-periodic connected component falls under case~\ref{line}.
                
    \item If $K=\mathbb{Q}^\mathrm{ab}$, i.e., the maximal abelian extension of $\mathbb{Q}$, then $\left|\ker \phi_{\ell}\right| = \ell$ and $\tau_{\0}$ is infinite by the Kronecker-Weber theorem.
    We now argue that case~\ref{line} cannot occur:
    let us denote the \emph{Weil's height} of $x \in \mathbb{Q}^\mathrm{ab}$ by $h(x)$.
    By \cite[Prop.\,1.2-(iii)]{Zannier20181}, for any $m \in \NN_{>0}$, we have
    \[ h\big( \phi_{\ell}^{(m)}(x) \big) = \ell^m h(x). \]
    However, abelian extensions of $\mathbb{Q}$ enjoy the \emph{Bogomolov property}, i.e., for every non-zero $x \in \mathbb{Q}^\mathrm{ab}$ that is not a root of unity, we have a uniform bound $h(x) \geq \epsilon > 0$ \cite[Thm. 1]{amorosoDvornicich:height}.
    Thus, for any given $y = \phi_{\ell}^{(m)}(x)$, we have 
    \[ \frac{h(y)}{\ell^m} = h(x) \geq \frac{\log 5}{12} \quad \implies \quad m \leq \frac{\log\big( \frac{12\, h(y)}{\log 5} \big)}{\log \ell}. \]
    Therefore, all the non-periodic connected components fall within case~\ref{sline}.   
\end{itemize}

Next, to cover the remaining possible combinations, we consider the following groups:
\begin{itemize}
    \item Let $K$ be a number field containing an $\ell$-th primitive root of $1$.
    We consider $G = \mathbb{R}^* \times K^*$ (with entry-wise multiplication), and
    \[ \phi: G \to G, \quad (x,y) \mapsto (x^3, y^\ell). \]
    It is easy to see that $\left|\ker \phi \right| = \ell$.
    Since $(\alpha,1) \in G$ has infinitely many ancestors for any $\alpha \neq \pm 1$, case~\ref{line} appears in $\fungraph{G}{\phi}$.
    However, points $(\alpha,\beta) \in G$ with $|\beta| \neq 1$ have only finitely many ancestors and cannot become periodic after a finite number of iterations of $\phi$, therefore they belong to connected components as in case~\ref{sline}.
    
    \item As above, but replacing $K$ with $\mathbb{Q}^{\mathrm{ab}}$. The only difference is that $\tau_{\0}$ is infinite.
    
    \item Let $\mu = \{e^{\pi i q}\}_{q \in \mathbb{Q}} \subseteq \C$ be the multiplicative group of roots of unity, and consider its endomorphism $\phi_{\ell} : x \mapsto x^\ell$.
    We have $\left|\ker \phi_{\ell}\right| = \ell$ and $|\tau_{\0}| = \infty$.
    The periodic elements in $\fungraph{\mu}{\phi_{\ell}}$ are of the form $e^{\pi i \frac{n}{m}}$, with $\gcd(m,\ell) = 1$.
    Thus, every element is preperiodic, so the connected components are all of type \ref{per}.
\end{itemize}
The above examples are summarized in Table~\ref{tab:mergedExamples}.

\begin{table}[ht!]
    \centering
    \begin{tabular}{||c|c|c|c|c||}
        \hline
        $G$ & $\phi \in \End(G)$ & $|\tau_{\0}|$  & Case~\ref{line} & Case~\ref{sline} \\
        \hline \hline
        $\F_q^*$ & $x \mapsto x^\ell, \quad \ell \mid (q-1)$ & \textcolor{lightgray}{$< \infty$} & \xmark & \xmark \\
        \hline
        $\mathbb{R}^*$ & $x \mapsto x^2$ & \textcolor{lightgray}{$2$} & \cmark & \xmark \\
        \hline
        $K^*$ & $x \mapsto x^\ell$ & \textcolor{lightgray}{$<\infty$} & \xmark & \cmark \\
        \hline
        $\mathbb{R}^* \times K^*$ & $(x,y) \mapsto (x^3, y^\ell)$ & \textcolor{lightgray}{$<\infty$} & \cmark & \cmark \\
        \hline
        $\mu$ & $x \mapsto x^\ell$ & $\bm{\infty}$ & \xmark & \xmark \\
        \hline
        $\C^*$ & $x \mapsto x^\ell$ & $\bm{\infty}$ & \cmark & \xmark \\
        \hline
        $(\mathbb{Q}^{\mathrm{ab}})^*$ & $x \mapsto x^\ell$ & $\bm{\infty}$ & \xmark & \cmark \\
        \hline
        $\mathbb{R}^* \times (\mathbb{Q}^{\mathrm{ab}})^*$ & $(x,y) \mapsto (x^3, y^\ell)$ & $\bm{\infty}$ & \cmark & \cmark \\
        \hline
    \end{tabular}
    \caption{Occurrence of the cases listed in \Cref{thm:structure}.
    Here $K$ is a number field containing an $\ell$-th primitive root of $1$, and $\mu$ is the group of all roots of unity in $\C$.}
    \label{tab:mergedExamples}
\end{table}

\section{Algorithm for iterated Hessian}
\label{App:iterHess}

The following pseudocode corresponds to the algorithm sketched in \Cref{prop:iterHess}.
A raw SAGE \cite{Sage} implementation is available at \url{https://github.com/DTaufer/Hessian-Graphs/}, and can handle any iterated Hessian over finite fields of size about $256$ bits in less than one second.
\begin{algorithm}[H]
    \caption{$\mathsf{IteratedHessian}$}\label{alg:iterHess}
    \begin{algorithmic}[1]
    \Require{$N \in \NN,\;j\in\F_q$}
    \Ensure{$\Hess^{(N)}(j)$}
    \end{algorithmic}
    \noindent
    \begin{minipage}{0.45\textwidth}
    \begin{algorithmic}[1]
    \State $u \gets $ a generator of  $\F_q^* / (\F_q^*)^6$
    \State $s \gets 0$
    \While{$j u^{2s}$ is not a cube}
        \State $s \gets s+1$
    \EndWhile
    \State $t\gets 0$
    \If{$u^{2s}(j-1728)$ is not a square}
        \State$t \gets 1$
    \EndIf
    \State $u \gets u^{2s+3t}$
    \State $x \gets $ cubic root of $j u$
     \State $y \gets $ square root of $(j-1728) u$
    \State $E \gets $ elliptic curve $y^2= x^3- 1728 u$
    \State $N' \gets-3^{\lfloor N/2 \rfloor} \bmod {|E(\F_{q})|}$
    \State $[x\colon y\colon z] \gets [N'][x\colon y \colon 1]$
    \end{algorithmic}
    \end{minipage}\hfill
    \begin{minipage}{0.45\textwidth}
    \begin{algorithmic}[1]
    \setcounter{ALG@line}{15} 
    \If{$z=0$}
        \State \textbf{return} $\infty$
    \Else 
        \State $j \gets{x}^3  u^{-1}$
        \If{$N$ is odd}
            \If{$j=0$}
                \State \textbf{return} $\infty$
            \Else
                \State  \textbf{return} $(6912-j)^3/(27j^2 )$
            \EndIf
        \Else
            \State \textbf{return} $j$
        \EndIf
    \EndIf
    \end{algorithmic}
    \end{minipage}
\end{algorithm}

\section{Some examples}
\label{App:examples}

In this section, we provide visual instances of the results presented in this paper.
In \Cref{fig:sidebyside} we observe that the Lattès structure induces many symmetries in Hessian graphs, which are not expected for functional graphs of generic rational functions, as discussed in \Cref{rmk:constantmatters}.

\begin{figure}[htbp]
    \centering
    \begin{subfigure}[b]{0.55\textwidth}
        \centering
        \includegraphics[width=0.99\textwidth]{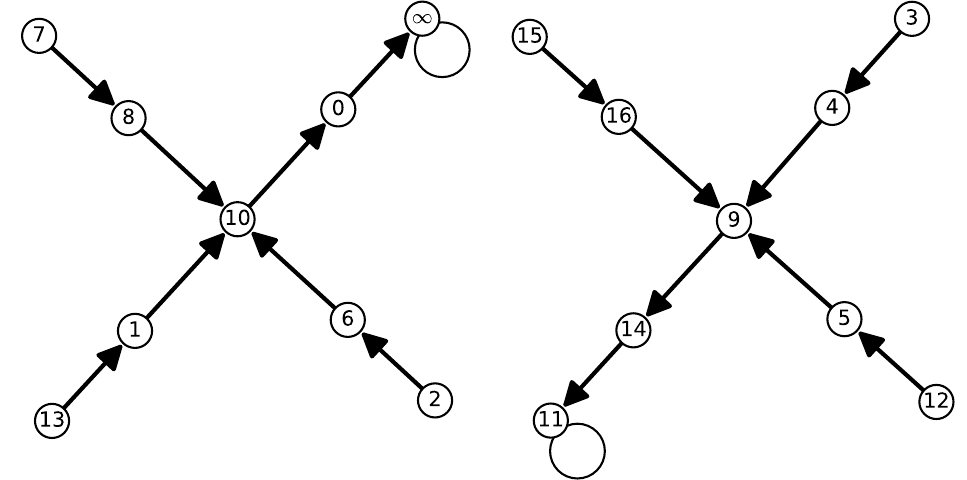}
        \caption{Hessian graph $\fungraph{\P^1(\F_{17})}{\Hess_{-6912}}$.}
        \label{fig:hessian17}
    \end{subfigure}
    \hfill
    \begin{subfigure}[b]{0.44\textwidth}
        \centering
        \includegraphics[width=0.63\textwidth]{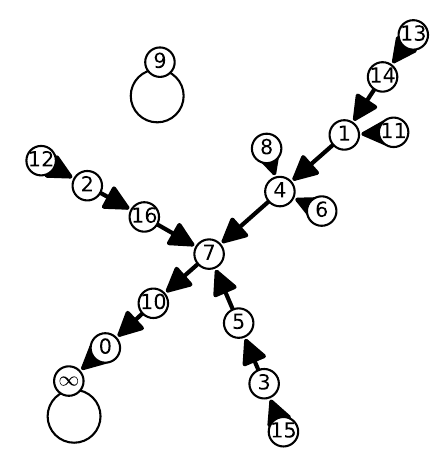}
        \caption{Functional graph
$\fungraph{\P^1(\F_{17})}{\Phi_{\frac{27}{8}} \circ \Hess_{-6912}}$.}
        \label{fig:randomFuncGraph}
    \end{subfigure}
    \caption{Comparison of two functional graphs.}
    \label{fig:sidebyside}
\end{figure}

In what follows, we provide some examples of the objects discussed in \Cref{sec:HessianFF}.
To ease the comparison of $\fungraph{\S_{\Hess}}{\psi}$ (Figures~\ref{fig:hessian29psi} and~\ref{fig:hessian31psi}) with their underlying Hessian graphs $\fungraph{\P^1(\F_q)}{\Hess}$ (Figures~\ref{fig:hessian29} and~\ref{fig:hessian31}), the vertices of $\fungraph{\S_{\Hess}}{\psi}$ are labeled with their image after $\pi_{\Hess}$.
The following color coding is used to highlight the coverings of \Cref{cor:Spi4Hessian}: the white vertices correspond to points arising from $E(\F_{q})$, while the gray ones arise from its quadratic twists.
The light-colored points arise from the cubic twists of $E(\F_{q})$, while the dark-colored ones are the respective quadratic twists.
We remark that leaves always have light colors, as prescribed by \Cref{thm:indeg1}.
The SAGE \cite{Sage} code employed to draw $\fungraph{\P^1(\F_q)}{\Hess}$ is available at \url{https://github.com/DTaufer/Hessian-Graphs/}.

\vspace{0.05cm}
\begin{figure}[ht]
    \centering
    \includegraphics[width=0.68\textwidth]{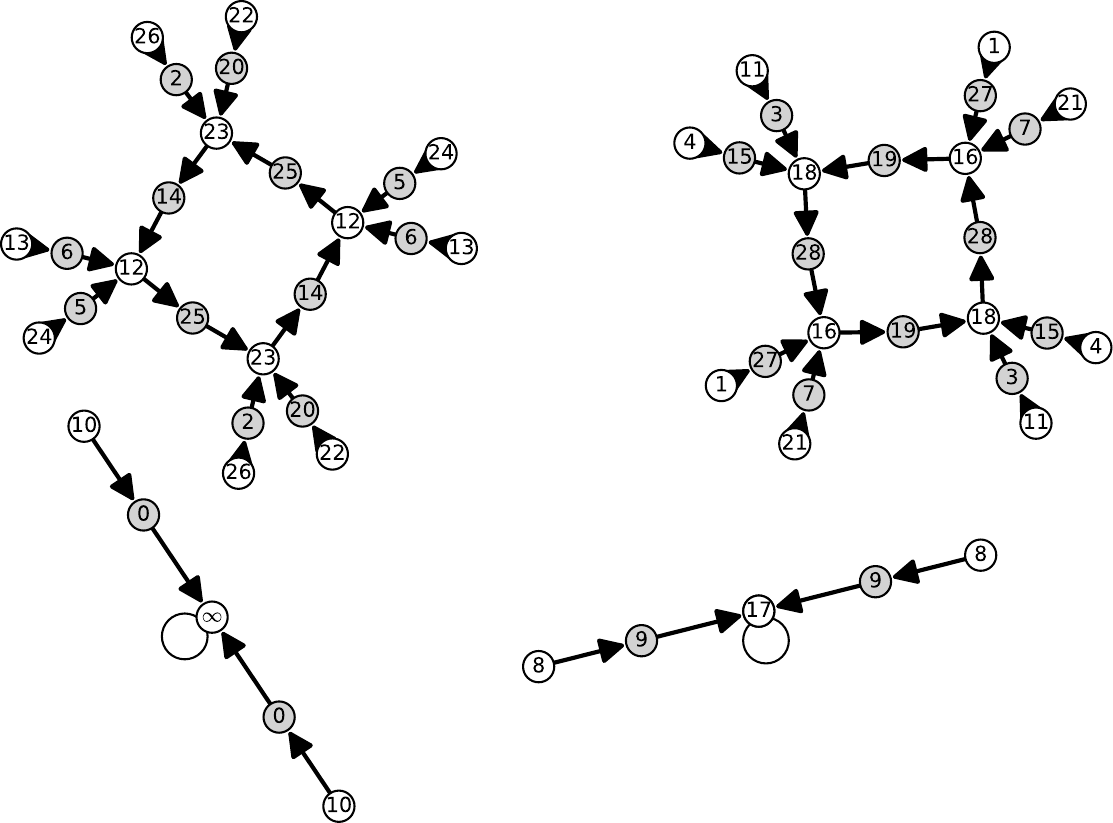}
    \caption{Partial functional graph of $\fungraph{\S_{\Hess}}{\psi}$ over $\F_{29}$.
    The complete $\fungraph{\S_{\Hess}}{\psi}$ can be obtained by repeating the above components three times (the remaining components have the same image after $\pi_{\Hess}$).}
    \label{fig:hessian29psi}
\end{figure}

\begin{figure}
    \centering
    \includegraphics[width=0.68\textwidth]{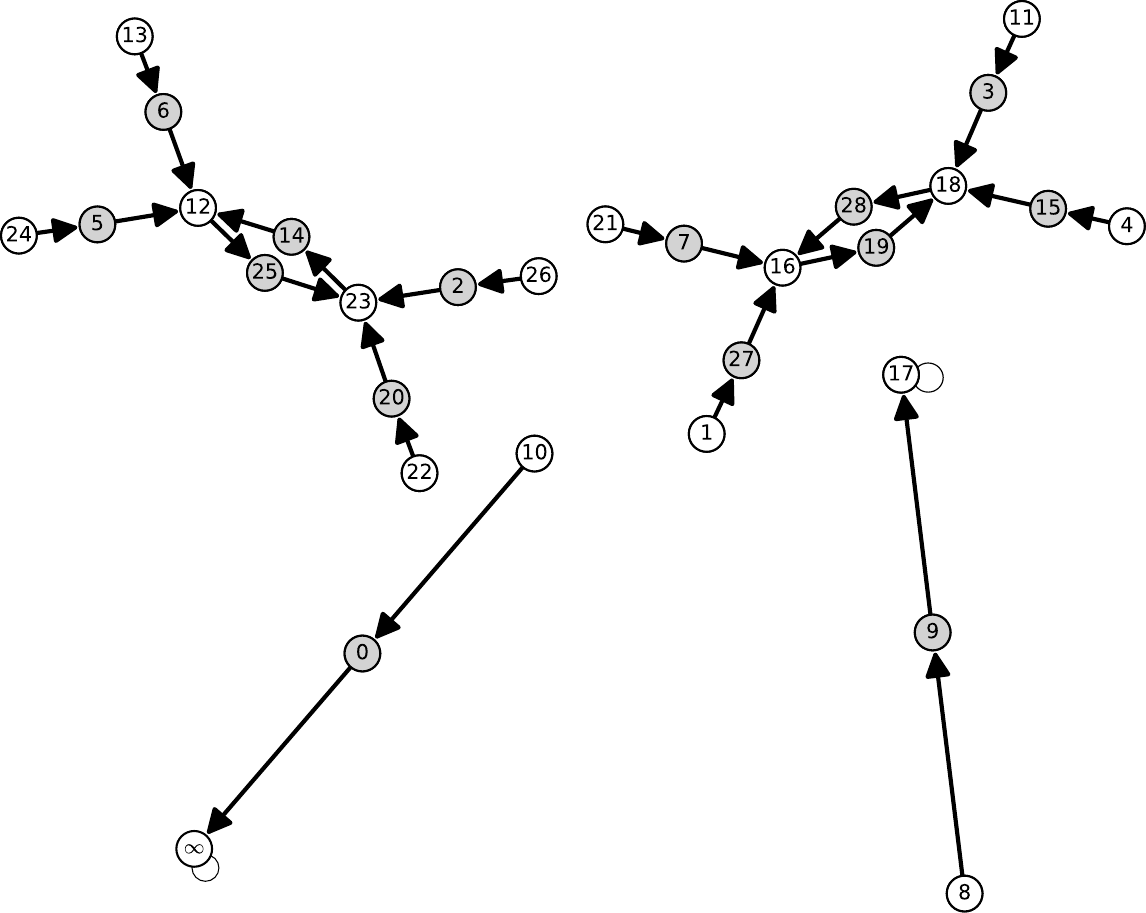}
    \caption{The Hessian graph $\fungraph{\P^1(\F_{29})}{\Hess}$.}
    \label{fig:hessian29}
\end{figure}


\begin{figure}
    \centering
\includegraphics[width=0.86\textwidth]{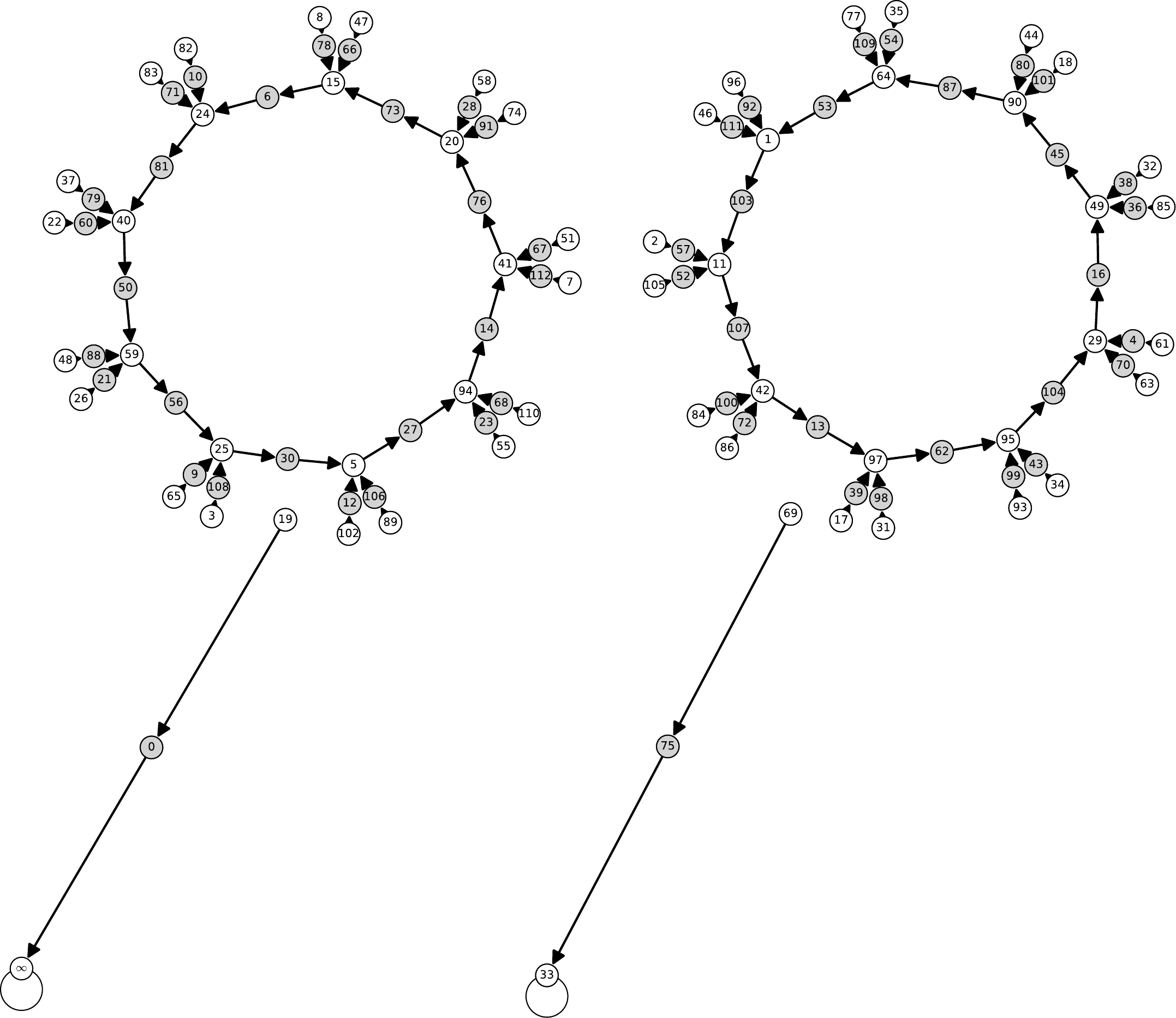}
    \caption{The Hessian graph $\fungraph{\P^1(\F_{113})}{\Hess}$.}
    \label{fig:hessian89}
\end{figure}

\begin{figure}[ht]
    \centering
    \includegraphics[width=0.85\textwidth]{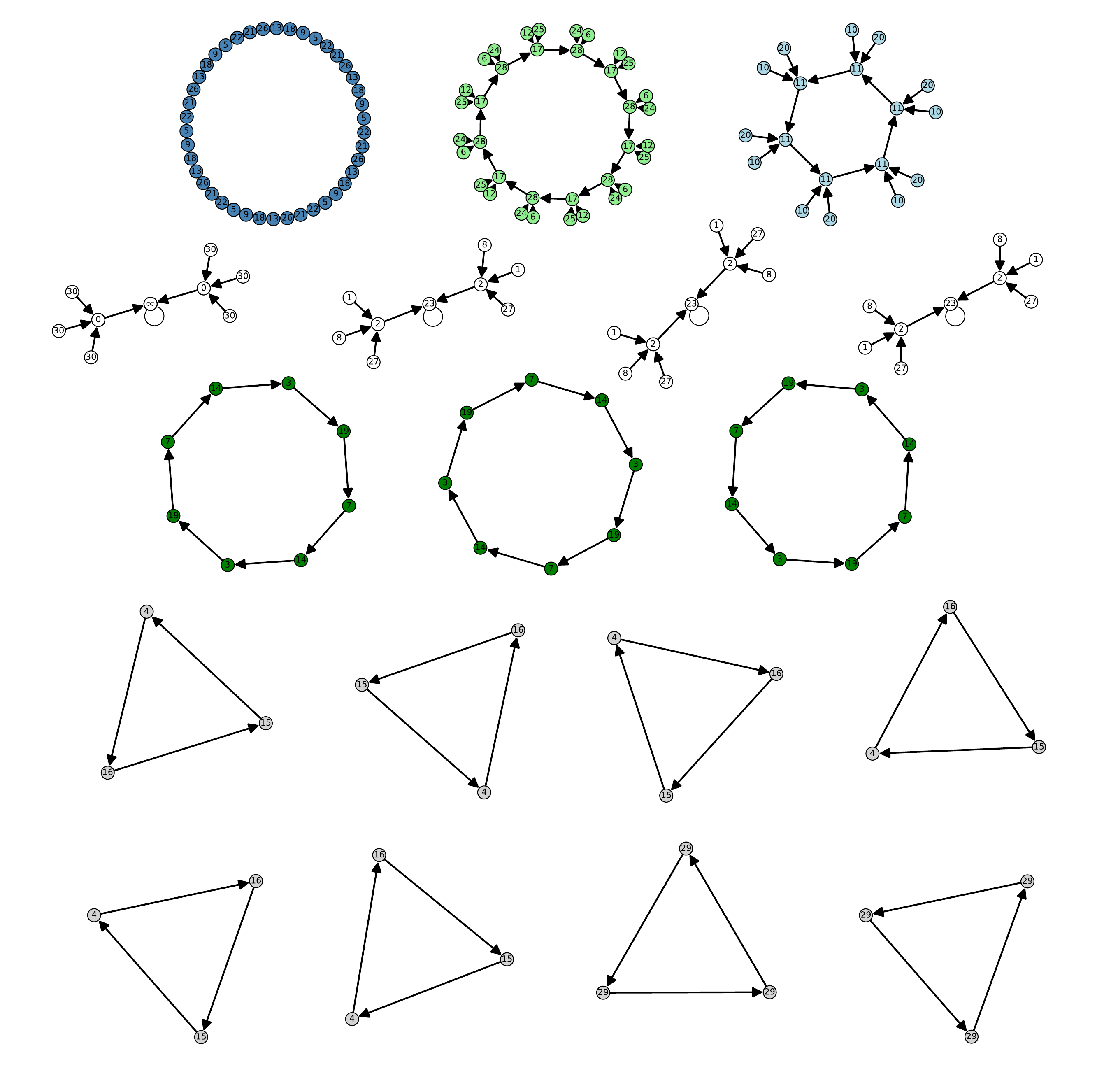}
    \caption{The functional graph $\fungraph{\S_{\Hess}}{\psi}$ over $\F_{31}$.}
    \label{fig:hessian31psi}
\end{figure}

\begin{figure}
    \centering
    \includegraphics[width=0.55\textwidth]{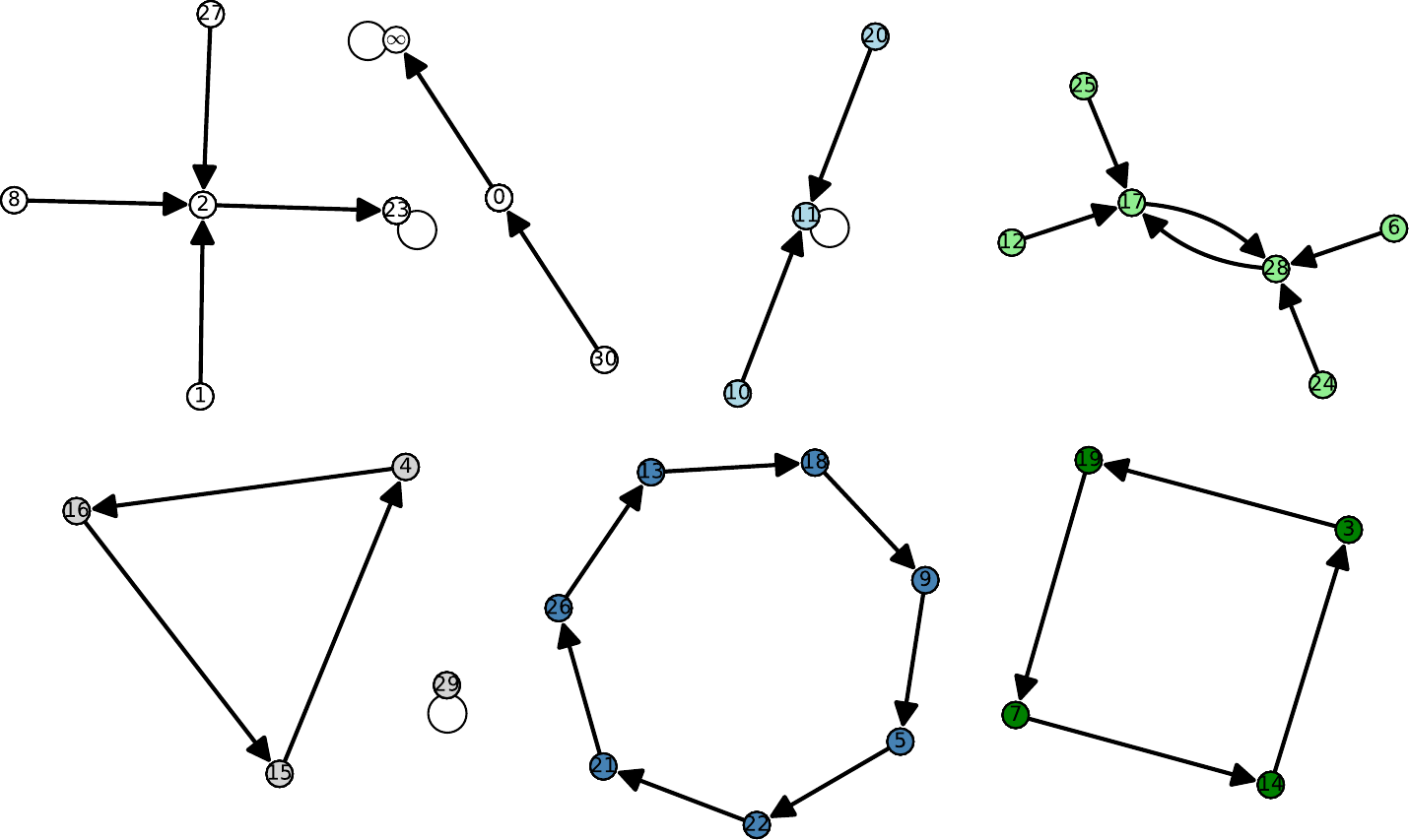}
    \caption{The Hessian graph $\fungraph{\P^1(\F_{31})}{\Hess}$.}
    \label{fig:hessian31}
\end{figure}

\begin{figure}
    \centering
    \includegraphics[width=1\linewidth]{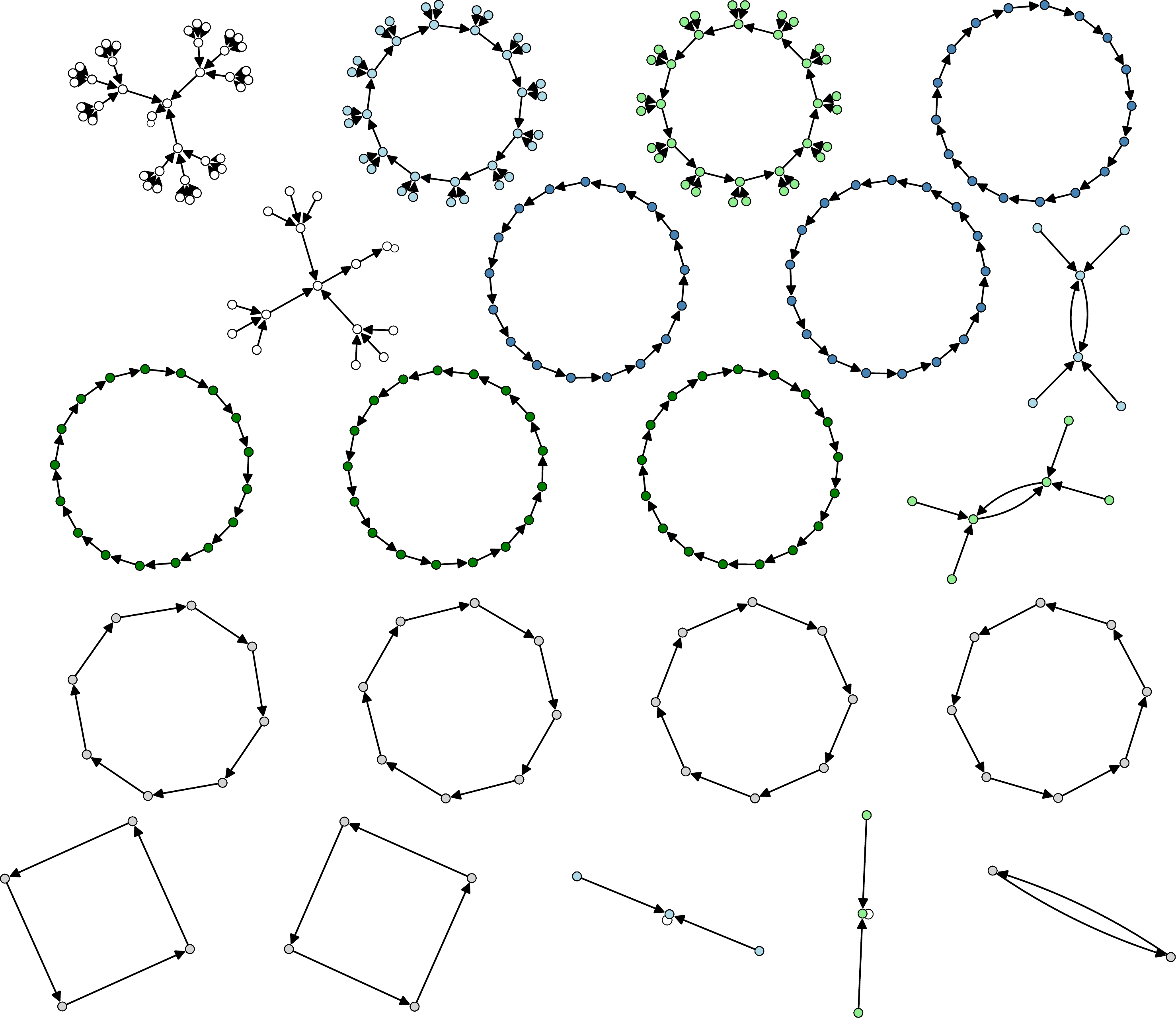}
    \caption{The Hessian graph $\fungraph{\P^1(\F_{17^2})}{\Hess}$.}
    \label{fig:hessian172}
\end{figure}
\end{document}